\documentclass[preprint,twoside,11pt]{article}

\usepackage{blindtext}
\usepackage{jmlr2e}
\usepackage{titletoc}   % nothing else is needed

\usepackage{amsmath,amssymb}
\allowdisplaybreaks
\usepackage{graphics} % for pdf, bitmapped graphics files
\usepackage{epsfig,color} % for postscript graphics files
\usepackage{mathtools}
\usepackage{thmtools}
\usepackage{lipsum,enumitem}
\usepackage{times} % assumes new font selection scheme installed

\usepackage{algorithm}
\usepackage{indentfirst}
\DeclareMathAlphabet{\mathcal}{OMS}{cmsy}{m}{n}
\SetMathAlphabet{\mathcal}{bold}{OMS}{cmsy}{b}{n}
\usepackage{algorithmic}
\usepackage{graphicx}
\usepackage{booktabs}
\usepackage{siunitx}       % nice column alignment for numbers (optional)
\usepackage{subfigure}
\usepackage{textcomp}
\usepackage{xcolor}
\usepackage{thm-restate}
\newcommand{\red}[1]{\textcolor{black}{#1}}

\newtheorem{assumption}{Assumption}
\usepackage{hyperref}

\hypersetup{
    colorlinks=true,
    citecolor=blue,
    linkcolor=red,
    filecolor=magenta,      
    urlcolor=cyan,
    pdftitle={Overleaf Example},
    pdfpagemode=FullScreen,
    }
    
\usepackage{lastpage}

\jmlrheading{}{}{}{}{}{21-0000}{Preprint Under Review}

\ShortHeadings{Nest-SGD for Sinkhorn-Regularized DRO}{Yang, Zhou, Lu}
\firstpageno{1}

\begin{document}

\title{Nested Stochastic Algorithm for Generalized Sinkhorn distance-Regularized Distributionally Robust Optimization}

\author{\name Yufeng Yang \email yufeng.yang@tamu.edu \\
       \addr Department of Computer Science and Engineering\\
       Texas A\&M University\\
       College Station, TX 77843, USA
       \AND
       \name Yi Zhou \email yi.zhou@tamu.edu \\
       \addr Department of Computer Science and Engineering\\
       Texas A\&M University\\
       College Station, TX 77843, USA
       \AND
       \name Zhaosong Lu\email zhaosong@umn.edu \\
       \addr Department of Industrial and Systems Engineering\\
       University of Minnesota \\
       Minneapolis, MN 55455, USA}

\editor{}

\maketitle

\begin{abstract}
Distributionally robust optimization (DRO) is a powerful technique to train robust models against data distribution shift. This paper aims to solve regularized nonconvex DRO problems, where the uncertainty set is modeled by a so-called generalized Sinkhorn distance and the loss function is nonconvex and possibly unbounded. Such a distance allows to model uncertainty of distributions with different probability supports and divergence functions. For this class of regularized DRO problems, we derive a novel dual formulation taking the form of nested stochastic optimization, where the dual variable depends on the data sample. 
To solve the dual problem, we provide theoretical evidence to design a nested stochastic gradient descent (SGD) algorithm, which leverages stochastic approximation to estimate the nested stochastic gradients. We study the convergence rate of nested SGD and establish polynomial iteration and sample complexities that are independent of the data size and parameter dimension, indicating its potential for solving large-scale DRO problems. We conduct numerical experiments to demonstrate the efficiency and robustness of the proposed algorithm.
\end{abstract}
\begin{keywords}
  Distributionally Robust Optimization; Lagrange
Duality; Nested Stochastic Optimization; First-order Algorithm; Adversarial Robustness.
\end{keywords}

\section{Introduction}
In classic machine learning, the primary goal is to achieve good predictive performance in the test set after training the model on a designated training set. The training problem is typically formulated as an empirical risk minimization (ERM) problem \citep{vapnikERM}. 
%In particular, stochastic gradient descent (SGD) \citep{ghadimi2013stochastic} and its variants \citep{ghadimi2016accelerated,adam,rmsprop,SNAG,Natasha,JMLR:v12:duchi11a,sarah, acceleratingsgd,SHB} have been widely applied to solve this type of problems. 
However, empirical risk minimization assumes that the training set and the test set follow the same underlying data distribution, which is often unrealistic and may result in poor test performance when data distribution shift exists. 

Data distribution shift is prevalent in real-world scenarios. It can be caused by many factors such as sampling bias, presence of anomalies, data merging and change of measurements, etc. To tackle this challenge, distributionally robust optimization (DRO) \citep{scarf1958} was proposed, which formulates the objective function as a min-max problem. DRO aims to learn a robust model by minimizing the expected risk over the worst-case data distribution within a predefined ambiguity set. This formulation offers a principled framework to learn the optimal resilient solution in the face of distribution uncertainty.

One key factor in DRO is the selection of an appropriate divergence measure for modeling the ambiguity set. Specifically, the divergence measure should not only be computationally tractable but also yield a solution that avoids excessive conservatism. In the existing literature, various divergence-based ambiguity sets have been studied. In \citet{wasserportfolio,wassersteinframework,NIPS2015_cc1aa436,esfahani2017datadriven,gaosa,LUO201920, joseOTDRO}, the authors focused on reformulating the expressions of objective functions under the worst-case distributions into a tractable form and exploring possible algorithms to tackle DRO problems under Wasserstein-type ambiguity sets. For more information, we refer the readers to \citet{surveyWDRO} for a comprehensive survey on Wasserstein DRO. In \citet{Hu2012KullbackLeiblerDC,bayraksan2015data,levy2020largescale,duchi2020learning,smoothedfDRO}, the authors analyzed alternative expressions of objective functions under the worst-case distribution, developed algorithms to solve DRO problems under Kullback–Leibler (KL) and $f$-divergence based ambiguity sets. However, the aforementioned divergence measures have certain limitations. For example, it is known that DRO with Wasserstein distance requires high computational complexity \citep{emdcomplexity,ba2009sublinear}. Also, both KL and $f$-divergence are not symmetric when assessing distributions. Furthermore, these two divergence measures require that the distributions share the same probability support, a strong condition that may fail to capture extreme distributions at certain points. We refer readers to check Examples \ref{ex: 1} and \ref{ex: 2} for some concrete applications.

The Sinkhorn distance, first introduced in \citet{cuturi2013Sinkhorn}, was designed to address the aforementioned limitations. Sinkhorn distance is symmetric and allows distributions from the same sample space to have different probability supports. Furthermore, Sinkhorn distance is a convex function with respect to distributions, ensuring computation tractability and efficiency for large-scale problems. In \citet{wang2023sinkhorn,joseOTDRO,SinkhornDROesaim}, constrained Sinkhorn DRO was initially investigated. Specifically, \citet{wang2023sinkhorn} derived the equivalent formulation of constrained Sinkhorn DRO problem and solved it via stochastic mirror descent algorithm, marking the first work to solve constrained Sinkhorn DRO using first-order optimization methods. However, the convergence analysis conducted in their work assumed that the loss function is convex and bounded, which may not hold in practical modern machine learning applications. Furthermore, the log-exponential compositional structure induced by the conjugate dual of the KL-divergence makes the objective function difficult to optimize and hinders convergence.

In this work, we develop and study dual formulations of \textit{generalized} Sinkhorn-distance regularized DRO problems (see formulation \eqref{primal1}). The problem takes the form of a nested stochastic optimization with a contextual variable. Unlike traditional stochastic optimization problems, our formulation makes dual decision variable dependent on data sample following distribution $\mathbb{P}$. \red{ \citet{hu2023contextual} studied contextual bilevel optimization, which can be applied to solve such problem. However, their convergence analysis is not directly applicable to proposed Sinkhorn DRO dual formulation \eqref{dual2}. Specifically, their convergence guarantee relies on the strong convexity of the inner objective, a condition that does not hold in our setting. Moreover, from algorithmic perspective, while Multi-level Monte Carlo (MLMC) significantly improves the per-iteration sample complexity for computing a variance-reduced stochastic gradient estimator, it still requires storing a non-constant number of samples that scale with the target error, and second-order information remains unavoidable under bilevel optimization framework, leading to increased computational cost and unnecessary overall sample complexity for our problem. To tackle aforementioned challenges, based on tractable expression of gradient proved in \citet{jin2021nonconvex} (see Lemma \ref{thm: differentiability}) and structural relationship between outer problem's gradient and inner problem's gradient, we establish controllable approximation error of outer objective's gradient given the inner objective's gradient satisfying mild conditions (see Theorem \ref{thm: grad_error_bound}).}
This observation enables us to further utilize sample-average approximation (SAA) to estimate stochastic gradients, solve dual problem efficiently via Nested SGD algorithm (see Algorithm \ref{alg2}) and establish convergence analysis without requiring large batch size of samples and additional assumptions including convex loss and strong convexity for inner-objective assumed in \citet{wang2023sinkhorn,hu2023contextual}. Finally, we train several models including logistic regression, LeNet \citep{LeNet} over CIFAR-10 \citep{cifar10} and MNIST \citep{mnist} dataset, and conduct experiments to evaluate test performance of proposed Sinkhorn DRO dual formulation \eqref{dual2} with other baseline methods over perturbed test dataset. Our results demonstrate that the Sinkhorn DRO formulation \eqref{dual2} and our proposed nested algorithm successfully improve models' robustness against distribution shift.
%This generalized notion not only retains the advantages of the original Sinkhorn distance but also allows to use a broader range of divergences to model the ambiguity set (deleted from summary of contributions).
\subsection{Summary of Contributions}
\begin{itemize}[topsep = 1pt, itemsep=1pt]
    \item \red{To preserve the advantages of Sinkhorn distance while broadening the class of divergence measures used to model the ambiguity set}, we consider a generalized Sinkhorn distance based on the family of $f$-divergences. Building on this formulation, we transform the primal problem using inverse CDF sampling and derive the dual form of the regularized Sinkhorn DRO problem via Lagrangian duality, under which strong duality holds. The resulting dual problem takes a novel form of contextual nested stochastic optimization problem, where the variable of the inner stochastic sub-problem depends explicitly on the data samples.
    %introducing the novel concept of an in-context variable in inner problem.
    \item \red{ 
   To solve the contextual nested stochastic optimization problem, we begin by reviewing the explicit gradient formula of the proposed Sinkhorn DRO objective. By exploiting the structural relationship between the objective and its sub-problem, we derive conditions under which the gradient approximation error can be made arbitrarily small. This insight motivates the development of a practical Nested-SGD algorithm (Algorithm~\ref{alg2}), which simplifies the approach taken by the research community to solve contextual nested problems within the general bilevel optimization framework. To establish convergence guarantee, we further analyze the smoothness properties and derive second-moment upper bounds. Our convergence analysis reveals a standard sample complexity of $\mathcal{O}(\varepsilon^{-4})$ for each loop of the algorithm, resulting in a total sample complexity of $\mathcal{O}(\varepsilon^{-8})$ for non-convex loss. However, considering the one-dimensional nature of the sub-problem of the proposed Sinkhorn DRO formulation and by appropriately adjusting the batch size, the overall problem can be efficiently solved with $\mathcal{O}(\varepsilon^{-4})$ iteration complexity.}
    \item To justify the efficiency of our proposed formulation and algorithm, we evaluate their practical performance on real-world datasets. The results, under distribution shifts simulated by various adversarial attack methods, demonstrate that our approach not only improves the robustness of diverse machine learning models but also scales effectively to large models and datasets.

\end{itemize}

\section{Related Work}
% This work contributes to research area related with 1. DRO; 2. Sinkhorn distance; 3. Nested contextual stochastic optimization. 
% We will go through the literatures intersects with three research lines as well as previous work which serves the cornerstone of our work.
% \\
\textsl{\textbf{DRO.}} The DRO framework shares strong connections with contrastive learning \citep{constrastive}, multiple instance learning \citep{MIP}, AUC maximization \citep{aucdroyang}, anomaly detection \citep{anomaly}, and self-supervised learning \citep{zihaoclip,xiyuanclip}. The key challenge in modeling lies in transforming the problem into a tractable formulation given the chosen ambiguity set. The first stream focuses on using information divergence to construct ambiguity set. Commonly employed divergence measures include the Wasserstein metric \citep{NIPS2015_cc1aa436,esfahani2017datadriven,gaosa,LUO201920,geometricwasserstein,residualdro}; KL divergence \citep{Hu2012KullbackLeiblerDC,bayesDRO,conicKL}; $f$-divergence \citep{levy2020largescale,phirecovering,duchi2020learning,jin2021nonconvex,namkoong2016stochastic,smoothedfDRO}; Sinkhorn distance \citep{wang2023sinkhorn, SinkhornDROesaim} and its variants \citep{UOTDRO}. \red{Recently, \citet{joseOTDRO} further proposed a framework to unify aforementioned constrained DRO problem based on optimal transport theory with conditional moment constraints.} Another stream for constructing ambiguity set is using special statistics, such as geometry shape constraints \citep{shapeDRO} and statistical moments \citep{scarf1958,momentdro,moment2dro} etc. 

\textsl{\textbf{Sinkhorn Distance.}} Sinkhorn distance has successful applications in areas like generative models \citep{autoencoders,pmlr-genevay18a}, matrix factorization \citep{nnmatrix}, image segmentation \citep{segmentation} etc. In \citet{cuturi2013Sinkhorn}, Sinkhorn matrix scaling algorithms was proposed to compute optimal transport map under Sinkhorn distance objective. Later, in \citet{altschuler2018nearlinear,aude2016stochastic}, greedy and stochastic variants of Sinkhorn scaling algorithms were proposed to clarify relationship between algorithm convergence and input dimensions. Some works also study Sinkhorn distance computation over data samples with special structures. In \citet{semidiscreteOT,altschuler2019massively}, they propose variant algorithms specially applying to data samples over compact Riemannian manifolds and Euclidean balls respectively. %However, these investigations emphasize on computation of optimal transport plans, which is redundant in DRO. This motivates our study for Sinkhorn DRO without the explicit computation of the optimal transport map.

\textsl{\textbf{Algorithms for Solving DRO.}} For Wasserstein metric ambiguity set, several works reformulate primal problem into tractable forms such as convex programming \citep{NIPS2015_cc1aa436},  
semidefinite programming \citep{LUO201920} and mixed integer programming \citep{MIP}. Subsequent works \citep{NIPS2015_cc1aa436, MIP} transform problems into the form which is directly solvable using software toolbox. In \citet{LUO201920}, they use cutting-surface method to solve semi-definite programming for general nonlinear objective and branch-and-bound algorithms for bilinear objective.
Another common technique to transform DRO is using duality theory \citep{gaosa,levy2020largescale}. Through this way, computation of shifted distribution can be avoided. In \citet{qidro1}, projected SGD and acceleration is used to solve the dual form of KL divergence constrained DRO. In \citet{jin2021nonconvex}, normalized SGD with momentum is used to optimize the dual of $f$-divergence regularized DRO. Later \citet{zhangrevistingDRO} revisits $f$-divergence regularized DRO and propose double SGD
with clipping to solve it. In \citet{qidro2}, stochastic Frank-Wolfe method is used to solve approximation for the dual of general Cressie-Read family divergence constrained DRO. \citet{wang2023sinkhorn} used stochastic mirror descent to solve the dual form of Sinkhorn distance constrained DRO and \citet{UOTDRO} used a projected sub-gradient method to solve the dual form of unbalanced Wasserstein distance constrained DRO. \citet{hu2023contextual} empirically applied contextual bilevel optimization methods to solve the primal formulation of Wasserstein DRO with side information.

\section{Notations}
\red{Throughout this work, we denote $\xi \sim \mathbb{Q}$ as data drawn from the underlying distribution $\mathbb{Q}$, $\zeta \sim \mathbb{P}$ as data drawn from the nominal distribution $\mathbb{P}$, and their corresponding reference measures are denoted by $\mu$ (for $\mathbb{P}$) and $\nu$ (for $\mathbb{Q}$). We denote $\mathcal{N}(\cdot,\cdot)$ as normal distribution following certain mean and variance. For the primal problem \eqref{primal1}, we denote $\ell(x;\xi)$ as the loss function associated with sample $\xi$ and parameter $x \in \mathbf{R}^d$ (i.e., the weights of a linear model or neural network), $\lambda$ as the regularization coefficient. For generalized Sinkhorn distance used to model the ambiguity set (see Definition \ref{gSinkhorn}), we denote $c(\cdot,\cdot)$ as the cost metric for measuring \textit{proximity} between samples $\zeta$ and $\xi$, $\beta$ as the regularization coefficient, $D_f$ as the $f$-divergence. In primal-dual problem reformulation, we introduce $\eta$ as Lagrange multiplier and $f^{*}(\cdot)$ as the conjugate dual function induced by the chosen information divergence.
For the proposed dual formulation~\eqref{dual2}, when $\eta_x^*(\zeta)= \arg\min\mathcal{L}_{\zeta}(x,\eta)$, we denote $\Psi(x)=\mathcal{L}_{\zeta}(x,\eta_x^*(\zeta))$ for simplicity, and use $\nabla \Psi(x)$ to denote the gradient with respect to $x$. For functions including $\mathcal{L}_{\zeta},\mathcal{L}_{\zeta,\xi}:\mathbf{R}^d\times \mathbf{R}\rightarrow \mathbf{R}$, which takes two arguments $x\in\mathbf{R}^d$, $\eta\in \mathbf{R}$, we use $\nabla_1$, $\nabla_2$ to denote the gradient with respect to $x$ and $\eta$.
For Algorithm \ref{alg1}, which optimizes the inner objective \eqref{eq: inner problem}, we denote $\alpha_d$ as the learning rate, $v(\cdot)$ as the stochastic gradient estimator with respect to $\eta$, $\tilde{d}$ as the output index, $\tilde{B}$ as the batch size, and $\tilde{\varepsilon}$ as the scaled target error. For Algorithm \ref{alg2}, which optimizes the dual formulation~\eqref{dual2}, we denote $\gamma_t$ as the learning rate, $\hat{g}^{B}(\cdot),g^{B}(\cdot)$ as inexact, exact stochastic gradient estimator of dual formulation~\eqref{dual2} with respect to $x$, $\tilde{t}$ as the algorithm’s output index, $B$ as the batch size, and $\varepsilon$ as the target error passed to Algorithm \ref{alg2}. Through this work, we denote $\|\cdot\|$, $\|\cdot\|_1$, and $\|\cdot\|_{\infty}$ as $\ell_2$, $\ell_1$, and $\ell_{\infty}$ norms over Euclidean space, respectively.}

\section{Regularized Nonconvex DRO with Generalized Sinkhorn Distance}

In this section, we first introduce a class of \textit{regularized} nonconvex distributionally-robust optimization (DRO) problems, where the data distribution uncertainty is quantified by generalized Sinkhorn distance. \red{We then study its strong dual formulation in Theorem \ref{thm: dual formulation} and compare it with the strong dual formulation of \textit{constrained} DRO quantified by Sinkhorn distance obtained in \citet{wang2023sinkhorn}}.

% \subsection{Regularized Nonconvex DRO with Generalized Sinkhorn Distance}
% In this section, we first introduce a class of regularized nonconvex distributionally-robust optimization (DRO) problems, where the data distribution uncertainty is quantified by generalized Sinkhorn distance. Then, we study its dual formulation.

\subsection{Problem Formulation}

In distributionally-robust optimization (DRO), the goal is to learn a model which achieves good and robust performance when the underlying data distribution is uncertain. Specifically, consider a machine learning problem with the nonconvex loss function denoted by $\ell(x;\xi)$, where $x\in \mathbf{R}^d$ denotes the collection of model parameters and $\xi$ corresponds to a data sample that follows an underlying data distribution $\mathbb{Q}$. Then, with a regularization parameter $\lambda>0$, we study the following regularized DRO problem, which is a popular formulation in robust machine learning \citep{regularizedro1, regularizedvariation2,regularizeddro3}.
\begin{align}
    \min_{x \in \mathbf{R}^d} \sup_{\mathbb{Q}} \Bigl \{ \mathbb{E}_{\xi \sim \mathbb{Q}} \bigr [ \ell(x;\xi)\bigr] - \lambda W_{\red{\beta}}(\mathbb{P}, \mathbb{Q}) \Bigl \},
    \label{primal1}
\end{align}
where $W_{\red{\beta}}(\mathbb{P}, \mathbb{Q})$ denotes a certain function (with parameter $\red{\beta}>0$) that measures the 
distance between a nominal data distribution $\mathbb{P}$ and the underlying data distribution $\mathbb{Q}$. In particular, the operation $\min_{x}\sup_{\mathbb{Q}}$ aims to optimize the model $x$ under the worst-case data distribution $\mathbb{Q}$ to enhance model robustness against the distribution shift from the nominal distribution $\mathbb{P}$. In the existing literature, many studies have considered KL-divergence \citep{Hu2012KullbackLeiblerDC,bayesDRO}, $f$-divergence \citep{levy2020largescale,duchi2020learning,jin2021nonconvex,namkoong2016stochastic,smoothedfDRO} and Wasserstein distance \citep{NIPS2015_cc1aa436,esfahani2017datadriven,gaosa,LUO201920,joseOTDRO} to quantify the above distribution shift. However, both the KL-divergence and the $f$-divergence require $\mathbb{Q}$ and $\mathbb{P}$ to have the same probability support, which is restrictive and undesirable for many machine learning applications. The following two examples illustrate the limitations of $f$-divergence.
\begin{example}\label{ex: 1}
In robust Markov Decision Process (MDP) \citep{geometryRMDPValue}, denote the underlying environment's transition kernel as $\mathbb{Q}$. Then, robust reinforcement learning aims to optimize the following robust state value function over the policy $\pi$.
%our goal is to find the optimal policy $\pi^*$ under the probability transition kernel $\mathbb{Q}$. The corresponding optimal robust value function can be written as 
\begin{align}
V^\pi(s) := \inf_{\mathbb{Q}: D_f(\mathbb{Q}|\mathbb{P})\leq \rho}\mathbb{E}\Big[\sum_{t=0}^{\infty} \gamma^t r_{t} \mid \pi, s_0=s\Big],\nonumber
\end{align}
where $\gamma$ is a discount factor, $s$ corresponds to the state and $r_{t}$ represents the reward obtained after the $t$-th state transition. Here, the robust value function $V^\pi(s)$ considers the worst-case environment transition kernel $\mathbb{Q}$ over the uncertainty set defined by the 
$f$-divergence, i.e., $\{\mathbb{Q}: D_f(\mathbb{Q}|\mathbb{P})\leq \rho \}$. In this setting, if the nominal transition kernel $\mathbb{P}$ cannot visit a certain crucial state $s$, then neither can the transition kernels $\mathbb{Q}$ from the uncertainty set visit that state $s$. This indicates that $f$-divergence does not handle ``unknown'' uncertainties (e.g., states that never visited by $\mathbb{P}$).
%If  is utilized to characterize the distribution uncertainty of transition kernel, i.e.,  it requires the underlaying distribution $\mathbb{Q}$ has the same the support with the nominal distribution $\mathbb{P}$. That indicates if the transition probability from $s_i$ to $s_j$ equal to 0, i.e., $\mathbb{P}(s_i|a_k,s_j)=0$ , then it must hold $\mathbb{Q}(s_i|a_k,s_j)=0$. Such restriction may result learned optimal policy $\pi^*$ failing to capture the state transition which may be detached from real application if such state transition is crucial in reality.
\end{example}

\begin{example}\label{ex: 2}
%In distributionally-robust optimization (DRO), one aims to learn a model that is robust to data distribution shift. For example, 
Consider the following $f$-divergence-regularized DRO problem.
\begin{align}
    \min_{x \in \mathbf{R}^d} \sup_{\mathbb{Q}}\Big\{\mathbb{E}_{\xi \sim \mathbb{Q}}\big[\ell(x;\xi) \big] - \lambda D_f(\mathbb{Q}|\mathbb{P}) \Big\},\nonumber
\end{align}
where 
%$\ell(x;\xi)$ corresponds to the loss associated with data sample $\xi$ and model parameters $x$, and 
the distribution shift on data is characterized by the $f$-divergence. Suppose we want to train a face detection model. If the nominal data distribution $\mathbb{P}$ only covers face images collected from the majority group and excludes the minority groups, then the $f$-divergence DRO cannot yield a robust model over all the groups.
%is a tensor collecting all parameters from each layer of neural network, the loss function $\ell$ evaluates the difference between the output between output from neural network (denoted as $\mathcal{NN}(\mathbf{W};\xi)$) with respect to true label $y$ Similarly, if the nominal distribution $\mathbb{P}$ over dataset ignores the minority classes, the same case will happen for shifted distribution $\mathbb{Q}$, which cannot eliminate fairness issue over minority cases.
\end{example}
On the other hand, the classic Wasserstein distance does not require the distributions $\mathbb{P}, \mathbb{Q}$ to have the same probability support. However, it is known that Wasserstein distance suffers from computational intractability for high-dimension data \citep{emdcomplexity,ba2009sublinear}, and hence is not suitable for large-scale problems in machine learning.

To tackle these challenges, inspired by the Wasserstein distance, Sinkhorn distance \citep{cuturi2013Sinkhorn,wang2023sinkhorn}, we consider the following \textit{generalized} Sinkhorn distance to quantify the data distribution shift. To elaborate, we consider a sample space $\Omega$ associated with $\sigma$-algebra $\mathcal{F}$. Furthermore, for distributions $\mathbb{Q}$, $\mathbb{P}$ over a measurable subset of $\mathcal{F}$, we assume they are absolutely continuous with regard to some reference measures $\nu$ and $\mu$, i.e., $\mathbb{Q} \ll \nu$, $\mathbb{P} \ll \mu$. 
\begin{definition}[Generalized Sinkhorn Distance]\label{gSinkhorn}
 Consider probability distributions $\mathbb{Q}$ and $\mathbb{P}$ over $\bigr(\Omega,\mathcal{F} \bigr)$ and let $\nu$ and $\mu$ be reference measures satisfying $\mathbb{Q}\ll \nu$,$\mathbb{P}\ll \mu$.
 Denote $\Gamma(\mathbb{P}, \mathbb{Q})$ as the set of joint distributions that have marginal distributions $\mathbb{P}, \mathbb{Q}$. For a fixed regularization parameter $\red{\beta}>0$ and a cost metric $c: \Omega\times\Omega \to \mathbf{R}$, the generalized Sinkhorn distance is defined as
\begin{align}
    W_{\red{\beta}}(\mathbb{P}, \mathbb{Q}) =\inf_{\gamma \in \Gamma(\mathbb{P}, \mathbb{Q})}\Bigl \{\mathbb{E}_{(\zeta, \xi) \sim \gamma}\bigr [c(\zeta, \xi)\bigr ]+\red{\beta} D_f(\gamma | \mathbb{P} \otimes \nu)\Bigl \},\nonumber
\end{align}
where $D_f$ corresponds to the $f$-divergence, that is, 
\(
D_f(\gamma | \mathbb{P} \otimes \nu) = \int f \big(\frac{\mathrm{d} \gamma(\zeta, \xi)}{\mathrm{d} \mathbb{P}(\zeta) \mathrm{d} \nu(\xi)}\big) \mathrm{d}\nu(\xi)\mathrm{d} \mathbb{P}(\zeta),\nonumber\)
where the function $f: [0,+\infty )\rightarrow  [-\infty, +\infty \bigr ]$ is convex and satisfies $f(1)=0$ and $f(0) = \lim_{t\rightarrow 0+}f(t)$, and $\frac{\mathrm{d} \gamma(\zeta, \xi)}{\mathrm{d} \mathbb{P}(\zeta) \mathrm{d} \nu(\xi)}$ represents the density ratio of $\gamma$ with respect to $\mathbb{P}\otimes \nu$.
\end{definition}

\begin{remark}
The absolute continuity condition $\mathbb{Q}\ll\nu$, $\mathbb{P}\ll\mu$ is crucial to guarantee that the generalized Sinkhorn distance is well-defined. Typical choices of the reference measure $\nu$ include the Lebesgue measure or the Gaussian measure. 
In addition, when $\mathbb{Q}\ll \nu$,   $D_f(\gamma|\mathbb{P} \otimes \nu)$ and $D_f(\gamma|\mathbb{P} \otimes \mathbb{Q})$ are equivalent up to a constant, which does not affect the optimal solution in the regularized setting. Thus, we consider the former term for simplicity. 
\end{remark}
The considered generalized Sinkhorn distance is regularized by the $f$-divergence , which generalizes the KL-divergence regularization adopted in the definition of the standard Sinkhorn distance \citep{wang2023sinkhorn}. Such generalization still allows the distributions $\mathbb{P}$ and $\mathbb{Q}$ to have different probability support. By adding $f$-divergence regularization, it preserves the joint convexity with respect to the probability distributions and thus guarantees the uniqueness of the optimal solution, which helps reduce the computation complexity. Moreover, the generalized Sinkhorn distance provides more flexibility to model data distribution uncertainty compared to other divergence-based measures \citep{levy2020largescale,jin2021nonconvex}.
\subsection{Dual Formulation}
With generalized Sinkhorn distance, the regularized DRO problem \eqref{gSinkhorn} can be rewritten as
\begin{align}
    &\min_{x\in \mathbf{R}^d}\sup_{\mathbb{Q}}\Big\{
    \mathbb{E}_{\xi \sim \mathbb{Q}}\bigr[\ell(x;\xi)\bigr]- \inf_{\gamma \in \Gamma(\mathbb{P},\mathbb{Q})}\bigr\{\mathbb{E}_{(\zeta,\xi)\sim \gamma}\bigr[\lambda c(\xi,\zeta)\bigr]+\lambda\red{\beta} D_{f}(\gamma | \mathbb{P}\otimes \nu) \bigr\} \Big\}.
    \label{primal1-0}
\end{align}
The primal Sinkhorn distance regularized DRO problem \eqref{primal1-0} is hard to solve, since it is challenging to obtain an analytical form of the worst-case distribution $\mathbb{Q}$. However, the generalized Sinkhorn distance involves special structures that can transform the primal regularized DRO problem \eqref{primal1} into a simpler dual form. The following theorem deduces an equivalent dual formulation (\red{See Appendix \ref{proof: dual formulation thm} for proof details}).
\begin{restatable}[Dual formulation]{theorem}{dual}
\label{thm: dual formulation}
    The DRO problem \eqref{primal1-0} has the following equivalent dual formulation
    \begin{align}
    &\min_{x\in \mathbf{R^d}} \mathbb{E}_{\zeta \sim \mathbb{P}}[\Psi_{\zeta}(x)], ~\text{where }\Psi_{\zeta}(x) = \min_{\eta \in \mathbf{R}}    \mathbb{E}_{\xi \sim \nu}\big[\underbrace{\lambda \red{\beta} f^*\big(\frac{\ell(x;\xi)-\lambda c(\zeta,\xi)-\eta}{\lambda \red{\beta}}\big)+\eta\big]}_{\mathcal{L}_{\xi,\zeta}(x,\eta)},
    %\min_{x\in \mathbf{R^d}} \mathbb{E}_{\zeta \sim \mathbb{P}}\Big[ \Psi_{\zeta}(x) := \min_{\eta \in \mathbf{R}}    \mathbb{E}_{\xi \sim \nu}\big[\lambda \red{\beta} f^*(\frac{\ell(x;\xi)-\lambda c(\zeta,\xi)-\eta}{\lambda \red{\beta}})\big]+\eta \Big],
    \label{dual2}
    \end{align}
    and $f^*$ denotes the conjugate function of $f$.
\end{restatable}
\begin{remark}[Technical Novelty]
Proving the equivalence between the primal problem \eqref{primal1-0} and the dual problem \eqref{dual2} is crucial. To elaborate, we first decompose the joint distribution as $\gamma(\zeta,\xi) = \gamma_{\zeta}(\xi)\mathbb{P}(\zeta)$, where $\gamma_{\zeta}$ corresponds to the conditional distribution over $\xi$. Then, by the principle of interchangeability (Theorem 7.92, Chapter 7.3.2 in \citet{reference4interchangebility}), we are able to swap the order between $\mathbb{E}_{\zeta\sim \mathbb{P}}$ and 
$\sup_{\gamma_{\zeta}}$ without changing the optimal value, which yields
\begin{align}
     \min_{x\in \mathbf{R}^d} \mathbb{E}_{\zeta \sim \mathbb{P}}  \Big[&  \Psi_{\zeta}(x) = \sup_{
      \gamma(\xi|\zeta)} \big \{ \mathbb{E}_{\xi \sim \gamma(\cdot|\zeta)}\bigr[\ell(x; \xi) - \lambda c(\zeta, \xi )\bigr]-\lambda \red{\beta} D_f\big(\gamma(\xi|\zeta) | \nu(\xi)\big) \big \}\Big ].\nonumber
\end{align}
Then, by utilizing techniques of data processing inequality, we show that $\Psi_{\zeta}(x)$ is equivalent as follows auxiliary function
\begin{align}
    \widetilde{\Psi}_{\zeta}(x)=\sup_{ {\mu}_{\gamma|\zeta}} \!\Big\{\! \mathbb{E}_{{\mu}_{\gamma|\zeta}} \bigr[\ell(x; \xi) \!-\! \lambda c(\zeta, \xi )\bigr] -\lambda \red{\beta} D_f({{\mu}_{\gamma|\zeta}} | {\mu}_\nu)\!\Big \}.\nonumber
\end{align}
 where $\sup_{{\mu}_{\gamma|\zeta}}$ corresponds to the supremum over all possible distributions ${\mu}_{\gamma|\zeta}$ induced by $\gamma(\xi|\zeta)$.
 Last, by utilizing inverse c.d.f sampling on $\widetilde{\Psi}_{\zeta}(x)$ introduced in \citet{levy2020largescale,duchi2020learning}, we are able to derive equivalent formulation $  \Psi_{\zeta}(x) =
 \min_{\eta \in \mathbf{R}}\int_0^1 \sup _{r \in \mathbf{R}_{+}}\\\big[rF^{-1}(u)-\eta(r-1)-\lambda\red{\beta} f(r)\big] \mathrm{d} u$, apply Lagrange duality and definition of convex conjugate dual transforming $\Psi_{\zeta}(x)$ into desired formulation.
\end{remark}
\red{ We notice that in \citet{wang2023sinkhorn, SinkhornDROesaim}, the authors also present strong dual formulation for the constrained Sinkhorn DRO problem with $D_f$ being the KL divergence and constraint radius $\rho$. Specifically, they showed the following equivalent formulation of the problem
\begin{align}
        \min_{x\in \mathbf{R}^d,\lambda>0}\Big\{\lambda\rho+ \lambda \red{\beta}
        &\mathbb{E}_{\zeta \sim \mathbb{P}}\Big[ \log\big( \mathbb{E}_{\xi \sim \nu}\big[ \exp(\frac{\ell(x;\xi)-\lambda c(\xi;\zeta) }{\lambda \red{\beta}})\big]\big)\Big]\Big\}.
        \label{baseline SinkhornDRO}
\end{align} 
Such dual formulation takes a compositional structure of the form $\mathbb{E}_{\zeta}\log(\mathbb{E}_{\xi}(\exp(t)))$, where controlling the variance and bias of its stochastic gradient estimator is challenging. To elaborate, \citet{wang2023sinkhorn} demonstrates that sample-average estimation of gradient leads to a suboptimal convergence rate under mild assumptions. Such limitation motivates them studying variance reduction technique-multilevel Monte Carlo method (MLMC), to improve the sample complexity thereafter.}

\red{As a comparison, for proposed Sinkhorn DRO dual formulation~\eqref{dual2} derived using Lagrange duality, %if we choose the same KL divergence, i.e., $f(t)=t\log(t)$, the resultant strong dual formulation is
%     \begin{align}
%         \min_{x \in \mathbf{R}^d} \mathbb{E}_{\zeta \sim \mathbb{P}}&\Big[\min_{\eta \in \mathbf{R}}\mathbb{E}_{\xi\sim \nu}\big[ \lambda \red{\beta}
%         \exp\big(\frac{\ell(x;\xi)-\lambda c(\zeta,\xi)-\eta}{\lambda \red{\beta}}\big)-\lambda\red{\beta}\big]+\eta\Big].\label{KLdualformulation}
%     \end{align}one can see that
it involves nested minimization structure, where the inner minimization problem is with respect to the dual variable and sample $\zeta$. Such nested structure is challenging as most nested problems are solved by using bilevel optimization frameworks, where computing hyper-gradient estimation is known to be time-consuming due to the requirement of second-order information \citep{bilevel_hessian}. Moreover, the dependency between the inner minimizer $\eta$ and $\zeta$ motivates us to reduce the number of sampled $\zeta$ to $\mathcal{O}(1)$-level per iteration. Thanks to Lemma \ref{thm: differentiability} proved by \citet{jin2021nonconvex}, which provides a tractable expression for evaluating gradients without requiring second-order information. As we later demonstrate in Theorem \ref{thm: grad_error_bound} and \ref{thm: convergenestsgd}, one can efficiently control the gradient approximation error and establish convergence of nested SGD for solving proposed dual formulation~\eqref{dual2} without querying a large batch of $\zeta$ and $\xi$ per iteration.} 

\section{Solving the Dual Problem via Nested SGD}\label{sec: SGD}
\red{To minimize optimization problem with contextual nested structure, \citet{hu2023contextual} studied bilevel SGD framework and attained near optimal sample complexity $\tilde{\mathcal{O}}(\varepsilon^{-4})$ when lower level problem is strongly convex and MLMC is used for hyper-gradient estimation. However, their analysis is not directly applicable to proposed Sinkhorn DRO dual formulation~\eqref{dual2}, as their framework assumes dependency between $\xi$ and $\zeta$, but does not account for the dependency between the dual variable $\eta$ and the sample $\zeta$. Additionally, most conjugate dual functions $f^*$ for information divergences are not globally strongly convex, which further limits the applicability of their analysis. As we mentioned before, \citet{wang2023sinkhorn,hu2023contextual} showed that MLMC requires storing multiple samples of $\zeta$ up to $\mathcal{O}(\log(\varepsilon^{-1}))$ to construct a stochastic estimator with $\mathcal{O}(\varepsilon)$-level approximation error, and second-order derivative remains unavoidable for evaluating hyper-gradient in contextual bilevel objective.}

\red{Motivated by these bottlenecks, in this section, we present a tractable expression for computing the gradient of proposed Sinkhorn DRO dual formulation \eqref{dual2} and we then find, by exploring the structural relationship between the inner and outer objective gradients, that the requirements of strong convexity assumption and sampling multiple $\zeta$ can be eliminated to control the approximation error between $\nabla\mathbb{E}_{\zeta}[\Psi_{\zeta}(x)]$ and $\nabla_1\mathbb{E}_{\zeta}[\mathcal{L}_{\zeta}(x,\eta_x^{\tilde{d}}(\zeta))]$ given $\eta_x^{\tilde{d}}(\zeta)$ is an output of stochastic oracle satisfying mild conditions.} Recall that the dual formulation~\eqref{dual2} consists of two stochastic optimization problems. 
%In this section, we explore the mathematical structures of the main problem $\min_{x} \mathbb{E}_{\zeta \sim \mathbb{P}}[\Psi_{\zeta}(x)]$ and potential algorithm design. 
Throughout, we denote the objective function of the inner problem in~\eqref{dual2} as
\begin{align}
    \mathcal{L}_{\zeta}(x,\eta)=\mathbb{E}_{\xi \sim \nu}\big[\lambda \red{\beta} f^* \big(\frac{\ell(x;\xi)-\lambda c(\zeta,\xi)-\eta}{\lambda \red{\beta}}\big)\big]+\eta.
    \label{eq: inner problem}
\end{align}
For simplicity, we denote $\eta_x^*(\zeta)=\arg\min_{\eta}L_{\zeta}(x,\eta)$ to highlight its dependence on the fixed parameter $x$ and data sample $\zeta$. We also denote $\mathbb{E}_{\zeta}\bigr[\Psi_{\zeta}(x)\bigr]$ as the objective function of the outer problem.

To analyze the problem structure, we adopt the following standard assumptions on the loss functions and convex conjugate dual of chosen $f$-divergence. 
\begin{assumption}\label{assum1}
The functions in Sinkhorn DRO dual formulation~\eqref{dual2} satisfy:
\begin{itemize}[topsep = 1pt, itemsep=1pt]
    %\item \red{For any $x,\eta$, the expectation over $\xi, \zeta$ for $\mathcal{L}_{\xi,\zeta}(x,\eta)$ and $\mathcal{L}_{\zeta}(x,\eta)$ defined in \eqref{dual2} exists.}
    %\item \red{For any $x,\eta$, the expectation and gradient are inter-changeable,
    %i.e., $\mathbb{E}_{\zeta}[\nabla_{1\text{ or }2}\mathcal{L}_{\zeta}(x,\eta)]=\nabla_{1\text{ or }2} \mathbb{E}_{\zeta}[\mathcal{L}_{\zeta}(x,\eta)]$ and $\mathbb{E}_{\xi}[\nabla_{1\text{ or }2}\mathcal{L}_{\zeta,\xi}(x,\eta)]=\nabla_{1\text{ or }2} \mathbb{E}_{\xi}[\mathcal{L}_{\zeta,\xi}(x,\eta)]$}. 
    \item For every $\xi$, $\ell(\cdot;\xi)$ is $G$-Lipschitz continuous.
    i.e., 
    $\big \| \ell(x;\xi) - \ell(y; \xi) \big \| \leq G \big \| x-y \big \|$.
    \item For every $\xi$, $\ell(\cdot;\xi)$ is continuously differentiable and $L$-smooth. i.e., $\big \| \nabla \ell(x;\xi) - \nabla \ell(y;\xi)  \big \| \leq L\big \| x-y \big \|$.
    \item The conjugate function $f^*$ is continuously differentiable and $M$-smooth.
    \item The objective function $\mathbb{E}_{\zeta \sim \mathbb{P}}\bigr[\Psi_{\zeta}(\cdot)\bigr]$ is bounded below.
    
\end{itemize}
\end{assumption}
\begin{remark}
Note that the loss function $\ell(x;\xi)$ is generally nonconvex. Regarding the smooth assumption on $f^*$ in the third item, some typical examples of $f$-divergence include the $\chi^2$-divergence, smoothed CVaR divergence \citep{jin2021nonconvex}, where their corresponding conjugate functions are given by $f^*(y)=-1+\frac{1}{4}(y+2)^2_{+}$ and $f^*(t)=\frac{1}{\alpha}\log(1-\alpha+\alpha\exp(t))$. \red{To analyze the dual formulation~\eqref{dual2} with KL-divergence, we need to adopt assumption $\sup_x \exp((\ell(x;\xi)-\lambda c(\zeta,\xi)-\eta)/\lambda\beta)<\infty$ holds almost-surely for every $\xi$ given $\zeta$.
%introduced by \citet{wang2023sinkhorn} including \textit{bounded loss assumption} $|\ell(x;\xi)|<\infty$ holds for any $x$, $\xi$, cost function $c(\zeta, \xi)$ is $\mathbb{P}\otimes\nu$-measurable satisfying $\nu\{ \xi:0\leq c(\zeta,\xi)<\infty\}=1$ for $\mathbb{P}$-almost every $\zeta$, and $|\eta|<\infty$.
%And \textit{integrability condition}, $\mathbb{E}_{\xi\sim\nu}[\exp(-c(\xi;\zeta)/\lambda\beta)]<\infty$ holds almost surely for every $\zeta\sim \mathbb{P}$. 
Under these assumptions, one can ensure that the convex conjugate dual $f^*(t)=\exp(t)-1$ is locally $M$-smooth within a bounded domain. We will later demonstrate in our ablation study (Appendix~\ref{Appendix: ablation study}) that the $f^*$ induced by the KL-divergence yields similar convergence behavior when optimizing dual formulation \eqref{dual2}, compared to choices of $f^*$ that satisfy the $M$-smoothness property, suggesting that this assumption does not limit practical applicability.}
\end{remark}
Since the dual problem takes a nested form, we need an efficient way to compute the gradient of the objective function.
The following lemma, proved by \cite{jin2021nonconvex}, provides an analytical formula for computing the exact gradient. (\red{See Lemma 2.6 in \citet{jin2021nonconvex} for proof details})
\begin{restatable}[Computation of $\nabla \Psi_{\zeta}(x)$ \citep{jin2021nonconvex}]{lemma}{jin}\label{jin}
Let Assumption \ref{assum1} hold and consider fixed $x$ and given $\zeta$. Then, the function $\Psi_{\zeta}(x)$ is differentiable and satisfies $\nabla \Psi_{\zeta}(x) = \nabla_1 \mathcal{L}_{\zeta}(x,\eta_x^*(\zeta))$, where $\allowdisplaybreaks \eta_x^*(\zeta) \in \arg\min_{\eta} \mathcal{L}_{\zeta}(x,\eta)$.
\label{thm: differentiability}
\end{restatable}

This lemma shows that, given the exact minimizer $\eta_{x}^*(\zeta)$ of the inner problem, one can directly evaluate the gradient of dual formulation~\eqref{dual2}. Notably, it eliminates the need to acquire second-order derivatives for computing $\nabla_1 \eta_{x}^*(\zeta)$, which are typically required under bilevel optimization framework \citep{bilevel_hessian,ghadimibilevelapproximation,hu2023contextual}. Motivated by Lemma~\ref{thm: differentiability}, we aim to develop conditions and algorithms to estimate $\eta_x^*(\zeta)$ with arbitrarily small error, thereby enabling the construction of an inexact gradient estimator for proposed Sinkhorn DRO dual formulation~\eqref{dual2} with an $\varepsilon$-level approximation error. \red{The following theorem shows that, for any fixed $x$ and $\zeta$, the approximation error in estimating $\nabla \Psi_{\zeta}(x)$ can be made arbitrarily small by querying an inner solution $\eta_{x}^{\tilde{d}}(\zeta)$ that is accurate on average. (See Appendix \ref{Proof: grad_error_bound thm} for proof details)}
%the threshold can further indicate $\varepsilon$-optimality condition for $\min_{\eta}\mathcal{L}_{\zeta}(x,\eta)$ and approximation error between $\nabla \Psi_{\zeta}(x)$ and $\nabla_1 \mathcal{L}_{\zeta}(x,\eta)$.

%\begin{restatable}{theorem}{approximationerror}\label{thm: grad_error_bound}
    % choose batch size $\tilde{B}= \Theta(\log(\frac{2}{\omega})R_2 \varepsilon^{-2})$ for some $\varepsilon>0$. Then, with probability at least $1-\omega$,  the stochastic approximation error of $v^{\tilde{B}}(\eta)$ can be bounded as 
    %\begin{align}
    %\big|\Lambda\big|=\big|v^{\tilde{B}}(\eta) - \nabla_{2} \mathcal{L}_{\zeta}(x,\eta) \big | \leq \frac{\varepsilon}{4}. \label{eq: sto_error}
    %\end{align}
    %Furthermore, suppose there exists a stochastic algorithm that solves \eqref{eq: inner problem} and produces a sequence $\{\eta_{x_t}^{d}(\zeta)\}_d$ with convergence guarantee
    %$
    %\mathbb{E}_d \big|\nabla_{2} \mathcal{L}_{\zeta}(x,\eta_{x_t}^{d}(\zeta)) \big| \leq \frac{\varepsilon}{2}
    %$, and set
    %$
    %\tilde{d}=\arg\min_d \big|v^{\tilde{B}}(\eta_{x_t}^{d}(\zeta)) \big|
    %$.
    %Then, with probability at least $1-\omega$, the full gradient $\nabla \Psi(x)$ and $\nabla_{2} \mathcal{L}_{\zeta}(x,\eta_{x}^{\tilde{d}}(\zeta))$ satisfy
    %\begin{align}
    %\bigl |\nabla_{2} \mathcal{L}_{\zeta}(x,\eta_{x}^{\tilde{d}}(\zeta))\bigl |\leq \varepsilon \quad \text{and}\bigl \|  \nabla \Psi_{\zeta}(x)\bigl \| \leq G\varepsilon + \bigl \|  \nabla_1 \mathcal{L}_{\zeta}(x,\eta_{x}^{\tilde{d}}(\zeta))\bigl \|, \forall  \zeta\sim \mathbb{P}.
    %\label{optimalrelation}
    %\end{align}
%\end{restatable}
\begin{restatable}[Gradient approximation error bound]{theorem}{approximationerror}\label{thm: grad_error_bound}
    \red{
    %Given target error $\varepsilon \in (0,1]$, set batch size $\tilde{B}= \Theta(R_2 \varepsilon^{-2})$. Then, for any $\eta$, the stochastic approximation error of $v^{\tilde{B}}(\eta)$ can be bounded as 
    %\begin{align}
    %\mathbb{E}_{\xi\sim\nu}\big|\Lambda\big|=\mathbb{E}_{\xi\sim\nu}\big|v^{\tilde{B}}(\eta) - \nabla_{2} \mathcal{L}_{\zeta}(x,\eta) \big | \leq \mathcal{O}(\varepsilon). \label{eq: sto_error}
    %\end{align}
    Consider a stochastic algorithm minimizing \eqref{eq: inner problem}. If the stochastic oracle outputs an $\eta_{x}^{\tilde{d}}(\zeta)$ converging to $\nabla_2\mathcal{L}_{\zeta}(x_t,\eta_x^*(\zeta))$ with scaled small target error $\tilde{\varepsilon}={\varepsilon}/{G}$, i.e.,
    \begin{align}
    \mathbb{E}_{\eta_{x}^{\tilde{d}}(\zeta)} \big|\nabla_{2} \mathcal{L}_{\zeta}(x,\eta_{x}^{\tilde{d}}(\zeta)) \big|^2 \leq \tilde{\varepsilon}^2,
    \label{eq: inner near-optima}
    \end{align}
    then the gradient $\nabla_{1} \mathcal{L}_{\zeta}(x,\eta_{x}^{\tilde{d}}(\zeta))$ approximates full gradient $\nabla \Psi(x)$ with error up to $\varepsilon$, i.e.,
    \begin{align}
    \bigl \|  \nabla \Psi_{\zeta}(x)- \mathbb{E}_{\eta_{x}^{\tilde{d}}(\zeta)}[  \nabla_1 \mathcal{L}_{\zeta}(x,\eta_{x}^{\tilde{d}}(\zeta))]\bigl \|^2 \leq \varepsilon^2, \forall  \zeta\sim \mathbb{P}.
    \label{optimalrelation}
    \end{align}}
\end{restatable}
\begin{remark}[Technical Novelty]
    The novelty of our proof lies in the usage of the monotonicity property of $(f^*)'$, which enables us to perform an equivalence transformation by moving the expectation $\mathbb{E}_{\xi \sim \nu}$ into the norm $\big \|\cdot \big\|$. Thanks to special structure of formulation~\eqref{dual2}, such operation swap doesn't change its value since each inner problem $\min_{\eta} \mathcal{L}_{\zeta}(x,\eta)$ depends on a fixed $\zeta$. In the proof, we also utilize the convexity of the conjugate function $f^*(\frac{\ell(x;\xi)-c(\zeta;\xi)-\eta}{\lambda \red{\beta}})$ in $\eta$, albeit the loss function $\ell(x;\xi)$ is generally nonconvex.
\end{remark}
\red{Note that Theorem~\ref{thm: grad_error_bound} requires an algorithm capable of generating an $\eta_{x}^{\tilde{d}}(\zeta)$ that satisfies condition \eqref{eq: inner near-optima}, which can be achieved by the vanilla SGD algorithm \citep{ghadimi2013stochastic} under mild assumptions. And the condition \eqref{eq: inner near-optima} cannot be further reduced to other forms, as the second-order condition \eqref{eq: inner near-optima} will be reused when estimating the second moment of the inexact gradient (See Lemma \ref{thm: boundedmoment3})}.

\red{Additionally, condition \eqref{optimalrelation} characterizes the relationship between the gradients $\nabla \Psi_{\zeta}(x)$ and $\nabla_1 \mathbb{E}_{\eta_x^{\tilde{d}}(\zeta)}[\mathcal{L}_{\zeta}(x,\eta_{x}^{\tilde{d}}(\zeta))]$ when the inner problem \eqref{eq: inner problem} is solved approximately by an inexact minimizer $\eta_{x}^{\tilde{d}}(\zeta)$. This observation motivates the design of a nested-type stochastic algorithm (detailed in the next subsection), in which the inner algorithm computes an inexact minimizer $\eta_{x}^{\tilde{d}}(\zeta)$ for $\mathcal{L}_{\zeta}(x,\eta)$, and the outer stochastic algorithm subsequently optimizes $\min_{x}\mathbb{E}_{\zeta\sim \mathbb{P}}[\mathcal{L}_{\zeta}(x,\eta)]$. This approach is justified by taking the expectation over $\zeta\sim\mathbb{P}$ and applying Jensen's inequality to \eqref{optimalrelation}, which further implies
\begin{align}
&\big\|\nabla_1 \mathbb{E}_{\zeta\sim \mathbb{P},\eta_{x}^{\tilde{d}}(\zeta)}\big[\mathcal{L}_{\zeta}(x,\eta_{x}^{\tilde{d}}(\zeta))\big] - \nabla \mathbb{E}_{\zeta\sim \mathbb{P}}\big[ \Psi_{\zeta}(x)\big]\big\|^2\nonumber\\
&\leq \mathbb{E}_{\zeta \sim \mathbb{P}}\big\|\nabla_1 \mathcal{L}_{\zeta}(x,\eta_{x}^{\tilde{d}}(\zeta)) - \nabla \Psi_{\zeta}(x)\big\|^2\leq \varepsilon^2.
\label{eq: induced from grad_error_bound}
\end{align}
Through condition \eqref{eq: induced from grad_error_bound}, we conclude that one can establish bias and variance guarantee with respect to the true gradient $\|\nabla \mathbb{E}_{\zeta\sim\mathbb{P}}[\Psi_{\zeta}(x)]\|$ by sub-sampling $\mathcal{O}(1)$-batches of $\zeta$.}

\section{Algorithm design and convergence analysis}

In this section, we propose a nested procedure composed of Algorithm~\ref{alg1} and Algorithm~\ref{alg2} to sequentially optimize $x$ and estimate $\eta_x^*(\zeta)$ via SGD-type algorithms. To facilitate the convergence analysis of these algorithms, we impose the following bounded variance assumptions on the loss function $\ell(x;\xi)$ and the cost metric $c(\zeta;\xi)$. Through the work, to account for randomness arising from high-dimensional data, we use the variance definition
$Var_{\varpi}(\rho(\varpi))=\mathbb{E}\bigr[\rho(\varpi)-\mathbb{E}\bigr[\rho(\varpi)\bigr]\bigr]^2$ in our assumption, where $\rho(\cdot): \mathbf{R}^d \rightarrow \mathbf{R}$ denotes a function mapping a random variable $\varpi \in \mathbf{R}^d$ to a scalar.
\begin{assumption}\label{assum2}
There exists $\sigma, \delta>0$ such that:
\begin{itemize}[topsep = 1pt, itemsep=1pt]
    \item For every $x$, the variance  of $\ell(x;\cdot)$ over $\xi$ is bounded by $\sigma^2$, i.e., $Var_{\xi}\big(\ell(x;\xi)\big)\leq \sigma^2$.
    \item For every $\zeta$, the variance of $c(\zeta,\cdot)$ over $\xi$ is bounded by $\delta^2$, i.e., $Var_{\xi}\big(c(\zeta,\xi)\big) \leq \delta^2$.
    \item For every $\xi$, the variance of $c(\cdot,\xi)$ over $\zeta$ is bounded by $\delta^2$, i.e., $Var_{\zeta}\big(c(\zeta,\xi)\big) \leq \delta^2$.
\end{itemize}
\end{assumption}
\subsection{Inexact Estimation of Inner Minimizer via SGD}
%The previous Section \ref{sec: SGD} solves the dual problem in \eqref{dual2} by assuming that the minimizer $\eta_x^*(\zeta)$ of the inner problem is given. 
%In this section, we propose a stochastic algorithm to obtain an approximate minimizer of the inner problem, and then develop a practical nested SGD algorithm to solve the entire dual problem without the need of exact minimizer $\eta_{x}^*(\zeta)$ assumed in lemma \ref{thm: differentiability}. 

Recall that the inner problem of the dual problem~\eqref{dual2} takes the following form. 
\begin{align}
    \allowdisplaybreaks
    \min_{\eta \in \mathbf{R}}& \mathcal{L}_{\zeta}(x,\eta)=\mathbb{E}_{\xi \sim \nu}\big[\lambda \red{\beta} f^*\big(\frac{\ell(x;\xi)-\lambda c(\zeta,\xi)-\eta}{\lambda \red{\beta}}\big)+\eta \big],
\end{align}
which is a stochastic optimization problem. Inspired by condition~\eqref{eq: inner near-optima} stated in Theorem~\ref{thm: grad_error_bound}, we utilize a SGD-type algorithm (Algorithm \ref{alg1}) for computing an inexact estimation of $\eta_{x}^*(\zeta)$. Specifically, 
for fixed $x$ and sample $\zeta$, we compute the mini-batch stochastic gradient estimator of $\nabla_2 \mathcal{L}_{\zeta}(x,\eta)$ at each iteration as follows
\begin{align}
    v^{\tilde{B}}_{x,\zeta}(\eta) =
    1-\frac{1}{\tilde{B}}\sum_{i=1}^{\tilde{B}}(f^*)^{\prime}\big(\frac{\ell(x;\xi_i) - c(\zeta,\xi_i)-\eta}{\lambda \red{\beta}}\big).
    \label{gradient2}
\end{align}
For simplicity, we simplify notation $v_{x,\zeta}^{\tilde{B}}(\eta)$ as $v^{\tilde{B}}(\eta)$ when the dependence on $x$ and $\zeta$ is clear from the context.
\begin{algorithm}[htbp]
\caption{SGD for estimating $\eta^*_{x}(\zeta)$}\label{alg1}
\begin{algorithmic}
\STATE \textbf{Input:} Initialization $\eta^0$; sample $\zeta$; number of iteration $D$; batch size $\tilde{B}$.
%\STATE Initialize $\widetilde{\eta}=\eta^0$.
    \WHILE{$d<D$}
        \STATE Draw samples $\left\{ \xi \right\}_{\tilde{B}}\sim \mathbb{\nu}$ with batch size $\tilde{B}$.
        %\State Evaluate $\mathbb{E}_{\zeta}L_{\zeta}(x_{t},\eta_{x_t}^d(\zeta))$ at $x_t$ and $\eta_{x_t}^d(\zeta)$.
        \STATE Compute $v(\eta_{x_t}^{d}(\zeta))$ via \eqref{gradient2}.
        %\STATE \red{Set learning rate $\alpha_d=\min\{, \frac{\varepsilon}{|v(\eta_{x_t}^{d}(\zeta))|}\}$}
        \STATE Update $\eta_{x_t}^{d+1}(\zeta) = \eta_{x_t}^{d}(\zeta) - \alpha_d v^{\tilde{B}}(\eta_{x_t}^{d}(\zeta))$. 
        %\IF{$ |v(\eta_{x_t}^{d+1}(\zeta)) |\leq | v(\eta_{x_t}^{d}(\zeta))|$}
        %    \STATE   Set $\widetilde{\eta}=\eta_{x_t}^{d+1}(\zeta)$.
        %\Else{ Keep $\widetilde{\eta}$ unchanged }
        %\ENDIF
        %\State Evaluate  $\mathbb{E}_{\zeta}L_{\zeta}(x_t,\eta_{x_t}^{d+1}(\zeta))$.
        %\If{$L(x_t,\eta_{x_t}^{d+1}(\zeta))<\mathbb{E}_{\zeta}L_{\zeta}(x_{t},\eta_{x_t}^d(\zeta))$}
        %\State Update $\eta_{x_t}^d(\zeta)$ by $\eta_{x_t}^{d+1}(\zeta)$
        %\EndIf
    \ENDWHILE
    \STATE {\bf Output:} \red{$\eta_{x}^{\tilde{d}}(\zeta)$, where $\tilde{d}$ uniformly sampled from $\{0, 1, 2, .... D-1 \}$}.
\end{algorithmic}
\end{algorithm}
%When the stochastic estimator $v^{\tilde{B}}(\eta)$ is used to approximate the full gradient
%$\nabla_2
%\mathcal{L}_{\zeta}(x,
%\eta^*_{x}(\zeta))$, it introduces inexactness to the gradient estimation as well as the convergence of the algorithm. Although the approximation error composes two parts, one is stochastic approximation error between $v^{\tilde{B}}(\eta)$ and $\nabla_2\mathcal{L}_{\zeta}(x,\eta)$, and another one is inexactness error between $\nabla_2\mathcal{L}_{\zeta}(x,\eta)$ and $\nabla_2\mathcal{L}_{\zeta}(x,\eta_{x}^*(\zeta))$,
%To quantify such approximation errors, we adopt the following assumption on the stochastic approximation error between $v^{\tilde{B}}(\eta)$ and $\nabla_2\mathcal{L}_{\zeta}(x,\eta)$.
%\begin{assumption}[Sub-Gaussian Noise]\label{assum3}
%For any $x$, $\zeta$ and $\eta$, the stochastic approximation error $\Lambda = v^{\tilde{B}}(\eta)-\nabla_2\mathcal{L}_{\zeta}(x,\eta)$ follows sub-Gaussian distribution with variance $R_2/\tilde{B}$ for some $R_2, \tilde{B} >0$, i.e.,  
$
%\mathbb{E}[\exp\big(s\Lambda\big)] \leq \exp (\frac{s^2R_2}{2\tilde{B}}),  \forall s\in \mathbf{R}.
$
%\end{assumption}
%In this section, we develop comprehensive convergence analysis of the proposed Nested SGD algorithm.

To analyze the convergence of the inner Algorithm \ref{alg1} based on Assumptions \ref{assum1} and \ref{assum2}, we first obtain the following expected smoothness property (\red{See Appendix \ref{Proof: inner smooth thm} for proof details})
\begin{restatable}[$K'$-smoothness of inner objective \eqref{eq: inner problem}]{lemma}{innersmooth}\label{thm: smooth2}
       Let Assumption \ref{assum1} hold and
       denote $\mathcal{L}_{\zeta,\xi}(x,\eta) =  \lambda \red{\beta} \mathbb{E}_{\xi\sim \nu}\bigr[ f^{*}(\frac{\ell(x;\xi) - c(\zeta,\xi) - \eta}{\lambda \red{\beta}})\bigr]+\eta$. Then, for any $\eta$ and $\eta'$, we have
    \begin{align}
    &\mathbb{E}_{\xi\sim \nu}\big \| \nabla_{2} \mathcal{L}_{\zeta,\xi}(x,\eta) - \nabla_{2} \mathcal{L}_{\zeta,\xi}(x, \eta') \big \|^2 \leq (K')^2\big \| \eta - \eta' \big \|^2, \nonumber
    \end{align}
where $K' = M(\lambda \red{\beta})^{-1}$.
\end{restatable}
\noindent We then obtain the upper bound estimate for second moment of $\nabla_2 \mathcal{L}_{\zeta}(x,\eta)$ as follows (\red{See Appendix \ref{proof: inner gradient bounded moment} for proof details}).
\begin{restatable}[Second moment bound for $\nabla_2\mathcal{L}_{\zeta,\xi}(x,\eta)$]{lemma}{secondmomentboundtwo}\label{thm: boundedmoment2}
    Let Assumption \ref{assum2} hold. Then, $v^{\tilde{B}}(\eta)$ is the unbiased estimator for $\nabla_2 \mathcal{L}_{\zeta}(x,\eta)$, and
    the second moment of $\nabla_2 \mathcal{L}_{\zeta,\xi}(x,\eta)$ satisfies
    \begin{align}
        \mathbb{E}_{\xi\sim\nu} \left \|  \nabla_2 \mathcal{L}_{\zeta,\xi}(x,\eta) \right \|^2 &\leq R_2+ \left \|\nabla_2 \mathcal{L}_{\zeta}(x,\eta) \right \|^2,
    \end{align}
    where $R_2= 2M^2(\lambda \red{\beta})^{-2}(\sigma^2+\lambda^2 \delta^2)$.
\end{restatable}
Based on $K'$-smooth and affine bounded second moment shown above, we establish the following convergence result of Algorithm \ref{alg1} (\red{See Appendix \ref{proof: thm innersgd} for proof details}).
\begin{restatable}[Convergence of Algorithm~\ref{alg1}]{corollary}{innersgd}\label{thm: innersgd}
    \red{Let Assumptions \ref{assum1} and \ref{assum2} hold, denote $\widehat{\Delta}=\sup_{\zeta}\{L_{\zeta}(x_t,\eta^0)\\-\Psi_{\zeta}(x_t)\}$. Apply Algorithm \ref{alg1} to solve the inner problem \eqref{eq: inner problem} with learning rate $\alpha_d =\min\{ \frac{1}{K'},\frac{\tilde{\varepsilon}^2}{R_2K'}\}$ and batch size $\tilde{B} = 1$, then Algorithm \ref{alg1} outputs an $\eta_{x}^{\tilde{d}}(\zeta)$ satisfying
    \begin{equation}
        \mathbb{E}_{\eta_{x}^{\tilde{d}}(\zeta)}\big | \nabla_{2} \mathcal{L}_{\zeta}(x,\eta_{x}^{\tilde{d}}(\zeta))\big|^2\leq \tilde{\varepsilon}^2.
    \end{equation}
    In particular, it takes $D=\mathcal{O}(\hat{\Delta}K'R_2{\tilde{\varepsilon}^{-4}})=\mathcal{O}(\hat{\Delta}K'R_2G^4{\varepsilon^{-4}})$ number of iterations to obtain an ${\tilde{\varepsilon}}$-stationary point, and the stochastic gradient oracle complexity is $\mathcal{O}(\widehat{\Delta} K'R_2G^4\varepsilon^{-4})$.}
\end{restatable}
\begin{remark}\label{remark: large batch size for inner problem}
    \red{Notice that solving the inner optimization problem~\eqref{eq: inner problem} to obtain an $\eta_{x}^{\tilde{d}}(\zeta)$ only requires solving a one-dimensional optimization problem, where evaluating its stochastic gradient does not require expensive tensor computation or backpropagation. Consequently, employing a large batch size can substantially reduce the iteration complexity without significantly increasing computational overhead. For example, by selecting  batch size as $\tilde{B}= \Theta(G^{-2}\tilde{\varepsilon}^{-2})=\Theta(\varepsilon^{-2})$ and setting the learning rate $\alpha_d = \min\{\frac{1}{K'},\frac{1}{2K'R_2G^2}\}$, Algorithm~\ref{alg1} generates an $\eta_{x}^{\tilde{d}}(\zeta)$ satisfying condition~\eqref{eq: inner near-optima} after $D \geq 8R_2\hat{\Delta}K'G^4{\varepsilon}^{-2}=\mathcal{O}(R_2\hat{\Delta}K'G^4\varepsilon^{-2})$ iterations.}
\end{remark}

\subsection{Nested SGD and its convergence analysis}
 %intuitively, with the help of Lemma \ref{thm: differentiability} and Theorem \ref{thm: grad_error_bound}, one can directly apply the standard SGD algorithm to solve the outer stochastic problem $\min_{x\in \mathbf{R}^d} \mathbb{E}_{\zeta \sim \mathbb{P}} \bigr[\Psi_{\zeta}(x)\bigr]$ when $\eta_x^*(\zeta)$ is available. 

%In this section, we consider such an ideal case to help develop an SGD algorithm (presented in Algorithm \ref{alg0}) to solve the outer problem. Specifically, at each iteration $t$, the algorithm samples one $\zeta$ and a batch of $\bigl\{\xi \bigl\}_{B}$ with batch size $B$ to construct a stochastic gradient estimator of $\nabla \mathbb{E}_{\zeta \sim \mathbb{P}}\bigr[\Psi_{\zeta}(x)\bigr]$, which takes the form
%\begin{align}
%    g_t^B= \frac{1}{B} \sum_{i=1}^{B} (f^*)'\Bigr(\frac{\ell(x;\xi_i)-c(\zeta,\xi_i)-\eta^*_{x_t}(\zeta)}{\lambda \red{\beta}}\Bigr) \nabla \ell(x_t;\xi_i).
%    \label{gradient1}
%\end{align}
%With these key lemmas, we obtain the following convergence rate of Algorithm \ref{alg0} by following
%the convergence analysis of the standard SGD developed in the existing literature \citep{ghadimi2013stochastic,arjevani2022lower}.

The previous section has shown that SGD-type algorithms can successfully generate an $\eta_x^{\tilde{d}}(\zeta)$ satisfying the conditions required by Theorem \ref{thm: grad_error_bound}. In this section, we introduce the nested SGD algorithm (Algorithm~\ref{alg2}) and establish its convergence guarantees. Specifically, at each iteration, Algorithm~\ref{alg2} samples a single $\zeta$ along with a mini-batch $\{\xi\}_B$, and obtain an inexact estimate of $\eta^*_{x}(\zeta)$ by calling Algorithm~\ref{alg1}. It is worth noting that the sampled $\xi$ from Algorithm~\ref{alg1} can be reused to save computation overhead. Algorithm~\ref{alg2} then evaluates a mini-batch gradient estimator of $\nabla \mathbb{E}_{\zeta}[\Psi(x_t)]$ as follows
\begin{align}
    \red{\hat{g}(x_t;\zeta;\xi_{B})= \frac{1}{B}\sum_{i=1}^{B} (f^*)'\big(\frac{\ell(x;\xi_i)-c(\zeta,\xi_i)-\eta_{x}^{\tilde{d}}(\zeta)}{\lambda \red{\beta}}\big) \nabla \ell(x_t;\xi_i)}.
    \label{gradient3}
\end{align}
\begin{remark}
    \red{In the stochastic gradient estimator \eqref{gradient3}, note that we only apply mini-batch sampling over $\xi$. This is because the inexact minimizer $\eta_x^{\tilde{d}}(\zeta)$ explicitly depends on $\zeta$ and thus can vary over different $\zeta$. We notice that \citet{wang2023sinkhorn} proposes RT-MLMC estimator, which is specifically designed to reduce the upper bound of the second moment using fewer samples. In our setting, their RT-MLMC gradient estimator can be rewritten as
    \begin{align}
        g^{\text{RT-MLMC}}(x_t)&=\frac{1}{n^{\circ}_q}\sum_{i=1}^{n^{\circ}_q} \frac{1}{\mathbb{P}(\iota=\iota_i)}A^{\iota_i}(x_t;\zeta^{\iota_i}),\text{ where }\nonumber \\
        A^{\iota_i}(x_t;\zeta^{\iota_i}) =& \hat{g}(x_t;\zeta^{\iota_i};\xi_{1:2^{l}}) - \frac{1}{2}\hat{g}(x_t;\zeta^{\iota_i};\xi_{1:2^{l-1}}) - \frac{1}{2}\hat{g}(x_t;\zeta^{\iota_i};\xi_{2^{l-1}+1:2^{l}}),
        \label{eq: rt-mlmc}
    \end{align}
    where $n_{q}^{\circ}$ denotes the randomly sampled levels $\iota_1,\dots,\iota_{n_q^{\circ}}$, where each level $\iota_i$ is sampled independently following the distribution $\mathbb{P}(\iota = \iota_i)=\frac{2^{-\iota_i}}{2-2^{-q}}$, for $l = 0,\dots,q$.
    To ensure convergence, \citet{wang2023sinkhorn} demonstrates that by choosing hyper-parameters $n^{\circ}_{q}=\mathcal{O}(1), q =\mathcal{O}(\log(\varepsilon^{-1}))$, RT-MLMC estimators ensures SGD achieving $\tilde{\mathcal{O}}(\varepsilon^{-2})$ near-optimal convergence when loss function $\ell(x;\xi)$ is convex and bounded. However, due to the nested structure of proposed Sinkhorn DRO dual formulation~\eqref{dual2}, we find that the RT-MLMC estimator is not well-suited for our setting, as \eqref{eq: rt-mlmc} requires storing $q = \mathcal{O}(\log(\varepsilon^{-1}))$ samples of $\zeta$. This, in turn, simultaneously increases the number of calls to Algorithm~\ref{alg1} per iteration within Algorithm~\ref{alg2} to obtain corresponding $\eta_{x_t}^{\tilde{d}}(\zeta^{\iota_i})$. The nested structure of the problem limits the effectiveness of variance reduction typically offered by the multilevel Monte Carlo method and fails to significantly improve the convergence order.
    Furthermore, it may even introduce additional challenges in sampling and computational efficiency, particularly when applying estimator \eqref{eq: rt-mlmc} within the backpropagation process of deep neural networks. As we show later in Theorem~\ref{thm: convergenestsgd}, our proposed Algorithm~\ref{alg2}, which employs gradient estimator \eqref{gradient3} and requires only a single sample of $\zeta$ and $\xi$ (namely $B=1$) per iteration, achieves a convergence rate of $T=\mathcal{O}(\varepsilon^{-4})$ in standard nonconvex setting.}
\end{remark}
At each iteration, Algorithm~\ref{alg2} then uses SGD-algorithm with respect to $\xi$ to update $x_t$ iteratively.
For simplicity, in the following article, we use the simplified notation $\hat{g}_t^{B}$ to denote expression~\eqref{gradient3} when the dependence on $x_t$, $\zeta$ and $\xi_{B}$ is clear from the context.
\begin{algorithm}[ht]
\caption{\red{Nested SGD for optimizing $x$}} \label{alg2}
\begin{algorithmic}
\STATE {\bf Input:} Initialization $x_0$; number of iteration $T$; learning rate $\gamma_t$; batch size $B$.
\WHILE{$t<T$}
%\STATE \red{In proof, we need to assume samples $\zeta$ and $\{\xi\}_B$ are drawn independently with each other.}
\STATE Draw samples $ \zeta \sim \mathbb{P}$ and $\big \{ \xi \big\}_B \sim \nu$ with batch size $B$ independently. 
%\COMMENT {\red{Notice that one can reuse sampled ${\xi}_{\tilde{B}}$ from algorithm \ref{alg1}.}}
\STATE Apply Algorithm \ref{alg1} to compute estimator $\eta_{x}^{\tilde{d}}(\zeta)$.
\STATE Compute $\hat{g}^B_t$ via \eqref{gradient3}.
\STATE  Update $ x_{t+1}=x_t-\gamma_t \hat{g}_t^{B}$.
\ENDWHILE
\STATE {\bf Output:} ${x}_{\tilde{t}}$, where $\tilde{t}$ sampled uniformly from $\big\{0, \ldots, {T-1} \big\}$.
\end{algorithmic}
\end{algorithm}

Next, we analyze the convergence of Algorithm \ref{alg2}. We first study the smoothness property of the objective function $\mathbb{E}_{\zeta \sim \mathbb{P}}\bigr[\Psi_{\zeta}(x)\bigr]$. We note that the bi-variate function $\mathbb{E}_{\zeta \sim \mathbb{P}}\bigr[ \mathcal{L}_{\zeta}(x,\eta) \bigr]$ has been shown to satisfy a generalized-smooth condition \citep{Zhang2020Why,chen2023generalizedsmooth}. However, if $\eta $ is chosen to be the minimizer $\eta_x^*(\zeta)$, we show that the objective function satisfies the following directional smoothness property \citep{directional} (\red{See Appendix \ref{proof: directional smooth thm} for proof details}).
\begin{restatable}[Directional Smoothness]{lemma}{outersmooth}
\label{thm:directionsmooth} Let Assumption \ref{assum1} hold. For any $x$ and $x'$, we have
\begin{align}
   &\mathbb{E}_{\zeta\sim \mathbb{P}}\big\| \nabla \Psi_{\zeta}(x)- \nabla_1 \mathcal{L}_{\zeta}(x',\eta_{x}^*(\zeta))   \big \|^2 \leq K^2\big \| x-x' \big \|^2,\text{ where } K=G^2(\lambda \red{\beta})^{-1}M+L.
\end{align}
\end{restatable}
Recall that $\eta_{x}^{\tilde{d}}(\zeta)$ is inexact estimation of $\eta_x^*(\zeta)$ obtained from Algorithm \ref{alg1}. To analyze convergence of Algorithm \ref{alg2}, we obtain the upper bound of the second moment for $\hat{g}_t^B$, which is stated in next lemma (\red{See Appendix \ref{proof: outer gradient bounded moment thm} for proof details}). 
%We note that this second moment bound is different from that in Lemma \ref{boundedmoment1}, which is for the mini-batch gradient estimator ${g}_t^B$ in \eqref{gradient1} that involves $\eta_x^*(\zeta)$.
%We also develop the following bound on the second moment of the stochastic gradient estimator shown in \eqref{gradient1}. 
%\begin{restatable}[Bounded Second Moment]{lemma}{secondmomentboundone}
%\label{boundedmoment1}
%    Let Assumption \ref{assum2} hold. For the mini-batch stochastic gradient estimator $g^B_t$ used in Algorithm \ref{alg0}, its second moment is bounded by
%    \begin{align}
%    \mathbb{E}_{\zeta \sim \mathbb{P},\xi_{\tilde{B}}\sim \nu} \big[\| g_t^B \|^2 \big]
%    \leq R_1+ \big \| \nabla_{1} \mathbb{E}_{\zeta \sim \mathbb{P}}\bigr[\mathcal{L}_{\zeta}(x_t,\eta^*_{x_t}(\zeta))\bigr]\big \|^2,
%    \label{bounded var1}
%    \end{align}
%    where $R_1=\frac{8G^2+20G^2M^2(\lambda \red{\beta})^{-2}\sigma^2+20G^2M^2\red{\beta}^{-2}\delta^2}{B} + G^2M^2\red{\beta}^{-2} \delta^2$. \\
%\end{restatable}
\begin{restatable}[Second moment bound of $\hat{g}_t^{B}$]{lemma}{secondmomentboundthree}\label{thm: boundedmoment3}
    \red{The second moment of the mini-batch gradient estimator $\hat{g}_t^B$ with batch size $B$ defined in \eqref{gradient3} is upper bounded as follows    
    %Denote biased mini-batch stochastic estimator of  $\nabla \mathbb{E}_{\zeta\sim \mathbb{P}}\big[\Psi_{\zeta}(x)\big]$ evaluated at  $x, \eta_{x}^{\tilde{d}}(\zeta)$ using \eqref{gradient1} as $\hat{g}^B_t$, its second moment satisfies
    \begin{align}
        \mathbb{E}_{\zeta \sim \mathbb{P},\eta_{x}^{\tilde{d}}(\zeta),\xi_{B}\sim \nu} \big \| \hat{g}_t^B \big \|^2\leq 
    \frac{R_1+10\varepsilon^2}{B}+ \big \| \nabla_{1} \mathbb{E}_{\zeta \sim \mathbb{P},\eta_{x}^{\tilde{d}}(\zeta)}\bigr[\mathcal{L}_{\zeta}(x_t,\eta_{x_t}^{\tilde{d}}(\zeta))\bigr]\big \|^2,
    \label{bounded var1-1}
    \end{align}
    where $ R_1={8G^2+24G^2M^2(\lambda \red{\beta})^{-2}\sigma^2+24G^2M^2\red{\beta}^{-2}\delta^2} $.}
\end{restatable}
Based on the above lemma, we obtain the following convergence result of Nested-SGD for minimizing $\mathbb{E}_{\zeta \sim \mathbb{P}}[\Psi_{\zeta}(x)]$ (\red{See Appendix \ref{proof: nested SGD convergence thm} for proof details}).

\begin{restatable}[Convergence of Algorithm \ref{alg2}]{theorem}{mainconvergence}\label{thm: convergenestsgd}
    \red{Let Assumptions \ref{assum1} and \ref{assum2} hold. 
    Denote $\Delta = \mathbb{E}_{\zeta \sim \mathbb{P}}\big[\Psi_{\zeta}(x_0)\big] - \inf_x \mathbb{E}_{\zeta \sim \mathbb{P}} \big[\Psi_{\zeta}(x)\big]$, %$\widehat{\Delta}=\max_{\zeta}\big\{L_{\zeta}(x_t,\eta^0)-\Psi_{\zeta}(x_t)\big\}$.
    %\begin{align}
        %\sum_{t=0}^{T-1}\frac{\gamma_t}{4}\big \| \nabla \mathbb{E}_{\zeta \sim \mathbb{P}}\bigr[\Psi_{\zeta}(x_t)\bigr] \big \|^2 \leq \Delta
        %+\frac{(4+B)KG^2\varepsilon^2}{B} \sum_{t=0}^{T-1}\gamma_t^2  
        %+ \frac{KR_1}{2}\sum_{t=0}^{T-1} \gamma_t^2 + G^2\varepsilon^2\sum_{t=0}^{T-1}\gamma_t.
    %\end{align} 
    apply Algorithm \ref{alg2} to solve the outer objective in \eqref{dual2} using a constant learning rate $\gamma_t = \gamma = \min\{\frac{1}{24K}, \frac{\varepsilon^2}{2KR_1}\}$, and set the batch size $B = 1$. At each iteration, query Algorithm \ref{alg1} to obtain an estimator $\eta^{\tilde{d}}_x(\zeta)$ for sampled $\zeta$. Then, the output $x_{\tilde{t}}$ of Algorithm \ref{alg2} satisfies
    \begin{align}
        \mathbb{E}_{x_{\tilde{t}}} &\big\| \nabla \mathbb{E}_{\zeta \sim \mathbb{P}}\bigr[\Psi_{\zeta}(x_{x_{\tilde{t}}})\bigr] \big\|^2 \leq 7\varepsilon^2, 
    \end{align}
after $T\geq \max\{96\Delta K\varepsilon^{-2},8\Delta KR_1\varepsilon^{-4}\}=\mathcal{O}(\Delta K R_{1}\varepsilon^{-4})$ number of iterations. 
%Moreover, for any $\varepsilon$ satisfying $\varepsilon = \min \{1,\frac{1}{\sqrt{3+4G^2}} \}$,
%to achieve $\mathbb{E}_{x_{\tilde{t}}} \big \| \nabla \mathbb{E}_{\zeta \sim \mathbb{P}} \bigr[\Psi_{\zeta}(x_{\tilde{t}})\bigr]\big \|\leq \tilde{\varepsilon}$, the stochastic gradient oracle complexity for algorithm \ref{alg2} is $T=\mathcal{O}(\Delta K R_{1} G^2\tilde{\varepsilon}^{-4})$.
}
\end{restatable}
This theorem indicates Algorithm \ref{alg2} takes $\mathcal{O}(\varepsilon^{-4})$ iterations to obtain an $\mathcal{O}(\varepsilon)$-stationary point, which implies the inexact minimizer $\eta_{x}^{\tilde{d}}(\zeta)$ has minor effects on the worst-case sample complexity of Algorithm \ref{alg2} for solving proposed Sinkhorn DRO dual formulation~\eqref{dual2}. Next, we summarize the overall sample and iteration complexity, as well as the per-iteration sample and memory complexities of Algorithms~\ref{alg1} and~\ref{alg2} (See Appendix~\ref{Appendix: discussion of complexity} for proof details).
\begin{restatable}[{Complexity Bound for $\min_x \mathbb{E}_{\zeta\sim\mathbb{P}}[\Psi_{\zeta}(x)]$}]{corollary}{totalcomplexity}\label{thm: total complexity bound}
   \red{Let Assumptions \ref{assum1} and \ref{assum2} hold. Then, the Nested-SGD algorithm (Algorithm~\ref{alg2}) returns an $\varepsilon$-stationary point with a total sample complexity of $\mathcal{O}(\varepsilon^{-8})$ for sampling $\xi$ and $\zeta$. Furthermore, by setting the batch sizes $B, \tilde{B} \sim \Theta(\varepsilon^{-2})$, the total iteration complexity becomes $T \times D \sim \mathcal{O}(\varepsilon^{-4})$. At each iteration, Algorithms~\ref{alg1} and~\ref{alg2} incur memory complexities of $\mathcal{O}(1)$ and $\mathcal{O}(d)$, respectively.}
\end{restatable}
% Compare with contextual bi-level papers.
Compared with existing stochastic bilevel methods—e.g., \citet{hu2023contextual,SJwrightnonconvexbilevel,bileveldoubleconvex,huangmomentumbilevel}—that study bilevel optimization framework, our nested-SGD Algorithm \ref{alg2} attains a moderate overall sample-complexity bound. However, the tighter bounds in those works highly rely on additional regularity assumptions that are incompatible with proposed dual formulation~\eqref{dual2}. Specifically, \citet{hu2023contextual} establish a sample complexity of $\tilde{\mathcal{O}}(\varepsilon^{-6})$ for mini-batch SGD under a strongly convex lower-level problem. \citet{huangmomentumbilevel,bilevelnonconvexPL} obtain sample complexity of $\tilde{\mathcal{O}}(\varepsilon^{-4})$ and $\tilde{\mathcal{O}}(\varepsilon^{-6})$, respectively, by assuming the lower-level objective satisfies the Polyak–{\L}ojasiewicz (PL) condition. Finally, \citet{bileveldoubleconvex,SJwrightnonconvexbilevel} report sample complexity of $\mathcal{O}(\varepsilon^{-4})$ and $\mathcal{O}(\varepsilon^{-7})$ for (i) a jointly convex upper–lower structure and (ii) a non-convex setting satisfying small-error proximal error-bound (EB) condition. Except \citet{hu2023contextual}, none of these studies incorporate an in-context variable in the lower-level problem. Later through numerical experiments in section \ref{sec: numerical experiments}, We show that Algorithm~\ref{alg2} enables proposed dual formulation~\eqref{dual2} to achieve performance comparable to that of \citet{wang2023sinkhorn}, while requiring only a small number of queries to $\zeta$ and $\xi$ under the same iteration budget $T$ and a small inner-loop depth $D$. Moreover, it remains scalable to large-scale problems, such as training robust deep neural networks under distribution shifts.

%We note that at each iteration, Algorithm \ref{alg2} needs to query Algorithm \ref{alg1} with stochastic gradient oracle complexity $\mathcal{O}(\hat{\Delta}K' R_2/B_2 \red{\beta}^{-4})$.
\subsection{Proof Sketch of Theorem \ref{thm: convergenestsgd}}

Based on directional smooth property stated in Lemma \ref{thm:directionsmooth}, we obtain the similar descent lemma as $L$-smooth function along the direction $x_{t+1}-x_t$ when $\eta_{x_t}^{*}(\zeta)$ is fixed (See \eqref{eq: directional smooth descent lemma} in Appendix \ref{pfdescentlemma}).
By replacing $x_{t+1}-x_t$ with biased gradient estimator, $\hat{g}_{t}$, one can obtain
\begin{align}
        \mathbb{E}_{\zeta\sim \mathbb{P}}\big[\Psi_{\zeta}(x_{t+1})\big]
        \leq \mathbb{E}_{\zeta\sim \mathbb{P}}\big[\Psi_{\zeta}(x_t)\big] -\gamma_t \langle \nabla \mathbb{E}_{\zeta\sim \mathbb{P}}\big[\Psi_{\zeta}(x_t)\big],\hat{g}_t \rangle + \frac{K\gamma_t^2}{2} \big \| \hat{g}_t \big \|^2. \nonumber
\end{align}
When $\hat{g}^B_{t}$ is used, it induces randomness from $x_t$, $\zeta$, $\big\{\xi\big\}_{\tilde{B}}$ and $\eta_{x}^{\tilde{d}}(\zeta)$. Taking expectation conditioned on $x_t$ over $\zeta,\eta_{x}^{\tilde{d}}(\zeta), \big\{\xi \big\}_{B}$ on both sides of above inequality, we have
\begin{align}
     &\mathbb{E}_{\zeta\sim \mathbb{P},\eta_{x_t}^{\tilde{d}}(\zeta),\xi_{B}\sim\nu} \big[ \Psi_{\zeta}(x_{t+1})|x_t \big]
       \leq \mathbb{E}_{\zeta\sim \mathbb{P}, \eta_{x_t}^{\tilde{d}}(\zeta),\xi_{B}\sim\nu} \big[ \Psi_{\zeta}(x_{t})| x_t \big] \nonumber\\
        &\qquad \quad - \mathbb{E}_{\zeta\sim \mathbb{P},\eta_{x_t}^{\tilde{d}}(\zeta),\xi_{B}\sim \nu} \big [  \langle \nabla \mathbb{E}_{\zeta\sim \mathbb{P}}\big[\Psi_{\zeta}(x_t)\big],\gamma_t \hat{g}^B_t \rangle | x_t \big ]  +
        \underbrace{\mathbb{E}_{\zeta\sim \mathbb{P}, \eta_{x_t}^{\tilde{d}}(\zeta),\xi_{B}\sim \nu} \big [ \frac{K\gamma_t^2}{2} \big \| \hat{g}^B_t \big \|^2 | x_t \big] }_{\textit{second moment}}\nonumber.
\end{align}
For upper bounding term \textit{``second moment"}, one can utilize \eqref{bounded var1-1} stated in Lemma \ref{thm: boundedmoment3}, which leads to 
\begin{align}
    &\mathbb{E}_{\zeta\sim \mathbb{P},\eta_{x_t}^{\tilde{d}}(\zeta),\xi_{B}\sim\nu} \big[ \Psi_{\zeta}(x_{t+1})|x_t \big]\leq
    \mathbb{E}_{\zeta\sim \mathbb{P},\eta_{x_t}^{\tilde{d}}(\zeta),\xi_{B}\sim \nu } \big[ \Psi_{\zeta}(x_{t})|x_t \big]+ \frac{K\gamma_t^2(R_1+10\varepsilon^2)}{2}\nonumber \\ 
    &\qquad \underbrace{-\mathbb{E}_{\zeta\sim \mathbb{P},\eta_{x_t}^{\tilde{d}}(\zeta), \xi_{B}\sim\nu} \big [ \left \langle \nabla \mathbb{E}_{\zeta\sim \mathbb{P}}\big[\Psi_{\zeta}(x_t)\big],\gamma_t \hat{g}^B_t\right \rangle | x_t \big ]}_{\textit{Term 1}}+\underbrace{\frac{K\gamma_t^2}{2} \big \| \nabla \mathbb{E}_{\zeta\sim \mathbb{P},\eta_{x_t}^{\tilde{d}}(\zeta)}\big[\mathcal{L}_{\zeta}(x_t,\eta_{x_t}^{\tilde{d}}(\zeta))\big] \big \|^2}_{\textit{Term 2}} \nonumber
\end{align}
For \textit{``Term 1"}, we expand this term as $-\mathbb{E}_{\zeta\sim \mathbb{P},\eta_{x_t}^{\tilde{d}}(\zeta), \xi_{B}\sim\nu} \big [ \gamma_t \langle \nabla \mathbb{E}_{\zeta\sim \mathbb{P}}[\Psi_{\zeta}(x_t)], \hat{g}^B_t - g_t^B+g_t^B \rangle | x_t \big ]$. Since $\nabla \mathbb{E}_{\zeta\sim\mathbb{P}}\Psi(x_t)$ is independent from randomness induced by $\xi_B$ and $\eta_{x_t}^{\tilde{d}}(\zeta)$, we move expectation over $\xi_B\sim\nu$ and $\eta_{x_t}^{\tilde{d}}(\zeta)$ into inner product. For $-\mathbb{E}_{\zeta\sim \mathbb{P}} \big [ \gamma_t \langle \nabla \mathbb{E}_{\zeta\sim \mathbb{P}}[\Psi_{\zeta}(x_t)], \mathbb{E}_{\eta_{x_t}^{\tilde{d}}(\zeta), \xi_{B}\sim\nu}[\hat{g}^B_t - g_t^B] \rangle | x_t \big ]$, we first apply Cauchy-Scharwarz inequality and then use condition \eqref{optimalrelation} stated in Theorem \ref{thm: grad_error_bound} to obtain upper bound $\gamma_t \varepsilon \mathbb{E}_{x_t}\big \|\nabla \mathbb{E}_{\zeta\sim \mathbb{P}}\big[\Psi_{\zeta}(x_t)\big] \big \|$. For $-\mathbb{E}_{\zeta\sim \mathbb{P},\eta_{x_t}^{\tilde{d}}(\zeta), \xi_{B}\sim\nu} \big [ \gamma_t \langle \nabla \mathbb{E}_{\zeta\sim \mathbb{P}}[\Psi_{\zeta}(x_t)],g_t^B \rangle | x_t \big ]$, it is equivalent to rewrite as $-\gamma_t\|\nabla \mathbb{E}_{\zeta\sim\mathbb{P}}[\Psi_{\zeta}(x_t)] \|^2$.

Similarly, For \textit{``Term2"}, by plus and minus an additional term $\nabla\mathbb{E}_{\zeta\sim\mathbb{P}}[\Psi_{\zeta}(x_t)]$ in squared norm, utilizing inequality $(a+b)^2\leq 2a^2+2b^2$ and condition \eqref{eq: induced from grad_error_bound} stated in Theorem \ref{thm: grad_error_bound},
we can upper bound ``Term2" by ${K\gamma_t^2}\varepsilon^2+K\gamma_t^2\|\nabla \mathbb{E}_{\zeta\sim\mathbb{P}}[\Psi_{\zeta}(x_t)] \|^2$. Re-arranging above inequality and applying fact $\gamma_t = \gamma=\min \{\frac{1}{24K},\frac{\varepsilon^2}{2KR_1} \}<\frac{1}{2K}$, we have
\begin{align}
     &\frac{\gamma_t}{2}\mathbb{E}_{x_t}\Big(\big \| \nabla \mathbb{E}_{\zeta\sim \mathbb{P}}\big[\Psi_{\zeta}(x_t)\big] \big \|^2 -2\varepsilon\big \| \nabla \mathbb{E}_{\zeta\sim \mathbb{P}}\big[\Psi_{\zeta}(x_t)\big] \big \|\Big) \nonumber\\
     &\leq \mathbb{E}_{x_t,\zeta\sim \mathbb{P},\eta_{x_t}^{\tilde{d}}(\zeta),\xi_{B}\sim\nu}\big[\Psi_{\zeta}(x_t) - \Psi_{\zeta}(x_{t+1})\big]  + \frac{KR_{1}\gamma_t^2}{2} + 6K\varepsilon^2\gamma_t^2. \nonumber
\end{align}
Summing above inequality from $0$ to $T-1$, applying $\gamma_t=\gamma = \min \{\frac{1}{24K},\frac{\varepsilon^2}{2KR_1} \}$ and\\ $ T\geq \max\{96\Delta K\varepsilon^{-2},8\Delta KR_1\varepsilon^{-4}\}$, we further conclude
\begin{align}
      \mathbb{E}_{x_{\tilde{t}}}\big[ \| \nabla \mathbb{E}_{\zeta \sim \mathbb{P}}\big[\Psi_{\zeta}(x_{\tilde{t}}) \big] \big\|^2\big] \nonumber
        \leq \frac{4\Delta}{T\gamma}+\frac{24K \varepsilon^2\gamma^2 T}{T\gamma} + \frac{2KR_{1}\gamma^2T}{T\gamma}+\frac{4  \varepsilon^2 T\gamma}{T\gamma }\leq 7\varepsilon^2.
\end{align}

%\begin{remark}
%    comparison with wang et al.
%\end{remark}
\section{Experiments}\label{sec: numerical experiments}

In this section, we evaluate the performance of proposed Sinkhorn DRO dual formulation~\eqref{dual2} in comparison to other baselines, including (constrained) Sinkhorn DRO dual formulation \eqref{baseline SinkhornDRO} from \citet{wang2023sinkhorn}, regularized $f$-divergence DRO from \citet{jin2021nonconvex, duchi2020learning}, and empirical risk minimization \citep{vapnikERM} under distribution shifts\footnote{Code available at: \url{https://github.com/ynyang94/GeneralSinkhorn-Regularized-DRO}}. To elaborate, we train logistic regression and LeNet \citep{LeNet} for classification tasks on real-world datasets, where we simulate distribution shifts by applying adversarial attacks on test dataset. \red{To align with assumption \ref{assum1}, for proposed Sinkhorn DRO dual formulation~\eqref{dual2}, we choose $f^*$ to be the conjugate dual of the $\chi^2$-divergence, which satisfies the $M$-smoothness property. We unify the training procedure across all formulations using vanilla SGD \citep{ghadimi2013stochastic}, implemented via PyTorch's autograd toolbox \citep{pytorchautograd}, \textit{without} employing additional heuristics such as random shuffling, learning rate scheduling, or weight decay. Test results are reported using the model parameters obtained at the last epoch, rather than the uniform averaging over iterates used in theoretical analysis. Additionally, for gradient estimator \eqref{gradient3}, we evaluate the expression using the sampled $\xi$ and $\eta_x^{\tilde{d}}(\zeta)$ obtained from Algorithm~\ref{alg1} to reduce computational overhead. To make a fair comparison between proposed Sinkhorn DRO dual formulation~\eqref{dual2} and dual formulation from \citet{wang2023sinkhorn}, we generate sample $\xi$ following same distribution. Due to page limitation, we refer readers to check Appendix \ref{Appendix: logistic regression experiments} and \ref{Appendix: details for LeNet} for more details on model initialization and algorithm hyper-parameter settings.
We also conduct linear regression experiment over synthetic data and ablation studies of proposed Sinkhorn DRO dual formulation~\eqref{dual2}, where we present the corresponding results in Appendix \ref{Appendix: linear regression experiments}, \ref{Appendix: ablation study} respectively. 
All the previously mentioned experiments were conducted on a PC computer with $32$GB memory, $24$ cores CPU running
Python 3.8.}
%The experiments were conducted on regression over synthetic data and classification over CIFAR10.

\subsection{Logistic Regression on CIFAR-10}

\begin{table*}[htbp]
\begin{center}
\resizebox{0.75\columnwidth}{!}{%
\begin{tabular}{ |c|c c c|c| }
 \hline
 FGSM (Logistic)  & Sinkhorn $\text{DRO}^1$ & f-DRO & ERM & Sinkhorn $\text{DRO}^2$   \\ 
 \hline 
 $\epsilon_{\text{FGSM}}=0.00$ & \textbf{77.07}\% & 77.00\% & 73.26\% & 75.27\%\\
 \hline
 $\epsilon_{\text{FGSM}}=0.01$ & 65.60\% & 63.80\% & 60.73\% & \textbf{67.13\%} \\
 \hline 
 $\epsilon_{\text{FGSM}}=0.02$ & 54.0\% & 46.07\% & 45.67\% & \textbf{58.13\%} \\ 
 \hline 
 $\ell_{\infty}$-PGD (Logistic) & Sinkhorn $\text{DRO}^1$ & f-DRO & ERM & Sinkhorn $\text{DRO}^2$ \\
 \hline 
 $\epsilon_{\text{PGD}}=0.01, \textit{iter}=20$ & 66.93\% & 66.0\% & 62.2\% & \textbf{68.2\%}\\
 \hline
  $\epsilon_{\text{PGD}}=0.02, \textit{iter}=20$ & 57.93\% & 53.27\% & 51.0\% & \textbf{60.47}\% \\
 \hline
 $\epsilon_{\text{PGD}}=0.03, \textit{iter}=20$ & 45.47\% & 37.93\% & 38.07\% & \textbf{50.60\%} \\
 \hline 
 
\end{tabular}}
\end{center}
\caption{ Test classification accuracy of logistic regression.}
\label{classification accuracy logistic}
\end{table*}

In this section, we apply the proposed Sinkhorn DRO dual formulation~\eqref{dual2} and other baselines on logistic regression over CIFAR-10 \citep{cifar10}, and test the classification accuracy on adversarial samples generated by the fast gradient sign method (FGSM) \citep{fgsm} and $\ell_{\infty}$-PGD \citep{pgdattack} attacks utilizing model parameters obtained at \textit{last} epoch. Table \ref{classification accuracy logistic} reports the test accuracy under different perturbation magnitudes, where Sinkhorn $\text{DRO}^1$ refers to formulations \eqref{baseline SinkhornDRO} from \citet{wang2023sinkhorn} and Sinkhorn $\text{DRO}^2$ refers to proposed Sinkhorn DRO dual formulation~\eqref{dual2}. The corresponding test loss curves are plotted in Figure \ref{logistic classification}. We found that when the test data are clean, the proposed Sinkhorn DRO dual formulation \eqref{dual2} achieves performance comparable to other baselines. However, as attack level increases, the model obtained via proposed Sinkhorn DRO formulation \eqref{dual2} is more robust than others, which demonstrates the advantage and effectiveness of Sinkhorn DRO and proposed Nested SGD Algorithm~\ref{alg2}.

\subsection{LeNet Classification on MNIST}

\begin{table*}[htbp]
\begin{center}
\resizebox{0.75\columnwidth}{!}{%
\begin{tabular}{ |c|c c c |c| } 
 \hline
 FGSM (LeNet) & Sinkhorn $\text{DRO}^1$ & f-DRO & ERM & Sinkhorn $\text{DRO}^2$  \\
 \hline
  $\epsilon_{\text{FGSM}}=0.00$ & 95.89\% & 94.60\% & 95.50\% & \textbf{96.80}\% \\
 \hline
 $\epsilon_{\text{FGSM}}=0.02$ & 83.10\% & 60.30\% & 79.20\% & \textbf{89.80}\% \\
 \hline
 $\epsilon_{\text{FGSM}}=0.05$ & 52.50\% & 16.90\% & 40.0\% &  \textbf{65.60}\% \\
 \hline
 PGD $\ell_{\infty}$ (LeNet) & Sinkhorn $\text{DRO}^1$ & f-DRO & ERM & Sinkhorn $\text{DRO}^2$ \\
 \hline
 $\epsilon_{\text{PGD}_\infty}=0.01, \textit{iter}=20$ & 91.50\% & 84.0\% & 90.80\% & \textbf{94.30}\%\\
 \hline
 $\epsilon_{\text{PGD}_\infty}=0.02, \textit{iter}=20$ & 85.10\% & {62.30}\% & 81.20\% & \textbf{91.30}\%\\
 \hline
 $\epsilon_{\text{PGD}_\infty}=0.05, \textit{iter}=20$ & 46.60\% & 7.90\% & 34.30\% & \textbf{66.20}\%\\
 \hline
 PGD $\ell_{2}$ (LeNet) & Sinkhorn $\text{DRO}^1$ & f-DRO & ERM  & Sinkhorn $\text{DRO}^2$\\ 
 \hline
 $\epsilon_{\text{PGD}_2}=0.5,\textit{iter}=30$ & 87.10\% & 71.39\% & 85.30\% & \textbf{92.50}\% \\
 \hline
 $\epsilon_{\text{PGD}_2}=0.8, \textit{iter}=30$ & 74.20\% & 46.20\% & 73.50\% & \textbf{87.10}\% \\
 \hline
 $\epsilon_{\text{PGD}_2}=1.2,\textit{iter}=30$ & 55.50\% & 18.90\% & 51.40\% & \textbf{74.00}\% \\
 \hline
 MIM (LeNet) & Sinkhorn $\text{DRO}^1$ & f-DRO & ERM & Sinkhorn $\text{DRO}^2$ \\
 \hline
 $\epsilon_{\text{MIM}}=0.01, \textit{iter}=30$ & \textbf{83.20\%} & 42.50\% & 20.20\% & 80.60\% \\
 \hline
 $\epsilon_{\text{MIM}}=0.02, \textit{iter}=30$ & \textbf{76.60\%} & 31.90\% & 15.20\% & 76.20\%\\
 \hline
 $\epsilon_{\text{MIM}}=0.05, \textit{iter}=30$ & 52.10\% & 10.60\% & 4.80\% & \textbf{57.99}\% \\
 \hline
\end{tabular}}
\end{center}
\caption{Test classification accuracy of LeNet under different adversarial attack methods.}
\label{classification accuracy}
\end{table*}

In this section, we apply proposed Sinkhorn DRO dual formulation~\eqref{dual2} and other baselines to train a LeNet \citep{LeNet} over MNIST \citep{mnist}, and test the classification accuracy on adversarial samples generated by FGSM \citep{fgsm}, $\ell_{\infty}, \ell_{2}$-PGD \citep{pgdattack} and momentum iterative method (MIM) \citep{mim} attacks utilizing model parameters obtained at \textit{last} epoch. Table \ref{classification accuracy} reports the test accuracies under different perturbation magnitudes and the corresponding test loss curves are plotted in Figure \ref{LeNet test loss 1}, \ref{LeNet test loss 2}, \ref{LeNet test loss 3} and \ref{LeNet test loss 4}. Specially, we found $f$-DRO is vulnerable against FGSM and PGD attacks. This might be due to the generalized-smoothness property of $f$-DRO objective \citep{jin2021nonconvex,chen2023generalizedsmooth}, which makes it hard to optimize using vanilla SGD. As for the proposed Sinkhorn DRO dual formulation \eqref{dual2}, we find that it achieves higher classification accuracy than the Sinkhorn DRO dual formulation \eqref{baseline SinkhornDRO} from \citet{wang2023sinkhorn} across most attack magnitudes, and exhibits a smaller drop in accuracy as the attack strength increases. This demonstrates the effectiveness of our proposed Sinkhorn DRO dual formulation \eqref{dual2} and supports the validity of our algorithmic analysis.

\section{Conclusion}
In this paper, we investigate generalized Sinkhorn distance-regularized distributionally robust optimization. By deriving a new dual formulation with strong duality guarantee, we show that the resultant Sinkhorn DRO problem has nested stochastic optimization structure, which enables us to design a Nested SGD algorithm with convergence guarantee under mild assumptions. Numerical studies demonstrate that our Sinkhorn DRO formulation is applicable to large-scale problems and can attain stronger robustness against distribution shifts through multiple datasets and tasks.
%\section*{Impact Statement}
%This paper presents work whose goal is to advance the field of robust machine learning. We hope our proposed Sinkhorn DRO formulation can inspire researchers rethinking about the underlying mechanisms of adversarial training, and encourage community to explore new algorithms and applications. There are potential societal applications 
%of our work, specially, we hope DRO framework could benefit underrepresented people in social equity research and provide framework of generating robust decisions for any potential social extreme events.


\section{Acknowledgements}
% List of funding resources, i.e., Funding in direct support of this work: NSF grant XXX, GPUs donated by YYY, scholarship by Company ZZZ.
Yufeng Yang and Yi Zhou’s work is supported by the National Science Foundation under grants DMS-2134223, ECCS-2237830. Zhaosong Lu's work is partially supported by the Office of Naval Research under grant N00014-24-1-2702, the Air Force Office of Scientific Research under grant FA9550-24-1-0343,  and the National Science Foundation under grant IIS-2211491.

%%%%%%%%%%%%%%%%%%%%%%%%%%%%%%%%%%%%%%%%%%%%%%%%%%%%%%%%%%%%%%%%%%%%%%%%%%%%%%%
%%%%%%%%%%%%%%%%%%%%%%%%%%%%%%%%%%%%%%%%%%%%%%%%%%%%%%%%%%%%%%%%%%%%%%%%%%%%%%%
% APPENDIX
%%%%%%%%%%%%%%%%%%%%%%%%%%%%%%%%%%%%%%%%%%%%%%%%%%%%%%%%%%%%%%%%%%%%%%%%%%%%%%%
%%%%%%%%%%%%%%%%%%%%%%%%%%%%%%%%%%%%%%%%%%%%%%%%%%%%%%%%%%%%%%%%%%%%%%%%%%%%%%%

\newpage
\bibliography{main}

\newpage

\appendix
\section*{Appendix}
\startcontents[app]        % start collecting appendix entries

\addcontentsline{toc}{section}{Table of Contents}

\noindent\rule{\textwidth}{0.2pt}\par\vspace{0.4\baselineskip}
% the mini-TOC itself
\vspace{-10pt}
\printcontents[app]{l}{1}{\setcounter{tocdepth}{2}}
\vspace{-5pt}
% bottom rule
\vspace{0.3\baselineskip}\noindent\rule{\textwidth}{0.2pt}

\renewcommand{\thelemma}{A.\arabic{lemma}} % or \thetheorem if shared counter
\setcounter{theorem}{0}

\section{Regression over synthetic data}\label{Appendix: linear regression experiments}

Through this section, we use synthetic training and test data. We generate the input samples with $3k$ measurements and dimension $d=10$ from a multivariate normal distribution, where the mean vector and covariance matrix are $0.5\mathbf{e}$ and $0.1\mathbf{I}$, respectively. Ground-truth model parameters $x^*$ are sampled from $\mathcal{N}(0,9e^{-2})$, and the corresponding output data $\zeta_{\text{output}}\in \mathbf{R}^{3k\times 10}$ follows the rule $\zeta_{\text{output}}= \zeta_{\text{train}} \cdot x^*+\epsilon_{\text{noise}}$, where $\epsilon_{\text{noise}} \sim \mathcal{N}(0,2.5e^{-2})$. For synthetic test data, we generate $500$ measurements following the same way as training data, we normalize all the data, apply Gaussian and Laplacian attack over test data to compare their performances. For primal parameters, we initialize them as $x_0\sim \mathcal{N}(x^*+5e^{-2},1e^{-2})$. For dual variable $\eta$ used in proposed Sinkhorn DRO dual formulation \eqref{dual2} and $f$-DRO, we initialize them as $\eta_0\in \mathbf{R}^{3k}\sim \mathcal{N}(5,2.25)$, $\eta_0=0.8$ respectively. 
%For Sinkhorn distance hyper-parameter $\beta$, we set it as $1.0$ for both our proposed formulation \eqref{dual2} and baseline Sinkhorn DRO formulation \eqref{baseline SinkhornDRO}. For regularization parameter $\lambda$, we set it as $0.8$ for all formulations. 

We fine-tuned the hyper-parameters for all models. The detailed formulation and algorithm settings are as follows. The loss function $\ell(.)$ is set to quadratic loss through all formulations. For proposed Sinkhorn DRO dual formulation \eqref{dual2} and Sinkhorn DRO dual formulation \eqref{baseline SinkhornDRO} from \citet{wang2023sinkhorn}, we set the reference measure $\nu$ as Gaussian measure following $\xi_{\text{train}}\sim\mathcal{N}(\zeta_{\text{train}},4e^{-2})$ for every $\zeta$. The cost metric and $f^*$ are set as $c(\zeta,\xi)=||\zeta-\xi||_2^2$ and $f^*(t)=\frac{1}{4}(t+2)^2_{+}-1$, which corresponds to the dual function of $\chi^2$-divergence. For regularization parameter $\lambda$ and $\red{\beta}$ used in generalized Sinkhorn distance (see Definition \ref{gSinkhorn}), we set them as $\lambda = 0.8, \red{\beta}=1.0$. We trained all formulations using vanilla stochastic gradient descent (SGD) \citep{ghadimi2013stochastic}. For $f$-DRO and ERM, we set the learning rates as $5e^{-4}$, $1e^{-3}$ respectively, and we optimize primal and dual variable of $f$-DRO in parallel following conclusion drawn from \citet{jin2021nonconvex}. For Sinkhorn DRO formulation \eqref{baseline SinkhornDRO} from \citet{wang2023sinkhorn}, we set the learning rate as $1e^{-3}$ and subsample $\xi_{\tilde{B}}$ with $\tilde{B}=8$ at each iteration. For proposed Sinkhorn DRO formulation \eqref{dual2}, we subsample $\xi_{\tilde{B}}$ with $\tilde{B}=8$, run algorithm \ref{alg1} with $5$ steps to minimize \eqref{eq: inner problem} at each iteration, and set the learning rates for algorithm \ref{alg2} and \ref{alg1} as $5e^{-2}$, $8e^{-2}$ respectively. For all algorithms (except inner SGD algorithm \ref{alg1}), we set batch size for sub-sampled $\zeta$ as $32$ and ran SGD for $80$ epochs. Figure \ref{LinearRegression} plots the training and test loss according to recorded checkpoints every $8$ epochs.

We evaluate our proposed Sinkhorn DRO dual formulation~\eqref{dual2}, $f$-DRO and ERM on regression task over synthetic data. We plot test (quadratic) loss on the test data in log-log scale at Figure \ref{LinearRegression}, where the left shows the loss value obtained from the test data under gaussian attack; the right shows loss value obtained from the test data under Laplacian attack. $\textit{SDRO}1$ refers to the Sinkhorn DRO dual formulation from \citet{wang2023sinkhorn} and $\textit{SDRO}2$ refers to proposed Sinkhorn DRO dual formulation \eqref{dual2}. Different marks represent different perturbation magnitudes $\mathrm{p}$. As we can see, our proposed Sinkhorn DRO dual formulation \eqref{dual2} attains comparable performance under distribution shifts with Sinkhorn DRO dual formulation from \citep{wang2023sinkhorn}. Additionally, we found $f$-DRO difficult to optimize using SGD with sample-average approximation. This observation is consistent with the findings of \citet{jin2021nonconvex, chen2023generalizedsmooth}, where the $f$-DRO formulation satisfies a generalized smoothness condition and requires advanced optimization algorithms \citep{zhangrevistingDRO, cutkoskynsgdm} to ensure convergence. 
\begin{figure}[!htbp]
    \centering
    \subfigure[Gaussian attack]{%
    \includegraphics[width=0.45\textwidth]{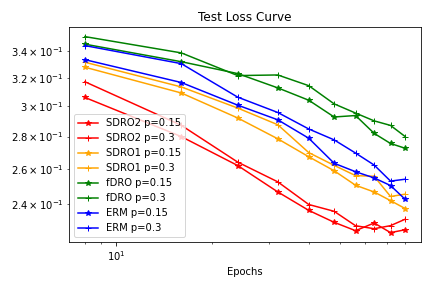}%
    \label{fig:gaussian}
}
    %\hspace{0.05\textwidth}
    \hfill
    \subfigure[Laplacian attack]{
        \includegraphics[width=0.45\linewidth]{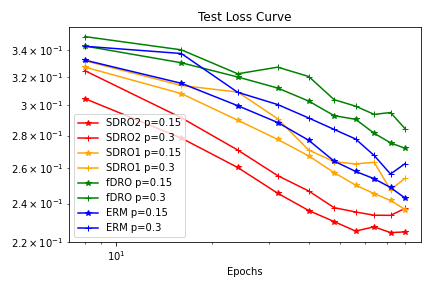}
   }

    \caption{Test performance of linear regression under Gaussian and Laplacian attack. }
    \label{LinearRegression}
\end{figure}

\section{Detailed settings for training logistic regression over compressed CIFAR-10}\label{Appendix: logistic regression experiments}

Through this section, we use CIFAR-10 \citep{cifar10} as our train and test data. We preprocess the dataset by resizing images, normalizing and utilizing pre-trained ResNet-50  \citep{ResNet} over ImageNet \citep{imagenet} to compress each image into a vector with feature dimension $d=250$. For test data, we subsampled $1500$ samples from compressed CIFAR-10 test data, generating adversarial examples utilizing model parameters obtained at last epoch to evaluate test performances through all methods. For proposed Sinkhorn DRO dual formulation~\eqref{dual2}, we initialize the primal and dual parameters as $x_0\sim\mathcal{N}(0,4e^{-2})$, $\eta_0\in \mathbf{R}^{50k}\sim \mathcal{N}(1,1e^{-2})$ respectively. For Sinkhorn DRO dual formulation from \citet{wang2023sinkhorn}, $f$-DRO and ERM, we adopt same initialization for primal parameters and set the dual variable $\eta$ for $f$-DRO as $\eta_0=1.5$.

\begin{figure*}[!htbp]
    \centering
    % First row of images
    \subfigure[$\epsilon_{\text{FGSM}}=0.00$]{
        \includegraphics[width=0.31\textwidth]{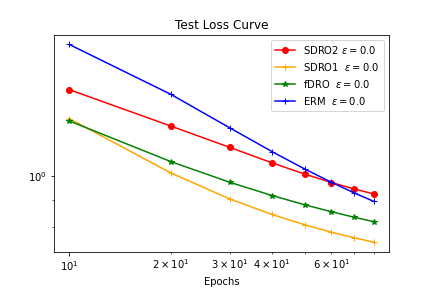}
    }
    \hfill
    \subfigure[$\epsilon_{\text{FGSM}}=0.01$]{
        \includegraphics[width=0.31\textwidth]{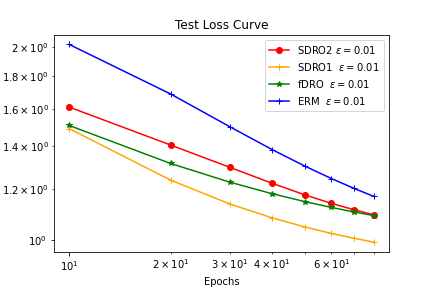}
    }
    \hfill
    \subfigure[$\epsilon_{\text{FGSM}}=0.02$]{
    
        \includegraphics[width=0.31\textwidth]{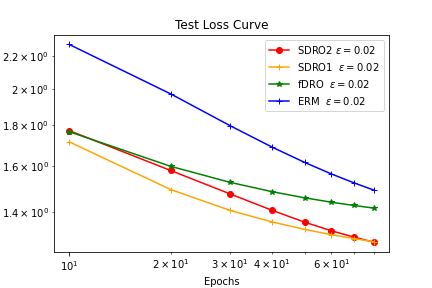}
    }
    %\vspace{0.1cm} % Adjust vertical spacing between rows if needed
    
    % Second row of images
    %\centering
    % First row of images
    \vspace{0.5cm} %
    \subfigure[$\epsilon_{\text{PGD}}=0.01$]{
        
        \includegraphics[width=0.31\textwidth]{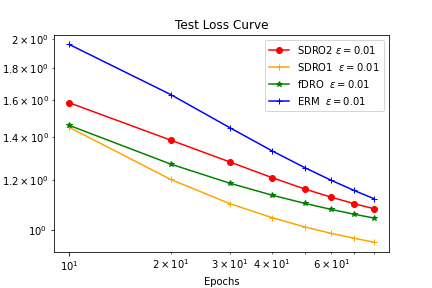}
    }
    \hfill
    \subfigure[$\epsilon_{\text{PGD}}=0.02$]{
        \includegraphics[width=0.31\textwidth]{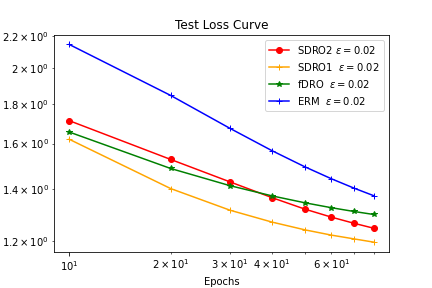}
    }
    \hfill
    \subfigure[$\epsilon_{\text{PGD}}=0.03$]{
        \includegraphics[width=0.31\textwidth]{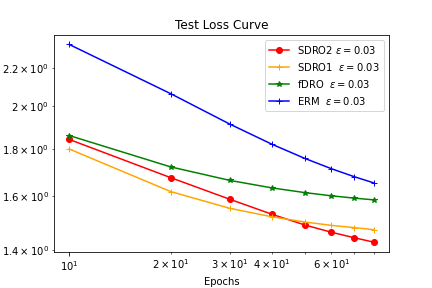}
    }
    \caption{Logistic Regression Test Loss under FGSM (top) and PGD (bottom) attack}
    \label{logistic classification}
\end{figure*} 

We fine-tuned all hyper-parameters for each baseline methods. The detailed formulation and algorithm setting are as follows. The loss function $\ell(.)$ is set to cross-entropy (CE) loss through all formulations. For proposed Sinkhorn DRO dual formulation~\eqref{dual2} and formulation \eqref{baseline SinkhornDRO} from \cite{wang2023sinkhorn}, we set reference measure $\nu$ as Gaussian measure following $\xi_{\text{train}}\sim \mathcal{N}(\zeta_{\text{train}},4e^{-2})$ and keep $c(\cdot,\cdot)$, $f^*(\cdot)$ to be $\ell_2$-norm and conjugate dual of $\chi^2$-divergence.
We trained all formulations using vanilla stochastic gradient descent (SGD) \citep{ghadimi2013stochastic}. For $f$-DRO and ERM, we set their learning rates as $8e^{-2}$, $3e^{-2}$ respectively, and we optimize the primal and dual variable of $f$-DRO in parallel following conclusion drawn from \citet{jin2021nonconvex}. For Sinkhorn DRO dual formulation \eqref{baseline SinkhornDRO} from \citet{wang2023sinkhorn}, we subsample $\xi_{\tilde{B}}$ with $\tilde{B}=2$ and set learning rate as $8e^{-2}$. For proposed Sinkhorn DRO dual formulation~\eqref{dual2}, we also subsample $\xi_{\tilde{B}}$ with $\tilde{B}=2$, run algorithm \ref{alg1} with $5$ steps to minimize inner objective \eqref{eq: inner problem} at each iteration and we set the learning rates for algorithm \ref{alg2} and \ref{alg1} as $8e^{-2}$,$1e^{-1}$ respectively. For all SGD algorithms (except inner Algorithm \ref{alg1}), we set batch size for sub-sampled $\zeta$ as $64$, ran algorithms for $80$ epochs and plot the test CE loss every $10$ epochs. Figure \ref{logistic classification} plots the test loss in log-log scale, where the first row represents the CE loss of test data under FGSM attack and the second row plots the CE loss of test data under $\ell_{\infty}$-PGD \citep{pgdattack}, where we set $\ell_{\infty}$-PGD attack iterations to be $20$ and step size $\alpha= \red{\epsilon_{\text{PGD}}}/4$ through all perturbation magnitudes.

Combined test accuracy reported in Table \ref{classification accuracy}, although every model's accuracy is affected when varying perturbation, we found our proposed Sinkhorn DRO dual formulation~\eqref{dual2} achieves highest test classification accuracies across most scenarios, and exhibits smallest accuracy drop the attack strength increases, which demonstrates the effectiveness of proposed Sinkhorn formulation \eqref{dual2} and Nested-SGD algorithm (Algorithm~\ref{alg2}).

\section{Detailed settings for training LeNet over MNIST data}\label{Appendix: details for LeNet}
Through this section, we use MNIST \citep{mnist} as our train and test data. We preprocess them by resizing images into $ 32\times 32$, normalizing them with mean and standard derivation, all equal to $0.5$. For the test data, we randomly subsampled $1000$ samples from MNIST test data, and generate adversarial test samples utilizing model parameters obtained at the last epoch. For proposed Sinkhorn DRO dual formulation \eqref{dual2} and other baseline methods, we initialize the primal parameters using kaiming initialization \citep{kaiminginitialization}. For the dual parameters utilized in proposed Sinkhorn DRO dual formulation \eqref{dual2} and $f$-DRO, we initialize them as $\eta_0\in \mathbf{R}^{60k}\sim \mathcal{N}(0.5,1e^{-2})$ and $\eta_0=1.0$ respectively.
\begin{figure*}[!htbp]
    \centering
    \subfigure[$\epsilon_{\text{FGSM}}=0.00$]{%
        \includegraphics[width=0.32\textwidth]{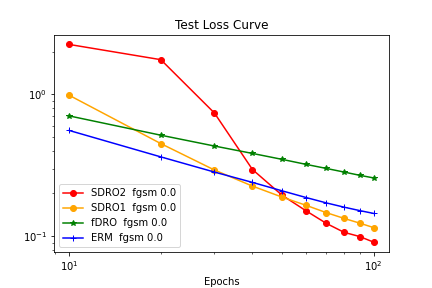}
    }
    \hfill
    \subfigure[$\epsilon_{\text{FGSM}}=0.02$]{%
        \includegraphics[width=0.32\textwidth]{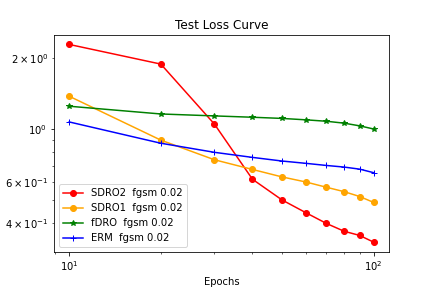}
    }
    \hfill
    \subfigure[$\epsilon_{\text{FGSM}}=0.05$]{%
        \includegraphics[width=0.32\textwidth]{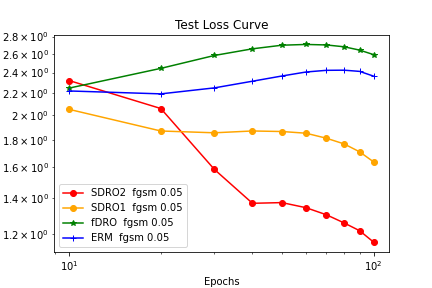}
    }

    \caption{Test loss curves of LeNet under FGSM attacks with different perturbation levels.}
    \label{LeNet test loss 1}
\end{figure*}
\begin{figure*}[!htbp]
    \centering
    \subfigure[$\epsilon_{\text{PGD}}=0.01$]{%
        \includegraphics[width=0.32\textwidth]{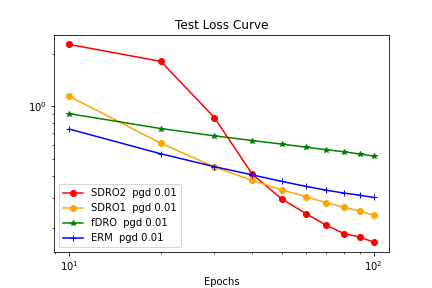}
    }
    \hfill
    \subfigure[$\epsilon_{\text{PGD}}=0.02$]{%
        \includegraphics[width=0.32\textwidth]{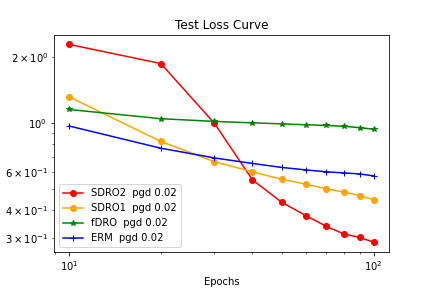}
    }
    \hfill
    \subfigure[$\epsilon_{\text{PGD}}=0.05$]{%
        \includegraphics[width=0.32\textwidth]{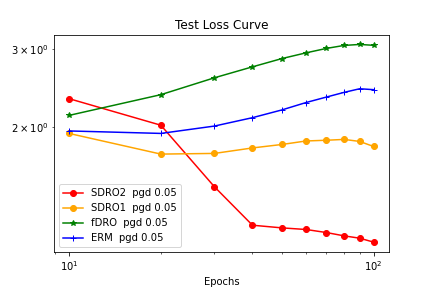}
    }

    \caption{Test loss curves of LeNet under $\ell_\infty$-PGD attacks with different perturbation levels.}
    \label{LeNet test loss 2}
\end{figure*}

We fine-tuned all hyper-parameters for each method. The detailed formulation and algorithm settings are as follows. The loss function $\ell(.)$ is set to cross-entropy (CE) loss through all formulations. For proposed Sinkhorn DRO dual formulation \eqref{dual2} and formulation \eqref{baseline SinkhornDRO} from \cite{wang2023sinkhorn}, we set their reference measure $\nu$ as Gaussian measure following $\xi_{\text{train}}\sim \mathcal{N}(\zeta_{\text{train}},2.25e^{-2})$ and keep $c(\cdot,\cdot)$, $f^*(\cdot)$ to be $\ell_2$-norm, conjugate dual of $\chi^2$-divergence. For regularization parameter $\lambda$, $\red{\beta}$ used in objective formulation and Sinkhorn distance, we set them as $\lambda = 0.5$, $\red{\beta} = 0.8$ respectively. We trained all formulations using vanilla stochastic gradient descent (SGD) \citep{ghadimi2013stochastic}. For $f$-DRO and ERM, we set their learning rate to be $1e^{-3}$, optimize primal and dual variables of $f$ in parallel according to conclusion drawn from \citet{jin2021nonconvex}. For Sinkhorn DRO dual formulation \eqref{baseline SinkhornDRO} from \citet{wang2023sinkhorn}, we subsample $\xi_{\tilde{B}}$ with $\tilde{B}=5$ and set learning rate as $1e^{-3}$.
For proposed Sinkhorn DRO dual formulation \eqref{dual2}, we subsample $\xi_{\tilde{B}}$ with $\tilde{B}=4$, run algorithm \ref{alg1} with $20$ steps for minimizing inner objective \eqref{eq: inner problem}, and set the learning rates for algorithm \ref{alg2} and \ref{alg1} as $5e^{-3}$ and$1e^{-1}$ respectively. We set the batch size for sub-sampled $\zeta$ as $128$ and ran SGD for $100$ epochs and record the loss every $10$ epochs. 

\begin{figure*}[!htbp]
    \centering
    \subfigure[$\epsilon_{\text{PGD}2}=0.5$]{%
        \includegraphics[width=0.32\textwidth]{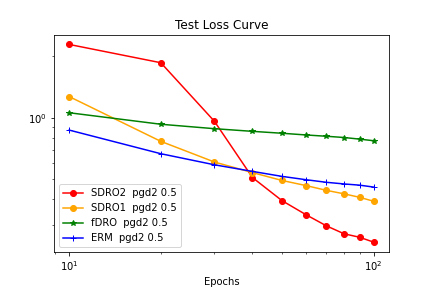}
    }
    \hfill
    \subfigure[$\epsilon_{\text{PGD}2}=0.8$]{%
        \includegraphics[width=0.32\textwidth]{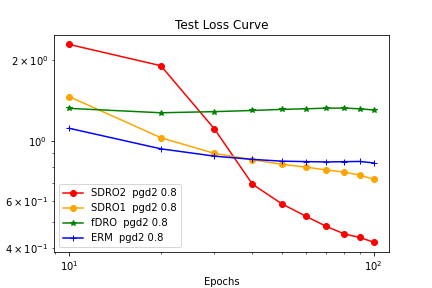}
    }
    \hfill
    \subfigure[$\epsilon_{\text{PGD}2}=1.2$]{%
        \includegraphics[width=0.32\textwidth]{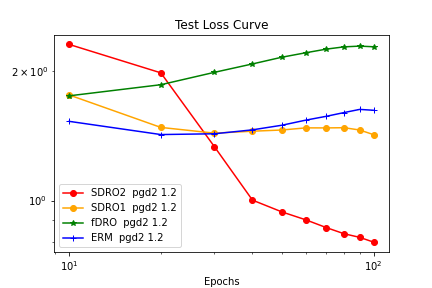}
    }

    \caption{Test loss curves of LeNet under $\ell_2$-PGD attacks with different perturbation levels.}
    \label{LeNet test loss 3}
\end{figure*}
\begin{figure*}[!htbp]
    \centering
    \subfigure[$\epsilon_{\text{MIM}}=0.01$]{%
        \includegraphics[width=0.32\textwidth]{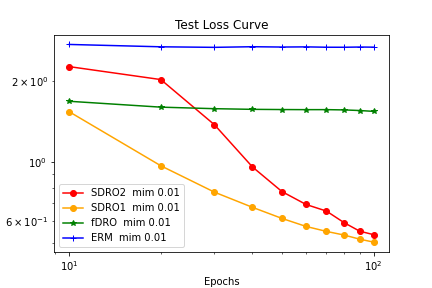}
    }
    \hfill
    \subfigure[$\epsilon_{\text{MIM}}=0.02$]{%
        \includegraphics[width=0.32\textwidth]{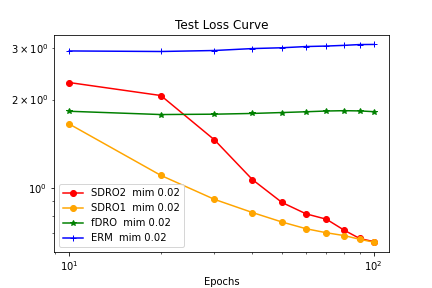}
    }
    \hfill
    \subfigure[$\epsilon_{\text{MIM}}=0.05$]{%
        \includegraphics[width=0.32\textwidth]{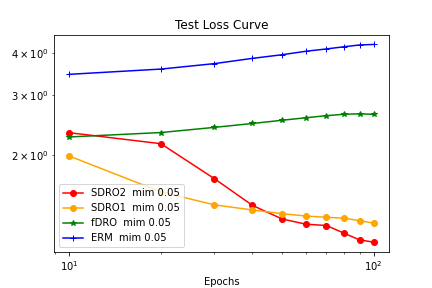}
    }

    \caption{Test loss curves of LeNet under MIM attacks with different perturbation levels.}
    \label{LeNet test loss 4}
\end{figure*}

Figure \ref{LeNet test loss 1}, \ref{LeNet test loss 2}, \ref{LeNet test loss 3} and \ref{LeNet test loss 4} plot the test loss in log-log scale, where the first to last row represents test CE-loss under FGSM, $\ell_{\infty}, \ell_{2}$-PGD attack \citep{pgdattack} and MIM attack \citep{mim} respectively. For $\ell_{\infty},\ell_{2}$-PGD attack, we set their learning rates as $\alpha_{\text{PGD}_{\infty}}= \epsilon_{\text{PGD}_{\infty}}/10, \alpha_{\text{PGD}_{2}}= \epsilon_{\text{PGD}_{2}}/10$; For MIM attack, we set the moving-average parameter $1.0$ and its learning rate $\alpha_{\text{mim}}= \epsilon_{\text{mim}}/15$. From above results, we conclude Sinkhorn DRO dual formulation in general is more robust against distribution shifts. Compared with Sinkhorn DRO dual formulation \eqref{baseline SinkhornDRO}, our method attains smaller accuracy drop and better classification accuracy in most scenarios, which demonstrates the effectiveness of proposed Sinkhorn DRO dual formulation \eqref{dual2} and its convergence analysis.

\section{Ablation Study}\label{Appendix: ablation study}
In this section, we conduct ablation studies for our proposed Sinkhorn DRO dual formulation \eqref{dual2} over linear and logistic regression, where we focus on components of \eqref{dual2} having potential effects of model robustness performance, including regularization parameter $\lambda$, the cost metric $c(\cdot,\cdot)$ and the choices of information divergence conjugate dual $f^*(\cdot)$. For linear regression, we slightly modify the training data generation and model initialization procedures as follows. The optimal model parameters $x^*$ are sampled from $\mathcal{N}(0, 2.25e^{-2})$, and the corresponding output training data are generated according to $\zeta_{\text{output}} = \zeta_{\text{train}} \cdot x^* + \epsilon_{\text{noise}}$, where $\epsilon_{\text{noise}} \sim \mathcal{N}(0, 1e^{-2})$. In addition to normalizing the training and test data as described in Appendix~\ref{Appendix: linear regression experiments}, we also normalize the initial parameters after sampling them from $x_0 \sim \mathcal{N}(x^*, 1)$. We apply gaussian attack on subsampled test dataset to evaluate the robust performance of linear regression trained by proposed Sinkhorn DRO dual formulation \eqref{dual2}.
For logistic regression, we adopt the same setup as described in Appendix~\ref{Appendix: logistic regression experiments}, except that we increase the iteration number of Algorithm~\ref{alg1} from $5$ to $8$. And we vary FGSM attack \citep{fgsm} strength $\epsilon_{\text{fgsm}}$ over subsampled test data to evaluate the robust performance of logistic regression trained via proposed Sinkhorn DRO dual formulation \eqref{dual2}.

\subsubsection{Effects of regularization  \texorpdfstring{$\lambda$}{lambda}}\label{Sub Appendix: ablation on regularization}
In this section, we vary $\lambda$ over the set $\{0.01, 0.1, 1.0, 10\}$. We fine-tune the learning rates for algorithm~\ref{alg2} and algorithm~\ref{alg1} to be $1e^{-2}$ and $1e^{-1}$, respectively, for all models. Figure~\ref{Ablation Study: Regularization} (left) plots the test (quadratic) loss for linear regression under different gaussian attack levels, where the numbers in the legend indicate the corresponding value of $\lambda$. We observe that using a smaller regularization parameter improves the model’s robustness against distributional shifts. In particular, when $\lambda = 0.01$, the green curves exhibit smaller shifts relative to others, which aligns with classical insights on the effect of regularization~\citep{prml}. However, from optimization perspective, a small $\lambda$ makes the proposed dual formulation~\eqref{dual2} more difficult to train with promising accuracy guarantee.
\begin{figure*}[!htbp]
    \subfigure[linear regression]{
        \includegraphics[width=0.45\textwidth]{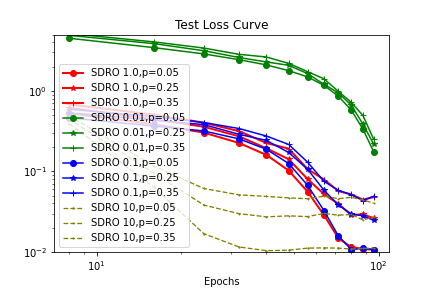}
    }
    \hfill
    \subfigure[logistic regression]{
        \includegraphics[width=0.45\textwidth]{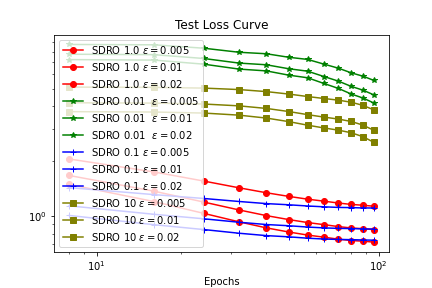}
    }
    \caption{Effects of Regularization $\lambda$}
   \label{Ablation Study: Regularization}
\end{figure*}
We also conduct the same experiments for logistic regression, where the learning rates of algorithm~\ref{alg2} and algorithm~\ref{alg1} are fine-tuned to $8e^{-2}$ and $1e^{-1}$, respectively. Figure~\ref{Ablation Study: Regularization} (right) plots the test loss curves for logistic regression under different FGSM attack levels. We observe that when $\lambda = 0.01$ or $10$, the proposed dual formulation~\eqref{dual2} fails to train a valid model, with the resulting classification accuracies dropping below $10\%$. In contrast, for $\lambda = 0.1$ and $1.0$, the test losses converge to similar scales across different attack levels. These results suggest that a proper choice of the regularization parameter lies in the range around $[0.1, 1.0]$.
%\begin{table}[ht]
%\centering
%\begin{tabular}{ |p{1.5cm}|p{1.5cm}|p{1.5cm}|p{1.5cm}|  }
% \hline
% \multicolumn{4}{|c|}{Classification Accuracy} \\
% \hline
% $\lambda = 0.01$& $\lambda=0.1$  & $\lambda = 1.0$ &$\lambda = 10$\\
% \hline
%  3.7\%  &  33.33\%  & 71.89\% &  63.0\%\\
%  7.8\% & 38.3\% & 74.7\%   & 67.0\% \\
%  15.9\% & 45.0\% & 77.44\% &  69.89\%\\
% \hline
%\end{tabular}
%\end{table}
%\subsubsection{Effects of sample size for $\xi$}
\subsubsection{Effects of cost metric \texorpdfstring{$c(\cdot,\cdot)$}{c(.,.)}}
\begin{figure*}[!htbp]
    \subfigure[linear regression]{
        \includegraphics[width=0.45\textwidth]{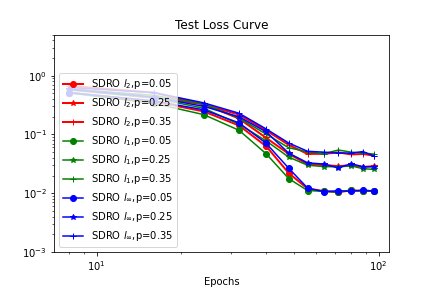}
    }
    \hfill
    \subfigure[logistic regression]{
        
        \includegraphics[width=0.45\textwidth]{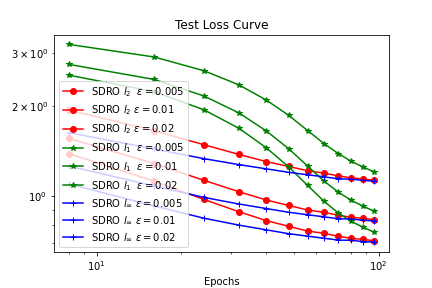}
    }
    \caption{Effects of cost metric $c(\cdot, \cdot)$}
   \label{Ablation Study: cost metric}
\end{figure*}
In this section, we vary the cost metric $c(\cdot, \cdot)$ among the $\{\ell_1, \ell_2, \ell_{\infty}\}$-norms to examine the impact of cost metric choices. Specifically, for the $\ell_2$-norm, we set $c(\zeta,\xi) = \|\zeta - \xi\|_2$; for the $\ell_1$-norm, we set $c(\zeta,\xi) = 0.2 \cdot |\zeta - \xi|_1$; and for the $\ell{\infty}$-norm, we set $c(\zeta,\xi) = 2 \cdot \|\zeta - \xi\|_{\infty}$. To ensure a fair comparison, we fix the regularization parameter at $\lambda = 0.8$ and retain the same learning rates for Algorithm~\ref{alg2} and Algorithm~\ref{alg1} as used in Section~\ref{Sub Appendix: ablation on regularization}. Figure~\ref{Ablation Study: cost metric} (left) plots the test loss for linear regression trained with different cost metrics under the proposed dual formulation~\eqref{dual2}. As shown in the figure, the choice of cost metric has a marginal effect on linear regression, as the test curves across different norms are nearly indistinguishable. However, for logistic regression, we observe that using the $\ell_1$-norm makes the proposed dual formulation~\eqref{dual2} less effective in learning a robust model compared to the $\ell_2$ and $\ell_{\infty}$ norms. This suggests that $\ell_2$ and $\ell_{\infty}$ norms are more reliable choices for $c(\cdot, \cdot)$ in practice.

\subsubsection{Effects of choices of \texorpdfstring{$f^*$}{f*}}
\begin{table}[ht]
    \centering
    \caption{Primal and dual expressions for different divergence measures}
    \label{f expression}
    \renewcommand{\arraystretch}{1.5}
    \begin{tabular}{c|c|c}
        \textbf{Divergence} & $f(t)$ & $f^*(t)$ \\
        \hline
        $\chi^2$ & $\frac{1}{2} (t - 1)^2$ & $-1 + \frac{1}{4} (t + 2)^2_+$ \\
        \hline
        \textit{KL} & $t \log t - t + 1$ & $\exp(t) - 1$\\
        \hline
        \textit{smoothed CVaR} & $f_{\alpha}^{\text{smo}}(t) = \begin{cases} 
        t \log t + \frac{1 - \alpha t}{\alpha} \log \frac{1 - \alpha t}{1 - \alpha}, & t \in [0, 1/\alpha)\\+\infty, & \textit{otherwise}
    \end{cases}$ & $\frac{1}{\alpha}\log(1-\alpha+\alpha\exp(t))$
    \end{tabular}
\end{table}

In this section, we test the effects of conjugate functions $f^*(\cdot)$. We select three classical divergence measures, including $\chi^2$-divergence, KL-divergence and smoothed CVaR divergence~\citep{jin2021nonconvex}. To elaborate, we list the primal and conjugate dual expressions in following Table~\ref{f expression}.

For the linear regression task, we vary the conjugate functions corresponding to different information divergences, while keeping the same learning rates for Algorithm~\ref{alg2} and Algorithm~\ref{alg1} as in Section~\ref{Sub Appendix: ablation on regularization}, and fix the regularization parameter at $\lambda = 0.8$. Specifically, for conjugate dual of smoothed-CVaR, we set $\alpha = 0.5$. Figure~\ref{Ablation Study: f-divergence} (left) plots the test loss for linear regression trained using the proposed dual formulation~\eqref{dual2}. We observe that using the conjugate dual of KL-divergence hinders fast convergence compared with other $f^*$ satisfying $M$-smoothness property. However, the test accuracy evaluated at last epoch reachs same level under the existence of gaussian attack. 
\begin{figure*}[ht!]
     \subfigure[linear regression]{
        \includegraphics[width=0.45\textwidth]{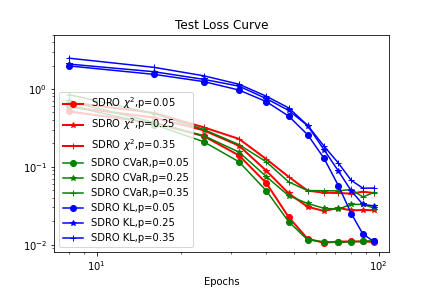}
    }
    \hfill
    \subfigure[logistic regression]{
        \includegraphics[width=0.45\textwidth]{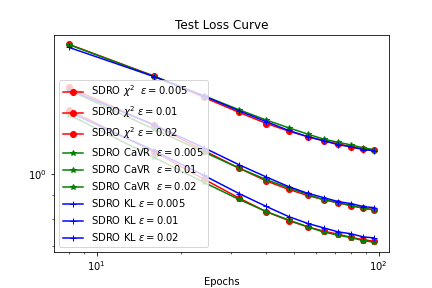}
    }
    \caption{Effects of conjugate function $f^*$}
   \label{Ablation Study: f-divergence}
\end{figure*}
For logistic regression, we observe that the robust performance achieved using the conjugate dual of the KL-divergence is comparable to that obtained with a conjugate dual $f^*$ satisfying the $M$-smoothness property. The performance gap across different models may be attributed to the sensitivity of the KL-divergence to model initialization, as the local $M$-smooth constant can vary depending on the starting point. Nevertheless, empirical results on both linear and logistic regression suggest that enforcing the $M$-smoothness assumption does not hinder practical applicability. As long as the model is not initialized in an ill-conditioned region, choosing $f^*$ as the conjugate dual of the KL-divergence yields similar convergence, supporting the practical validity of the $M$-smoothness assumption.

\section{Proof of Theorem \ref{thm: dual formulation}}\label{proof: dual formulation thm}
\dual*
\begin{proof}
%Since $\mathbb{P}$ is a static nominal distribution, and the expectation over loss function $\mathbb{E}_{\xi \sim \mathbb{Q}}\bigr[\ell(x;\xi)\bigr]$ is a constant for specific $\mathbb{Q}$. 
Merging the infimum operator $\inf_{\gamma \in \Gamma(\mathbb{P},\mathbb{Q})}$ with the supremum operator $\sup_{\mathbb{Q}}$ in \eqref{primal1-0}, we obtain the following equivalent form.
\begin{align}
    \min_{x\in \mathbf{R}^d} \sup_{\mathbb{Q},\gamma \in \Gamma(\mathbb{P}, \mathbb{Q})} \Biggr\{ \mathbb{E}_{\xi \sim \mathbb{Q}}\Bigr[\ell(x;\xi)\Bigr] -\Bigr[\mathbb{E}_{(\zeta,\xi)\sim \gamma}\bigr[\lambda c(\zeta,\xi)\bigr]+ \lambda \red{\beta} D_f(\gamma | \mathbb{P}\otimes \nu)\Bigr]\Biggr\}.
    \label{primal1-0-1}
\end{align}
Regarding the joint distribution $\gamma(\zeta,\xi)$, we decompose it as $\gamma(\zeta,\xi) = \gamma(\xi | \zeta)\mathbb{P}(\zeta)$, where $\mathbb{P}(\zeta)$ denotes the marginal distribution and $\gamma(\xi | \zeta)$ corresponds to the conditional distribution given $\zeta$. Then, the constraint $\gamma \in \Gamma(\mathbb{P}, \mathbb{Q})$ is equivalent to $\mathbb{E}_{\zeta \sim \mathbb{P}}\bigr[\gamma(\xi|\zeta)\bigr]=\mathbb{Q}$, and hence \eqref{primal1-0-1} reduces to
\begin{align}
     \min_{x\in \mathbf{R}^d} \sup_{\mathbb{Q} \text{ s.t. } \mathbb{E}_{\zeta \sim \mathbb{P}}\bigr[\gamma(\xi|\zeta)\bigr]=\mathbb{Q}} \Biggr \{ \mathbb{E}_{\zeta \sim \mathbb{P}}  \Bigr[ \mathbb{E}_{\xi \sim \gamma( \cdot \mid \zeta)}\bigr[\ell(x; \xi) - \lambda c(\zeta, \xi )\bigr] -\lambda \red{\beta} D_f\big(\gamma(\xi|\zeta) | \nu(\xi)\big) \Bigr] \Biggr \}.
     \label{primal1-1}
\end{align}
We claim that the optimal value of \eqref{primal1-1} equals that of the following problem, which takes supremum over all possible conditional distributions $\gamma(\xi|\zeta)$.
\begin{align}
         \min_{x\in \mathbf{R}^d}\sup_{\gamma(\xi | \zeta)} \Biggr \{ \mathbb{E}_{\zeta \sim \mathbb{P}}  \Bigr[ \mathbb{E}_{\xi \sim \gamma(\cdot|\zeta)}\bigr[\ell(x; \xi) - \lambda c(\zeta, \xi )\bigr] -\lambda \red{\beta} D_f\big(\gamma(\xi|\zeta) | \nu(\xi)\big) \Bigr] \Biggr \}.
     \label{primal1-2}
\end{align}
To show this, for any fixed $x$, suppose the supremum of \eqref{primal1-1} is achieved by a certain conditional distribution $\gamma(\xi|\zeta)$ that satisfies $\mathbb{E}_{\zeta \sim \mathbb{P}}\bigr[ \gamma(\xi|\zeta)\bigr]=\mathbb{Q}$, and such $\gamma(\xi|\zeta)$ is feasible for the supremum of \eqref{primal1-2}. Thus, the supremum of \eqref{primal1-1} is lower than the supremum of \eqref{primal1-2}. On the other hand, for any fixed $x$, suppose the supremum of \eqref{primal1-2} is achieved by a certain conditional distribution $\gamma(\xi|\zeta)$. Then, the distribution $\mathbb{Q}$ given by $\mathbb{Q} = \mathbb{E}_{\zeta\sim \mathbb{P}}\bigr[ \gamma(\xi|\zeta) \bigr]$ is feasible for the supremum of \eqref{primal1-1}. Consequently, the supremum of \eqref{primal1-1} is higher than that of \eqref{primal1-2}. In summary, \eqref{primal1-1} and \eqref{primal1-2} are equivalent. 

%Since for regularized DRO, there is no constraint on worst-case distribution $\mathbb{Q}$. Fix $x$, \eqref{primal1-1} reaches supremum when $\mathbb{Q}$ attains worst case distribution. There exists a corresponding conditional distribution $\gamma(\zeta |\xi)$ satisfying . However, existing $\gamma(\zeta|\xi)$ may not be optimal for \eqref{primal1-2}. Thus, the supremum of \eqref{primal1-1} is smaller than supremum of \ref{primal1-2}. On the contrary, fix $x$, when \eqref{primal1-2} reaches supremum and obtains worst-case distribution $\gamma(\xi|\zeta)$, there is a corresponding $\mathbb{Q}$ by the transformation $\mathbb{E}_{\zeta\sim \mathbb{P}}\bigr[ \gamma(\zeta|\xi) \bigr]=\mathbb{Q}$, which is feasible but might not be optimal solution for \eqref{primal1-1}. This indicates the supremum of \eqref{primal1-1} is larger than \eqref{primal1-2}. Combine above arguments, the optimal value of \eqref{primal1-1} and \eqref{primal1-2} under worst-case distribution are equal. 
%We can transform the problem of finding worst-case marginal distribution $\mathbb{Q}$ into finding the worst-case conditional distribution $\gamma(\xi|\zeta)$ for each $\zeta$. 

Furthermore, by principle of interchangeability (Theorem 7.92, Chapter 7.3.2 from \citet{reference4interchangebility}), \eqref{primal1-2} can be rewritten as
\begin{align}
     \min_{x\in \mathbf{R}^d} \mathbb{E}_{\zeta \sim \mathbb{P}}  \Biggr[ \underbrace{ \sup_{
      \gamma(\xi|\zeta)} \Biggr \{ \mathbb{E}_{\xi \sim \gamma(\cdot|\zeta)}\bigr[\ell(x; \xi) - \lambda c(\zeta, \xi )\bigr] -\lambda \red{\beta} D_f\big(\gamma(\xi|\zeta) | \nu(\xi)\big) \Biggr \}}_{\Psi_{\zeta}(x)}\Biggr ].
     \label{primal2}
\end{align}
Next, for every fixed $\zeta$ and $x$, denote ${\mu}_{\gamma|\zeta}$ and ${\mu}_\nu$ as the distributions of the scalar random variable $\ell(x;\xi)-\lambda c(\zeta,\xi)$ under $\gamma(\xi|\zeta)$ and $\nu(\xi)$, respectively. 
we show that the $\Psi_{\zeta}(x)$ defined above is equivalent to the following auxiliary function
\begin{align}
    \widetilde{\Psi}_{\zeta}(x)=\sup_{ {\mu}_{\gamma|\zeta}} \Bigg \{ \mathbb{E}_{{\mu}_{\gamma|\zeta}} \bigr[\ell(x; \xi) - \lambda c(\zeta, \xi )\bigr] -\lambda \red{\beta} D_f({{\mu}_{\gamma|\zeta}} | {\mu}_\nu) \Bigg \},
     \label{primal3}
\end{align}
where $\sup_{{\mu}_{\gamma|\zeta}}$ corresponds to the supremum over all possible distributions ${\mu}_{\gamma|\zeta}$ induced by $\gamma(\xi|\zeta)$.
To show this, for any fixed $x$, suppose the supremum in $\Psi_\zeta(x)$ is achieved by a certain $\gamma(\xi|\zeta)$, and denote the induced distribution of $\ell(x;\xi)-\lambda c(\zeta,\xi)$ as ${\mu}_{\gamma|\zeta}$. It is straightforward to show that 
\begin{align}
\mathbb{E}_{\xi \sim \gamma(.|\zeta)}\bigr[\ell(x; \xi) - \lambda c(\zeta, \xi )\bigr] = \mathbb{E}_{{\mu}_{\gamma|\zeta}}\bigr[\ell(x; \xi) - \lambda c(\zeta, \xi )\bigr].\nonumber 
\end{align}
Moreover, by the data processing inequality, it holds that $D_f\big({\mu}_{\gamma|\zeta} | {\mu}_\nu\big) \leq D_f\big(\gamma(\xi| \zeta) | \nu(\xi)\big)$. Since ${\mu}_{\gamma|\zeta}$ is feasible for the supremum of \eqref{primal3}, we conclude that $\Psi_{\zeta}(x)\leq \widetilde{\Psi}_{\zeta}(x)$. Conversely, for any fixed $x$, suppose the supremum in $\widetilde{\Psi}_{\zeta}(x)$ is achieved by a certain ${\mu}_{\gamma|\zeta}$, then the corresponding $\gamma(\zeta|\xi)$ (which induces ${\mu}_{\gamma|\zeta}$) is feasible for the supremum in $\Psi_{\zeta}(x)$, and hence we have that $\widetilde{\Psi}_{\zeta}(x) \leq \Psi_{\zeta}(x)$. Finally, we conclude that $\Psi_{\zeta}(x)= \widetilde{\Psi}_{\zeta}(x)$.

Using inverse c.d.f. sampling based on the cumulative distribution function over $\mu_{\nu}$, the $f$-divergence between $\mu_{\gamma|\zeta}$ and $\mu_{\nu}$ can be rewritten as
\begin{align}
D_f(\mu_{\gamma|\zeta}|\mu_{\nu}) = \int f\Big(\frac{\mathrm{d}\mu_{\gamma|\zeta}}{\mathrm{d}\mu_{\nu}}\Big)\mathrm{d}\mu_{\nu}= \int f\Big(\frac{\mathrm{d} \mu'}{\mathrm{d}\textit{Unif}([0,1])}\Big)\mathrm{d}\textit{Unif}([0,1]) =D_f\big(\mu'\mid \textit{Unif}([0,1])\big) ,\nonumber
\end{align}
where $\textit{Unif}([0,1])$ represents the uniform distribution over $[0,1]$, $\mu'=\mu_{\nu}^{-1} \circ \mu_{\gamma|\zeta}$. Moreover, for fixed $x$ and $\zeta$, denote the cumulative distribution function of the scalar random variable $\ell(x;\xi)-\lambda c(\zeta,\xi)$ as $F(t)=\mathbb{P}\big(\ell(x;\xi)-\lambda c(\zeta,\xi)\le t\big)$. We can further transform $\Psi_{\zeta}(x)$ into
\begin{align}
    \Psi_{\zeta}(x) &=\sup_{\mu'} \Bigg \{ \mathbb{E}_{F^{-1}(u)\sim \mu'}\bigr[F^{-1}(u)\bigr] -\lambda \red{\beta} D_f\big(\mu'| \textit{Unif}([0,1])\big) \Bigg \}\\
    &=\sup_{\mu'} \int F^{-1}(u)\mathrm{d}\mu'(u) - \lambda \red{\beta} \int f\Big(\frac{\mathrm{d}\mu'(u)}{\mathrm{d}\textit{Unif}([0,1])(u)}\Big)\mathrm{d}\textit{Unif}([0,1])(u) \\
    &=\sup_{\mu'} \int \Big(F^{-1}(u)\frac{\mathrm{d}\mu'(u)}{\mathrm{d}\textit{Unif}([0,1])(u)}-\lambda \red{\beta} f\Big(\frac{\mathrm{d}\mu'(u)}{\mathrm{d}\textit{Unif}([0,1])(u)}\Big)\Big) \mathrm{d}\textit{Unif}([0,1])(u) \\
    &=\sup_{r\in \mathcal{R}}\int_{0}^{1} \bigr[r(u)F^{-1}(u) - \lambda \red{\beta} f(r(u)) \bigr]\mathrm{d}u, \label{eq: inverse}
\end{align}
where $r(u)=\frac{\mathrm{d}\mu'}{\mathrm{d}\textit{Unif}}(u)$ and $\mathcal{R}=\{r:[0,1]\rightarrow \mathbf{R}_{+} \mid \int_{0}^{1} r(u) \mathrm{d}u = 1 \}$. Introduce a dual variable $\eta$ for the constraint $\mathcal{R}$, the Lagrange dual formulation of \eqref{eq: inverse} can be written as
\begin{align}
\Psi_{\zeta}(x) & =\min _{\eta \in \mathbf{R}} \mathcal{L}_{\zeta}(x, \eta), \text { where } \\
& \mathcal{L}_{\zeta}(x, \eta )=\sup _{r}\int_0^1 \big[r(u)F^{-1}(u)-\eta(r(u)-1)-\lambda\red{\beta} f\big(r(u)\big)\big] \mathrm{d} u.
\label{dual0}
\end{align}
Since for fixed $\eta$, $\mathcal{L}_{\zeta}(x,\eta)$ can be denoted as $\sup_{r}\mathbb{E}_{u \sim \textit{Unif}[0,1]}\big[\vartheta_{\eta}(r(u),u) \big]$, where $\vartheta_{\eta}(r(u),u) = r(u)F^{-1}(u)-\eta(r(u)-1)-\lambda \red{\beta} f\big(r(u)\big)$ is a continuous function, by principle of interchangeability stated in Theorem 7.92 \citep{reference4interchangebility}, the order of $\sup$ and integral can be swapped, which yields that
\begin{align}
    \Psi_{\zeta}(x) & =\min_{\eta \in \mathbf{R}} \underbrace{\int_0^1 \sup _{r \in \mathbf{R}_{+}}\big[rF^{-1}(u)-\eta(r-1)-\lambda\red{\beta} f(r)\big] \mathrm{d} u}_{\mathcal{L}_{\zeta}(x,\eta)}.
    \label{dual01}
\end{align}
Define the conjugate function $f^*(v)=\sup_{r \in \mathbf{R}_{+}}\big\{vr-f(r) \big\}$, we further obtain that
\begin{align}
    \Psi_{\zeta}(x)&=\min_{\eta\in \mathbf{R}} \int_{0}^{1}\lambda \red{\beta} f^*\Big(\frac{F^{-1}(u)-\eta}{\lambda \red{\beta}}\Big)\mathrm{d}u +\eta
    %&= \min_{\eta \in \mathbf{R}}\int \lambda \red{\beta} f^*(\frac{\ell(x;\xi)-\lambda c(\zeta,\xi)-\eta}{\lambda \red{\beta}})\mathrm{d}\nu +\eta\\
    =\min_{\eta \in \mathbf{R}}    \mathbb{E}_{\xi \sim \nu}\Big[\lambda \red{\beta} f^*\Big(\frac{\ell(x;\xi)-\lambda c(\zeta,\xi)-\eta}{\lambda \red{\beta}}\Big)\Big]+\eta.
    \label{dual1}
\end{align}
Finally, the dual formulation of \eqref{primal2} is expressed as
\begin{align}
    \min_{x\in \mathbf{R^d}} \mathbb{E}_{\zeta \sim \mathbb{P}}\Biggr[ \underbrace{ \min_{\eta \in \mathbf{R}}    \mathbb{E}_{\xi \sim \nu}\Big[\lambda \red{\beta} f^*\Big(\frac{\ell(x;\xi)-\lambda c(\zeta,\xi)-\eta}{\lambda \red{\beta}}\Big)\Big]+\eta}_{\Psi_{\zeta}(x)} \Biggr],\nonumber
\end{align}
which gives the desired result.
\end{proof}

\section{Proof of Lemma \ref{jin}}
\jin*
\begin{proof}
\noindent We refer to Lemma 2.6 in \citet{jin2021nonconvex} for the detailed proof.
\end{proof}

\section{Proof of Theorem \ref{thm: grad_error_bound}}\label{Proof: grad_error_bound thm}
\approximationerror*
\begin{proof}
    First, taking square root and applying Jensen's inequality on both sides, \eqref{eq: inner near-optima} implies
    \begin{align}
        \mathbb{E}_{\eta_x^{\tilde{d}}(\zeta)}|\nabla_2 \mathcal{L}_{\zeta}(x,\eta_x^{\tilde{d}}(\zeta)) |\leq \tilde{\varepsilon}.
    \end{align}
    Since at $\eta_x^*(\zeta)$, by optimality condition, $\nabla_{2}\mathcal{L}_{\zeta}(x,\eta_x^*(\zeta))=0$, we have
    \begin{align}
        &\mathbb{E}_{\eta_{x}^{\tilde{d}}(\zeta)}\big|\nabla_2 \mathcal{L}_{\zeta}(x,\eta_{x}^{\tilde{d}}(\zeta)) \big| \nonumber\\
        =&\mathbb{E}_{\eta_{x}^{\tilde{d}}(\zeta)}\big | \nabla_{2} \mathcal{L}_{\zeta}(x,\eta_{x}^{\tilde{d}}(\zeta))- \nabla_{2}\mathcal{L}_{\zeta}(x,\eta_x^*(\zeta)) \big | \nonumber\\
       =&\mathbb{E}_{\eta_{x}^{\tilde{d}}(\zeta)} \Big |\mathbb{E}_{\xi\sim \nu}\Big[ (f^{*})' \Big(\frac{\ell(x;\xi) - c(\zeta,\xi)- \eta_{x}^{\tilde{d}}(\zeta)}{\lambda \red{\beta}}\Big)\Big] - \mathbb{E}_{\xi\sim \nu}\Big[(f^{*})' \Big(\frac{\ell(x;\xi) - c(\zeta,\xi) - \eta_{x}^*(\zeta)}{\lambda \red{\beta}}\Big)\Big]\Big | \nonumber \\
       \leq& \tilde{\varepsilon}. \label{epsilon-optimal}
    \end{align}
    Applying expression \eqref{epsilon-optimal}, for given $x$ and every $\zeta$, the approximation error between $\nabla \Psi(x)$ and $\mathbb{E}_{\eta_x^{\tilde{d}}(\zeta)}[\nabla_1 \mathcal{L}(x,\eta_x^{\tilde{d}}(\zeta))]$ can be upper bounded as follows
    \begin{align}
        &\big \|\mathbb{E}_{\eta_{x}^{\tilde{d}}(\zeta)}[\nabla_1 \mathcal{L}_{\zeta}(x,\eta_{x}^{\tilde{d}}(\zeta))] - \nabla \Psi_{\zeta}(x) \big \|\nonumber \nonumber\\
        =&\big \|\mathbb{E}_{\eta_{x}^{\tilde{d}}(\zeta)}[\nabla_1 \mathcal{L}_{\zeta}(x,\eta_{x}^{\tilde{d}}(\zeta)) - \nabla \Psi_{\zeta}(x)] \big \|\nonumber \nonumber\\
        \overset{(i)}{\leq}&\mathbb{E}_{\eta_{x}^{\tilde{d}}(\zeta)}\big \|  \nabla_1 \mathcal{L}_{\zeta}(x,\eta_{x}^{\tilde{d}}(\zeta)) - \nabla \Psi_{\zeta}(x) \big \|\nonumber\\
        =&\mathbb{E}_{\eta_{x}^{\tilde{d}}(\zeta)} \big \|  \mathbb{E}_{\xi\sim \nu} \big [ (f^{*})'\Big(\frac{\ell(x;\xi)- \lambda c(\zeta,\xi)-\eta_{x}^{\tilde{d}}(\zeta)}{\lambda \red{\beta}}\Big) \nabla \ell(x;\xi)  
        -(f^{*})'\Big(\frac{\ell(x;\xi)- \lambda c(\zeta,\xi)-\eta_x^*(\zeta)}{\lambda \red{\beta}}\Big)\nabla \ell(x;\xi)\big] \big \| \nonumber\\
        =&\mathbb{E}_{\eta_{x}^{\tilde{d}}(\zeta)}\big \|  \mathbb{E}_{\xi\sim \nu}\big [(f^{*})'\Big(\frac{\ell(x;\xi)- \lambda c(\zeta,\xi)-\eta_{x}^{\tilde{d}}(\zeta)}{\lambda\red{\beta}}\Big) -(f^{*})'\Big(\frac{\ell(x;\xi)- \lambda c(\zeta,\xi)-\eta_x^*(\zeta)}{\lambda\red{\beta}}\Big)\big]\cdot\nabla \ell(x;\xi) \big \| \nonumber\\
        \overset{(ii)}{\leq}& \mathbb{E}_{\eta_{x}^{\tilde{d}}(\zeta),\xi\sim \nu} \big \|  \big [(f^{*})'\Big(\frac{\ell(x;\xi)- \lambda c(\zeta,\xi)-\eta_{x}^{\tilde{d}}(\zeta)}{\lambda\red{\beta}}\Big) -(f^{*})'\Big(\frac{\ell(x;\xi)- \lambda c(\zeta,\xi)-\eta_x^*(\zeta)}{\lambda\red{\beta}}\Big)\big]\cdot \nabla \ell(x;\xi) \big \| \nonumber\\
        \overset{(iii)}{=}& \mathbb{E}_{\eta_{x}^{\tilde{d}}(\zeta), \xi\sim \nu} \big|  \big [(f^{*})'\Big(\frac{\ell(x;\xi)- \lambda c(\zeta,\xi)-\eta_{x}^{\tilde{d}}(\zeta)}{\lambda\red{\beta}}\Big) -(f^{*})'\Big(\frac{\ell(x;\xi)- \lambda c(\zeta,\xi)-\eta_x^*(\zeta)}{\lambda\red{\beta}}\Big)\big] \big|\cdot \big \| \nabla \ell(x;\xi) \big \|\nonumber\\
        \overset{(iv)}{\leq}& G \mathbb{E}_{\eta_{x}^{\tilde{d}}(\zeta),\xi\sim\nu} \big \|  (f^{*})'\Big(\frac{\ell(x;\xi)- \lambda c(\zeta,\xi)-\eta_{x}^{\tilde{d}}(\zeta)}{\lambda\red{\beta}}\Big) -(f^{*})'\Big(\frac{\ell(x;\xi)- \lambda c(\zeta,\xi)-\eta_x^*(\zeta)}{\lambda\red{\beta}}\Big) \big \|\nonumber\\
        \overset{(v)}{=}& G  \mathbb{E}_{\eta_{x}^{\tilde{d}}(\zeta)}\big \| \mathbb{E}_{\xi\sim \nu} \big [(f^{*})'\Big(\frac{\ell(x;\xi)- \lambda c(\zeta,\xi)-\eta_{x}^{\tilde{d}}(\zeta)}{\lambda\red{\beta}}\Big) -(f^{*})'\Big(\frac{\ell(x;\xi)- \lambda c(\zeta,\xi)-\eta_x^*(\zeta)}{\lambda\red{\beta}}\Big)\big] \big \|\nonumber\\
        =& G\mathbb{E}_{\eta_{x}^{\tilde{d}}(\zeta)}\big \| \nabla_2 \mathcal{L}_{\zeta}(x,\eta_{x}^{\tilde{d}}(\zeta)) - \nabla_2 \mathcal{L}_{\zeta}(x,\eta_x^*(\zeta))\big \|\nonumber\\
        \overset{(vi)}{\leq}& G\tilde{\varepsilon}=\varepsilon, \label{optimalrelation1}
    \end{align}
    where (i) applies Jensen's inequality to extract out expectation over $\eta_x^{\tilde{d}}(\zeta)$; (ii) applies Jensen's inequality to extract out expectation over $\xi$; (iii) extracts scalar\\ $(f^{*})'\Big(\frac{\ell(x;\xi)- \lambda c(\zeta,\xi)-\eta_{x}^{\tilde{d}}(\zeta)}{\lambda\red{\beta}}\Big) -(f^{*})'\Big(\frac{\ell(x;\xi)- \lambda c(\zeta,\xi)-\eta_x^*(\zeta)}{\lambda\red{\beta}}\Big)$ out; (iv) applies $G$-Lipschitz assumption of $\ell(x;\xi)$ stated at Assumption \ref{assum1}; (v) applies monotonicity property of $(f^{*})'$ with regard to $\eta$ given $x$ and fixed $\zeta$ (the sign of $(f^{*})'(\frac{\ell(x;\xi)- \lambda c(\zeta,\xi)-\eta_{x}^{\tilde{d}}(\zeta)}{\lambda\red{\beta}}) -(f^{*})'(\frac{\ell(x;\xi)- \lambda c(\zeta,\xi)-\eta_x^*(\zeta)}{\lambda\red{\beta}} ) $ is fixed regardless of $\xi$, which enables to move the expectation over $\xi$ into the norm without changing its value); (vi) applies condition \eqref{epsilon-optimal}.
    Squaring both sides and re-arranging LHS, we conclude that 
    \begin{align}
    \bigl \|  \nabla \Psi_{\zeta}(x) - \mathbb{E}_{\eta_{x}^{\tilde{d}}(\zeta)} [\nabla_1 \mathcal{L}_{\zeta}(x,\eta_{x}^{\tilde{d}}(\zeta))]\bigl \|^2 \leq \varepsilon^2 , \forall  \zeta\sim \mathbb{P},
    \end{align}
    which gives the desired result.
\end{proof}

\section{Proof of Lemma \ref{thm: smooth2}}\label{Proof: inner smooth thm}
\innersmooth*
\begin{proof}
    Notice that the gradient of $\mathcal{L}_{\zeta,\xi}(x,\eta)$ takes the form 
    \begin{align}
        \nabla_2\mathcal{L}_{\zeta,\xi}(x,\eta)=1-\mathbb{E}_{\xi\sim \nu}\Big[(f^*)' \Big(\frac{\ell(x;\xi)-c(\zeta,\xi)-\eta}{\lambda \red{\beta}}\Big)\Big]. \nonumber
    \end{align}
    Then, we obtain that
    \begin{align}
        &\mathbb{E}_{\xi\sim \nu}\big \| \nabla_{2} \mathcal{L}_{\zeta,\xi}(x,\eta) - \nabla_{2} \mathcal{L}_{\zeta,\xi}(x, \eta') \big \|^2\nonumber\\
        =&\mathbb{E}_{\xi\sim \nu} \big \|(f^*)' \Big(\frac{\ell(x;\xi)-c(\zeta,\xi)-\eta}{\lambda \red{\beta}}\Big)-(f^*)' \Big(\frac{\ell(x;\xi)-c(\zeta,\xi)-\eta'}{\lambda \red{\beta}}\Big) \big \|^2 \nonumber \\
        \leq & (M (\lambda \red{\beta})^{-1})^2 \big \| \eta - \eta'\big \|^2, \nonumber
    \end{align}
    where the inequality is due to the $M$-smoothness assumption of $f^*$ stated in assumption \ref{assum1}.
\end{proof}
\section{Proof of Lemma \ref{thm: boundedmoment2}}\label{proof: inner gradient bounded moment}
\noindent Through this work, we utilize the following proposition for variance computing.
\begin{proposition}[Variance computing]\label{var_comp}
    Given two i.i.d. random variables $X_1$ and $X_2$, the variance can be calculated as
    \begin{align}
    2Var(X_1)=2Var(X_2) = \mathbb{E} \big \|X_1-X_2 \big \|^2.
    \label{Var}
    \end{align}
\end{proposition}
\begin{proof}
The proof simply extends from variance definition, 
\begin{align}
    2Var(X)=&2\mathbb{E}\big \| X_1-\bar{X}\big \|^2 \nonumber\\
    =&2\mathbb{E}\| X_1-\bar{X}\|^2+2\mathbb{E}(X_1 - \bar{X})(X_2 - \bar{X}) \nonumber \\
    =& \mathbb{E} \|X_1 - \bar{X} \|^2+2\mathbb{E}(X_1 - \bar{X})(X_2 - \bar{X}) + \mathbb{E}\|X_2 - \bar{X} \|^2 \nonumber\\
    =& \mathbb{E} \|X_1 - \bar{X} \|^2+2\mathbb{E}(X_1 - \bar{X})(\bar{X}-X_2) + \mathbb{E}\| \bar{X}-X_2 \|^2 \nonumber\\
    =& \mathbb{E} \big \|X_1-\bar{X}+\bar{X}-X_2 \big \|^2
    \nonumber\\
    =&\mathbb{E} \big \|X_1-X_2 \big \|^2, \nonumber
\end{align}
where $\bar{X}$ denotes the mean and we use the facts that (i) $\mathbb{E}\|X_1-\bar{X} \|^2=\mathbb{E}\|X_2-\bar{X} \|^2=Var(X)$; (ii) $\mathbb{E}(X_1 -\bar{X})(X_2 - \bar{X})=0$ for i.i.d random variables.
\end{proof}

\secondmomentboundtwo*
\begin{proof}
Utilizing proposition \ref{var_comp}, we have
\begin{align}
&Var_{\xi}\big(\nabla_2 \mathcal{L}_{\zeta,\xi}(x, \eta )\big)\nonumber\\
=&\frac{1}{2}
\mathbb{E}_{\xi_1,\xi_2}\Big[(f^*)'\Big(\frac{\ell (x;\xi_1)-\lambda c(\zeta,\xi_1)- \eta}{\lambda \red{\beta}}\Big)-(f^*)'\Big(\frac{\ell\left(x;\xi_2\right)-\lambda c(\zeta,\xi_2)- \eta}{\lambda \red{\beta}}\Big)\Big]^2 \nonumber\\
\leq& \frac{1}{2}M^2(\lambda \red{\beta})^{-2} \mathbb{E}_{\xi_1,\xi_2}\big \| \ell(x;\xi_1) - \ell(x;\xi_2) - (\lambda c(\zeta,\xi_1)- \lambda c(\zeta,\xi_2))\big \|^2 \nonumber\\
\overset{(i)}{\leq} & \frac{1}{2}M^2(\lambda \red{\beta})^{-2} 2\cdot (\mathbb{E}_{\xi_1, \xi_2}\|\ell(x;\xi_1)-\ell(x;\xi_2) \|^2+\lambda^2\mathbb{E}_{\xi_1,\xi_2}\|c(\zeta,\xi_1)-c(\zeta,\xi_2) \|^2)\nonumber\\
\overset{(ii)}{\leq}& \frac{1}{2}M^2(\lambda \red{\beta})^{-2} 2(2Var\big(\ell(x;\xi)\big)+2Var_{\xi}\big(c(\zeta,\xi))\big) \nonumber\\
\overset{(iii)}{\leq}& 2M^2 (\lambda\red{\beta})^{-2} (\sigma^2 + \lambda^2\delta^2),\nonumber 
\end{align}
where (i) applies the fact that given any vectors $a$,$b$, we have $\| a-b\|^2 \leq 2\|a\|^2+2\|b\|^2$, (ii) applies proposition \ref{var_comp} and (iii) applies bounded variance assumptions \ref{assum2} for $\ell(\cdot,\xi)$ and $c(\zeta,\cdot)$.
\end{proof}

\section{Proof of Corollary \ref{thm: innersgd}} \label{proof: thm innersgd}
\innersgd*
\begin{proof}
    Notice that the objective function is $K'$-smooth, we have
    \begin{align}
           \mathcal{L}_{\zeta}(x_t,\eta_{x_t}^{d+1}(\zeta))&\overset{(i)}{\leq} \mathcal{L}_{\zeta}(x_t, \eta_{x_t}^{d}(\zeta)) - \langle \nabla_2 \mathcal{L}(x_t,\eta_{x_t}^{d}(\zeta)), \eta_{x_t}^{d}(\zeta)-\eta_{x_t}^{d+1}(\zeta) \rangle +\frac{K'}{2} |\eta_{x_t}^{d+1}(\zeta)-\eta_{x_t}^{d}(\zeta)|^2 \nonumber \\ 
           & \overset{(ii)}{=} \mathcal{L}_{\zeta}(x_t, \eta_{x_t}^{d}(\zeta)) - \langle \nabla_2 \mathcal{L}(x_t,\eta_{x_t}^{d}(\zeta)), \alpha_d v^{\tilde{B}}(\eta_{x_t}^{d}(\zeta)) \rangle +\frac{K'\alpha_d^2}{2} {|v^{\tilde{B}}(\eta_{x_t}^{d}(\zeta))|}^2, \nonumber
        \end{align}
        where (i) applies descent lemma for $K'$-smooth function; (ii) applies update rule $\eta_{x_t}^{d+1}(\zeta)=\eta_{x_t}^{d}(\zeta)-\alpha_dv^{\tilde{B}}(\eta_{x_t}^{d}(\zeta))$.
        Taking expectation over $\xi$ on both sides, we further obtain
        \begin{align}
            &\mathbb{E}_{\xi_{\tilde{B}}\sim\nu}[\mathcal{L}_{\zeta}(x_t, \eta_{x_t}^{d+1}(\zeta))]\nonumber\\
            &\leq \mathbb{E}_{\xi_{\tilde{B}}\sim \nu}[\mathcal{L}_{\zeta}(x_t, \eta_{x_t}^{d}(\zeta))] - \alpha_d\mathbb{E}_{\xi_{\tilde{B}}\sim\nu}|\nabla_2 \mathcal{L}_{\zeta}(x_t,\eta_{x_t}^{d}(\zeta)) |^2+\frac{K'\alpha_d^2}{2}\mathbb{E}_{\xi_{\tilde{B}}\sim \nu}|v^{\tilde{B}}(\eta_{x_t}^{d}(\zeta)) |^2 \nonumber\\
            &\overset{(i)}{\leq} \mathbb{E}_{\xi_{\tilde{B}}\sim \nu}[\mathcal{L}_{\zeta}(x_t, \eta_{x_t}^{d}(\zeta))] - \alpha_d \mathbb{E}_{\xi_{\tilde{B}}\sim\nu}|\nabla_2 \mathcal{L}_{\zeta}(x_t,\eta_{x_t}^{d}(\zeta)) |^2+\frac{K'R_2\alpha_d^2}{2\tilde{B}}+\frac{K'\alpha_d^2}{2}|\nabla_2\mathcal{L}_{\zeta}(x_t, \eta_{x_t}^{d}(\zeta)) |^2 \nonumber\\
            &=\mathbb{E}_{\xi_{\tilde{B}}\sim \nu}[\mathcal{L}_{\zeta}(x_t, \eta_{x_t}^{d}(\zeta))]-\alpha_d (1-\frac{K'\alpha_d}{2})\mathbb{E}_{\xi_{\tilde{B}}\sim\nu}|\nabla_2 \mathcal{L}_{\zeta}(x_t,\eta_{x_t}^{d}(\zeta)) |^2+\frac{K'\alpha_d^2R_2}{2}\nonumber\\
            &\overset{(ii)}{\leq} \mathbb{E}_{\xi_{\tilde{B}}\sim \nu}[\mathcal{L}_{\zeta}(x_t, \eta_{x_t}^{d}(\zeta))]-\frac{\alpha_d}{2}\mathbb{E}_{\xi_{\tilde{B}}\sim\nu}|\nabla_2 \mathcal{L}_{\zeta}(x_t,\eta_{x_t}^{d}(\zeta)) |^2+\frac{K'\alpha_d^2R_2}{2}
        \end{align}
        where (i) utilizes second moment upper bound stated in Lemma \ref{thm: boundedmoment2}; (ii) utilizes the fact $\alpha_d\leq \frac{1}{K'}$ and $\tilde{B}=1$.
        Since $\mathcal{L}_{\zeta}(x_t, \eta_{x_t}^{d}(\zeta))$ does not include randomness from $\xi$, re-organizing above inequality gives
        \begin{align}
            \frac{\alpha_d}{2}|\nabla_2 \mathcal{L}_{\zeta}(x_t, \eta_{x_t}^{d}(\zeta))|^2\leq \mathbb{E}_{\xi\sim\nu}[\mathcal{L}_{\zeta}(x_t,\eta_{x_t}^{d}(\zeta))-\mathcal{L}_{\zeta}(x_{t+1},\eta_{x_t}^{d}(\zeta))]+\frac{K'\alpha_d^2R_2}{2}
        \end{align}
        Summing above equation through $d\in\{0....D-1\}$, we have
        \begin{align}
            \frac{\alpha_d}{2D}\sum_{d=0}^{D-1}|\nabla_2 \mathcal{L}_{\zeta}(x_t, \eta_{x_t}^{d}(\zeta))|^2\leq \frac{\hat{\Delta}}{D}+\frac{K'\alpha_d^2R_2}{2}
        \end{align}
        To make sure the RHS smaller than or equal to $\tilde{\varepsilon}^2$, we need $\alpha_d \leq \frac{\tilde{\varepsilon}^2}{2K'R_2}$,
        and $D\geq\frac{\hat{8\Delta}K'R_2}{\tilde{\varepsilon}^4}$. Then, applying average argument, above inequality implies for any $\tilde{d}$ uniformly sampled from $\{0,....D-1\}$, we have
        \begin{align}
            \mathbb{E}_{\eta_{x}^{\tilde{d}}(\zeta)}|\nabla_2 \mathcal{L}_{\zeta}(x_t, \eta_{x}^{\tilde{d}}(\zeta))|^2=\frac{1}{D}\sum_{d=0}^{D-1}\mathbb{E}_{\eta_{x_t}^{d}(\zeta)}|\nabla_2 \mathcal{L}_{\zeta}(x_t, \eta_{x_t}^{d}(\zeta))|^2\leq \tilde{\varepsilon}^2,
        \end{align}
        which gives the desired result.
        %Above inequality further implies $\mathbb{E}_{\eta_{x}^{\tilde{d}}(\zeta)}|\nabla_2 \mathcal{L}_{\zeta}(x_t,\eta_{x}^{\tilde{d}}(\zeta))|\leq \tilde{\varepsilon}$.
\end{proof}

\section{Proof of Lemma \ref{thm:directionsmooth}}\label{proof: directional smooth thm}
\outersmooth*
\begin{proof}
    For any fixed $\zeta$, define the following two quantities.
    \begin{align}
    A &= \mathbb{E}_{\xi\sim \nu} \Big[ (f^{*})' \Big(\frac{\ell(x;\xi)-c(\xi;\zeta)-\eta_x^*(\zeta)}{\lambda \red{\beta}}\Big)\nabla \ell(x;\xi)- (f^{*})'\Big(\frac{\ell(x;\xi)-c(\xi;\zeta)-\eta_x^*(\zeta)}{\lambda \red{\beta}}\Big)\nabla \ell(x';\xi)\Big],\nonumber \\
    B &=  \mathbb{E}_{\xi\sim \nu} \Big[(f^{*})'\Big(\frac{\ell(x;\xi)-c(\xi;\zeta)-\eta_x^*(\zeta)}{\lambda \red{\beta}}\Big)\nabla \ell(x';\xi) - (f^{*})'\Big(\frac{\ell(x';\xi)-c(\xi;\zeta)-\eta_x^*(\zeta)}{\lambda \red{\beta}}\Big)\nabla \ell(x';\xi)\Big].\nonumber
    \end{align}
    It's easy to show that $A+B =\nabla\Psi_{\zeta}(x) -\nabla_1 \mathcal{L}_{\zeta}(x',\eta_x^*(\zeta))$. Our proof strategy is to bound $A$ and $B$ separately and combine the bounds together to give the results.

\noindent We first proceed upper bound of $A$. Note that $\eta_x^*(\zeta)\in \arg\min_{\eta}\mathcal{L}_{\zeta}(x,\eta)$, and the first-order optimality condition over $\eta$ gives that 
    \begin{align}
        1 - \mathbb{E}_{\xi \sim \nu}\Big[(f^{*})'\Big(\frac{\ell(x;\xi) - c(\xi;\zeta)) - \eta_x^*(\zeta)}{\lambda \red{\beta}}\Big)\Big]=0.
    \end{align}
    Moreover, note that the derivative of the conjugate function $f^*$ satisfies $(f^*)'(v)=r^*(v)=\arg \max_{r\in \mathbf{R}_+}\langle r, v\rangle - f(r)$. Thus, we must have $(f^{*})'(v)  \geq 0$, and the above equation further implies that
    %the maximization argument states that if the function $f^*$ is differentiable, then its derivative is the maximizing argument in the computation of the convex conjugate. By definition of Sinkhorn distance, $f$ is a convex function defined from $\big[0,+\infty\big)$ to $\big[-\infty, +\infty \big]$, which immediately yields the fact $f^{**}=f$. The derivative $(f^{*})'(v)$ at $v$ can be computed via $(f^*)'(v)=\arg \max_{r}\langle r, v\rangle - f(v)$. For primal variable $r$, it is defined over $\big[0,+\infty)$, thus we must have $(f^{*})'(v)  \geq 0$, {\color{red} still unclear, $r(v)$ not defined, and notation of $r$ is misleading},  we further obtain that 
    \begin{align}
    \mathbb{E}_{\xi \sim \nu}\Big |(f^{*})'\Big(\frac{\ell(x;\xi)-c(\xi;\zeta)-\eta_x^*(\zeta)}{\lambda \red{\beta}}\Big)\Big | = \Big |\mathbb{E}_{\xi \sim \nu}\Big[ (f^{*})'\Big(\frac{\ell(x;\xi)-c(\xi;\zeta)-\eta_x^*(\zeta)}{\lambda \red{\beta}}\Big)\Big] \Big |=1.
    \label{max_argument}
    \end{align}
    Consequently, using $L$-smoothness of $\ell(x;\xi)$ and the equality \eqref{max_argument}, we conclude that
    \begin{align}
        \big \| A \big \|\leq L \mathbb{E}_{\xi}\Big |f^{*'}\Big(\frac{\ell(x;\xi)-c(\xi;\zeta)-\eta_x^*(\zeta)}{\lambda \red{\beta}}\Big) \Big | \big \| x-x' \big \| \leq L \big \|x-x' \big \|.
        \label{Bound4A}
    \end{align}
    Next, we bound $B$. By $G$-Lipschitz continuity of $\ell(x;\xi)$ and $M$-smoothness of $(f^{*})'$, we have
    \begin{align}
        \big\| &B \big \| \nonumber \\
        \leq& \mathbb{E}_{\xi} \big \|  \Big [(f^{*})'\Big(\frac{\ell(x;\xi)- \lambda c(\zeta,\xi)-\eta_x^*(\zeta)}{\lambda\red{\beta}}\Big) -(f^{*})'\Big(\frac{\ell(x';\xi)- \lambda c(\zeta,\xi)-\eta_x^*(\zeta)}{\lambda\red{\beta}}\Big)\Big]\nabla \ell(x';\xi) \big \| \nonumber\\
        \leq& \mathbb{E}_{\xi} \big \|  (f^{*})'\Big(\frac{\ell(x;\xi)- \lambda c(\zeta,\xi)-\eta_x^*(\zeta)}{\lambda\red{\beta}}\Big) -(f^{*})'\Big(\frac{\ell(x';\xi)- \lambda c(\zeta,\xi)-\eta_x^*(\zeta)}{\lambda\red{\beta}}\Big) \big \| \big \| \nabla \ell(x';\xi) \big \| \nonumber\\
        \leq& G \mathbb{E}_{\xi} \big \|  (f^{*})'\Big(\frac{\ell(x;\xi)- \lambda c(\zeta,\xi)-\eta_x^*(\zeta)}{\lambda\red{\beta}}\Big) -(f^{*})'\Big(\frac{\ell(x';\xi)- \lambda c(\zeta,\xi)-\eta_x^*(\zeta)}{\lambda\red{\beta}}\Big) \big \| \nonumber\\
        \leq&  G M (\lambda \red{\beta})^{-1} \mathbb{E}_{\xi} \big \|  \ell(x;
        \xi) - \ell(x';\xi) \big \| \nonumber \\
        \leq& (\lambda \red{\beta})^{-1} MG^2 \big \| x - x' \big \|.
        \label{Bound4B}
    \end{align}
    Combining \eqref{Bound4A} and \eqref{Bound4B}, we obtain that
    \begin{align}
    \big \| \nabla \Psi_{\zeta}(x)-\nabla \mathcal{L}_{\zeta}(x',\eta_{x}^*(\zeta))   \big \| \leq \big \| A\big \| + \big \| B \big \| \leq (G^2(\lambda \red{\beta})^{-1}M+L)\big \| x-x' \big \|. \nonumber
    \end{align}
    Squaring both sides and taking expectation over $\zeta$ gives the claimed result.
\end{proof}
\section{Proof of Lemma \ref{thm: boundedmoment3}}\label{proof: outer gradient bounded moment thm}
%\noindent Before proving Lemma \ref{thm: boundedmoment3}, we first prove following lemma \ref{boundedmoment1} where $\eta_{x}(\zeta)^*$ is available.

%\begin{lemma}[Bounded Second Moment]\label{boundedmoment1}
%    Let Assumption \ref{assum2} hold. For the mini-batch stochastic gradient estimator,
%    \begin{align}
%           \hat{g}_t^B= \frac{1}{B} \sum_{i=1}^{B} (f^*)'\Bigr(\frac{\ell(x;\xi_i)-c(\zeta,\xi_i)-\eta_{x}^{\tilde{d}}(\zeta)_{x_t}(\zeta)}{\lambda \red{\beta}}\Bigr) \nabla \ell(x_t;\xi_i).\nonumber
%    \end{align}
%    its second moment is bounded by
%    \begin{align}
%    \mathbb{E}_{\zeta \sim \mathbb{P},\xi_{\tilde{B}}\sim \nu, \eta_{x}^{\tilde{d}}(\zeta)} \big[\| g_t^B \|^2 \big]
%    \leq R_1+ \big \| \nabla_{1} \mathbb{E}_{\zeta \sim \mathbb{P},\eta_{\tilde{d}}}\bigr[\mathcal{L}_{\zeta}(x_t,\eta_{x}^{\tilde{d}}(\zeta)_{x_t}(\zeta))\bigr]\big \|^2,
%    \label{bounded var1}
%    \end{align}
%    where $R_1=\frac{8G^2+8G^2\varepsilon^2+20G^2M^2(\lambda \red{\beta})^{-2}\sigma^2+20G^2M^2\red{\beta}^{-2}\delta^2}{B} + G^2M^2\red{\beta}^{-2} \delta^2$. \\
%\end{lemma}
\secondmomentboundthree*
\begin{proof}
    Throughout, given $x$, $\eta$,
    we denote $\mathcal{L}_{\zeta,\xi}(x,\eta) = f^*(\frac{\ell(x;\xi)-c(\zeta,\xi)-\eta}{\lambda \red{\beta}})$ to simplify notation. Notice that the outputs from Algorithm \ref{alg1}, $\tilde{d}$ also contains randomness, we use $\mathbb{E}_{\eta_{x}^{\tilde{d}}(\zeta)}(\cdot)$ to denote expectation over $\eta_{x}^{\tilde{d}}(\zeta)$ through the proof.
    Then, for $\mathbb{E}_{\zeta\sim \mathbb{P},\xi\sim \mathbb{\nu},\eta_{x}^{\tilde{d}}(\zeta)} \big \|\nabla_1 \mathcal{L}_{\zeta,\xi}(x,\eta_x^{\tilde{d}}(\zeta)\big\|^2$,
    we decompose it as follows
\begin{align}
    & \mathbb{E}_{(\zeta\sim \mathbb{P},\eta_{x}^{\tilde{d}}(\zeta)),\xi\sim \mathbb{\nu}} \big \| \nabla_1 \mathcal{L}_{\zeta,\xi}(x,\eta_x^{\tilde{d}}(\zeta))\big \|^2 \nonumber\\
    =& \mathbb{E}_{(\zeta \sim \mathbb{P},\eta_{x}^{\tilde{d}}(\zeta))}\big[\mathbb{E}_{\xi \sim \mathbb{\nu}}\big \| \nabla_1 \mathcal{L}_{\zeta,\xi}(x,\eta_x^{\tilde{d}}(\zeta))\big \|^2\big] \nonumber\\
    \overset{(i)}{=}& \mathbb{E}_{(\zeta \sim \mathbb{P},\eta_{x}^{\tilde{d}}(\zeta))}\big[Var_{\xi}\big(\nabla_1 \mathcal{L}_{\xi,\zeta}(x,\eta_x^{\tilde{d}}(\zeta))\big)+ \big \|\nabla_1 \mathcal{L}_{\zeta}(x,\eta_x^{\tilde{d}}(\zeta)) \big \|^2\big]\nonumber\\
    =& \mathbb{E}_{(\zeta \sim \mathbb{P},\eta_{x}^{\tilde{d}}(\zeta))}\big[Var_{\xi}\big(\nabla_1 \mathcal{L}_{\xi,\zeta}(x,\eta_x^{\tilde{d}}(\zeta))\big)\big] + \mathbb{E}_{(\zeta\sim \mathbb{P},\eta_{x}^{\tilde{d}}(\zeta))} \big\|\nabla_1 \mathcal{L}_{\zeta}(x,\eta_x^{\eta_{x}^{\tilde{d}}(\zeta)}(\zeta)) \big \|^2 \nonumber \\
    \overset{(ii)}{=}& \mathbb{E}_{(\zeta \sim \mathbb{P},\eta_{x}^{\tilde{d}}(\zeta))}\big[Var_{\xi}\big(\nabla_1 \mathcal{L}_{\xi,\zeta}(x,\eta_x^{\tilde{d}}(\zeta))\big)\big] +\big[Var_{(\zeta,\eta_{x}^{\tilde{d}}(\zeta))}\big(\nabla_1 \mathcal{L}_{\zeta}(x,\eta_x^{\tilde{d}}(\zeta))\big) \big ] \nonumber\\
    &\quad+ \big[\big \| \nabla_1 \mathbb{E}_{\zeta\sim \mathbb{P},\eta_{x}^{\tilde{d}}(\zeta)}\big[\mathcal{L}_{\zeta}(x,\eta_x^{\tilde{d}}(\zeta))\big]\big \|^2\big], \nonumber
    %=& \mathbb{E}_{\zeta\sim \mathbb{P},\eta_{x}^{\tilde{d}}(\zeta)}\big[Var_{\xi}\big(\nabla_1 \mathcal{L}_{\xi,\zeta}(x,\eta_x^{\tilde{d}}(\zeta))\big)\big] + \mathbb{E}_{\eta_{x}^{\tilde{d}}(\zeta)}\big[Var_{\zeta}\big(\nabla_1 \mathcal{L}_{\zeta}(x,\eta_x^{\tilde{d}}(\zeta))\big)\big] + \mathbb{E}_{\eta_{x}^{\tilde{d}}(\zeta)}\big[\big \| \nabla \mathbb{E}_{\zeta\sim \mathbb{P}}\big[\mathcal{L}_{\zeta}(x,\eta_{x}^{\tilde{d}}(\zeta))\big]\big \|^2\big], \nonumber
\end{align}
where $(i)$ and $(ii)$ apply equation $\mathbb{E}\bigr[\rho(\varpi)^2\bigr]=Var_{\varpi}(\rho(\varpi))+(\mathbb{E}\bigr[\rho(\varpi)\bigr])^2$, which holds for any random variable $\varpi$ and $\rho(\cdot):\mathbf{R}^d\rightarrow \mathbf{R}$.\\
Next, we bound $\mathbb{E}_{\zeta\sim\mathbb{P}, \eta_{x}^{\tilde{d}}(\zeta)}\big[Var_{\xi}\big(\nabla_1 \mathcal{L}_{\xi,\zeta}(x,\eta_x^{\tilde{d}}(\zeta))\big)\big]$ and $Var_{(\zeta,\eta_x^{\tilde{d}}(\zeta))}\big(\nabla_1 \mathcal{L}_{\zeta}(x,\eta_x^{\tilde{d}}(\zeta))\big)$ separately. \\
For $Var_{\xi}\big(\nabla_1 \mathcal{L}_{\xi,\zeta}(x,\eta_x^{\tilde{d}}(\zeta))\big)$, we have
\begin{align}
    &Var_{\xi}\big(\nabla_1 \mathcal{L}_{\zeta,\xi}(x,\eta_x^{\tilde{d}}(\zeta))\big) \nonumber\\
    \overset{(i)}{=}& \frac{1}{2} \mathbb{E}_{\xi_1,\xi_2}\big \| (f^{*})'\Big ( \frac{\ell(x;\xi_1)-\lambda c(\zeta,\xi_1)-\eta_x^{\tilde{d}}(\zeta)}{\lambda \red{\beta}}  \Big )\nabla \ell(x;\xi_1) \nonumber\\
    &\qquad -(f^{*})' \Big ( \frac{\ell(x;\xi_2)-\lambda c(\zeta,\xi_2)-\eta_x^{\tilde{d}}(\zeta)}{\lambda \red{\beta}}  \Big )
    \nabla \ell(x;\xi_2)\big \|^2 \nonumber\\
    \overset{}{=}& \frac{1}{2} \mathbb{E}_{\xi_1,\xi_2}\big \| (f^{*})'\Big ( \frac{\ell(x;\xi_1)-\lambda c(\zeta,\xi_1)-\eta_x^{\tilde{d}}(\zeta)}{\lambda \red{\beta}}  \Big )\nabla \ell(x;\xi_1)\nonumber\\ 
    &\qquad -(f^{*})'\Big ( \frac{\ell(x;\xi_1)-\lambda c(\zeta,\xi_1)-\eta_x^{\tilde{d}}(\zeta)}{\lambda \red{\beta}}  \Big )\nabla \ell(x;\xi_2) \nonumber \\
    &\qquad+(f^{*})' \Big ( \frac{\ell(x;\xi_1)-\lambda c(\zeta,\xi_1)-\eta_x^{\tilde{d}}(\zeta)}{\lambda \red{\beta}}  \Big )\nabla \ell(x;\xi_2)\nonumber \\
    &\qquad-(f^{*})' \Big ( \frac{\ell(x;\xi_2)-\lambda c(\zeta,\xi_2)-\eta_x^{\tilde{d}}(\zeta)}{\lambda \red{\beta}}  \Big ) \nabla \ell(x;\xi_2)\big \|^2 \nonumber\\
    \overset{(ii)}{\leq}& \mathbb{E}_{\xi_1,\xi_2}\Big[ (f^{*})'\Big ( \frac{\ell(x;\xi_1)-\lambda c(\zeta,\xi_1)-\eta_x^{\tilde{d}}(\zeta)}{\lambda \red{\beta}} \Big )^2\big \| \nabla \ell(x;\xi_1) - \nabla \ell(x;\xi_2) \big \|^2 \Big ]\nonumber \\
    &\qquad+\mathbb{E}_{\xi_1,\xi_2} \Big [\big \| \nabla \ell(x;\xi_2) \big \|^2 \cdot \Big ( (f^{*})'\Big ( \frac{\ell(x;\xi_1)-\lambda c(\zeta,\xi_1)- \eta_x^{\tilde{d}}(\zeta)}{\lambda \red{\beta}} \Big) \nonumber\\
    & \quad \qquad - (f^{*})'\Big ( \frac{\ell(x;\xi_2)-\lambda c(\zeta,\xi_2) - \eta_x^{\tilde{d}}(\zeta)}{\lambda \red{\beta}} \Big ) \Big )^2 \Big ]\nonumber\\
    \overset{(iii)}{\leq}& 4G^2 \mathbb{E}_{\xi_1}\Big [ (f^{*})'\Big ( \frac{\ell(x;\xi_1)-\lambda c(\zeta,\xi_1)-\eta_x^{\tilde{d}}(\zeta)}{\lambda \red{\beta}} \Big )^2 \Big ] \nonumber \\
    & \quad +G^2M^2(\lambda \red{\beta})^{-2}\mathbb{E}_{\xi_1,\xi_2}\big [ \ell(x;\xi_1)-\lambda c(\zeta,\xi_1) - \ell(x;\xi_2)+\lambda c(\zeta,\xi_2) \big ]^2 \nonumber \\
    \overset{(iv)}{\leq}& 4G^2 \mathbb{E}_{\xi_1}\Big [ (f^{*})'\Big ( \frac{\ell(x;\xi_1)-\lambda c(\zeta,\xi_1)-\eta_x^{\tilde{d}}(\zeta)}{\lambda \red{\beta}} \Big )^2 \Big ]+2G^2M^2(\lambda \red{\beta})^{-2}(2\sigma^2+2\lambda^2 \delta^2), 
    \label{1st_bound_of_var}
\end{align}
where (i) applies proposition \ref{var_comp};
(ii) applies the fact $(a+b)^2 \leq 2(a^2+b^2)$; (iii) applies the assumption G-Lipschitz continuity of $\ell(x;\xi)$ function and M-smoothness of $f^*$; (iv) applies bounded variance assumption of $\ell(x,\xi)$ and $c(\zeta,\xi)$. \\
Applying inequality $a^2 \leq 2(a-1)^2+2$, the term in \eqref{1st_bound_of_var},
$\mathbb{E}_{\xi_1}\Big [ (f^{*})'\big ( \frac{\ell(x;\xi_1)-\lambda c(\zeta,\xi_1)-\eta_x^{\tilde{d}}(\zeta)}{\lambda \red{\beta}} \big )^2 \Big ]$ can be further upper bounded as
\begin{align}
    &\mathbb{E}_{\xi_1}\Big [ (f^{*})'\Big( \frac{\ell(x;\xi_1)-\lambda c(\zeta,\xi_1)-\eta_x^{\tilde{d}}(\zeta)}{\lambda \red{\beta}} \Big )^2 \Big ] \nonumber \\
    &\leq 2 + 2\mathbb{E}_{\xi_1}\Big [ 1- (f^{*})'\Big( \frac{\ell(x;\xi)-\lambda c(\zeta,\xi_1)- \eta_x^{\tilde{d}}(\zeta)}{\lambda \red{\beta}} \Big )  \Big ]^2 \nonumber\\
    &\leq 2\Big(1+\big | \nabla_{2} \mathcal{L}_{\zeta}(x,\eta_x^{\tilde{d}}(\zeta))\big |^2 + Var_{\xi}\big(\nabla_{2} \mathcal{L}_{\zeta,\xi}(x,\eta_x^{\tilde{d}}(\zeta))\big)\Big). 
    \label{2nd_bound_of_var}
\end{align}
Since at $\eta_{x}^{\tilde{d}}(\zeta)$, by \eqref{eq: inner near-optima} stated in Theorem \ref{thm: innersgd}, we conclude $\mathbb{E}_{\eta_{x}^{\tilde{d}}(\zeta)}|\nabla_2 \mathcal{L}_{\zeta}(x,\eta_x^{\tilde{d}}(\zeta))|^2 \leq \tilde{\varepsilon}^2={\varepsilon^2}/{G^2}$.
Moreover, from lemma \ref{thm: boundedmoment2}, we have
\begin{align}
Var_{\xi}\big(\nabla_2 \mathcal{L}_{\zeta,\xi}(x,\eta_x^{\tilde{d}}(\zeta))\big) \leq 2M^2(\lambda \red{\beta})^{-2}(\sigma^2+\lambda^2\delta^2).
\label{3rd_bound_of_var}
\end{align}
Thus, combining inequalities \eqref{1st_bound_of_var} \eqref{2nd_bound_of_var} \eqref{3rd_bound_of_var}, we have
\begin{align}
    &\mathbb{E}_{\zeta\sim\mathbb{P},\eta_{x}^{\tilde{d}}(\zeta)}\big[Var_{\xi}\big(\nabla_1 L_{\zeta,\xi}(x,\eta_x^{\tilde{d}}(\zeta))\big)\big]\nonumber\\
    \leq& \mathbb{E}_{\zeta\sim\mathbb{P},\eta_{x}^{\tilde{d}}(\zeta)}\big[4G^2\cdot 2(1+\big | \nabla_2 L_{\zeta}(x,\eta_x^{\tilde{d}}(\zeta))\big |^2+2M^2(\lambda \red{\beta})^{-2}(\sigma^2+\lambda^2\delta^2))\nonumber\\
    &\qquad +2G^2M^2(\lambda \red{\beta})^{-2}(2\sigma^2+2\lambda^2 \delta^2) \big]\nonumber\\
    \leq & 8G^2+8G^2\tilde{\varepsilon}^2+20G^2M^2(\lambda \red{\beta})^{-2}\sigma^2+ 20G^2M^2\red{\beta}^{-2}\delta^2\nonumber\\
    =& 8G^2+8\varepsilon^2+20G^2M^2(\lambda \red{\beta})^{-2}\sigma^2+ 20G^2M^2\red{\beta}^{-2}\delta^2.
    \label{5th_bound_of_var}
\end{align}
For $Var_{(\zeta, \eta_{x_t}^{\tilde{d}}(\zeta))}\big(\nabla_1 \mathcal{L}_{\zeta}(x,\eta_x^{\tilde{d}}(\zeta))\big)$,
we have 
\begin{align}
&Var_{(\zeta, \eta_{x}^{\tilde{d}}(\zeta))}\big ( \nabla_1 \mathcal{L}_{\zeta}(x,\eta_x^{\tilde{d}}(\zeta)) \big ) \nonumber\\
=& \frac{1}{2}\mathbb{E}_{(\zeta_1,\eta_{x}^{\tilde{d}}(\zeta_1)),(\zeta_2,\eta_x^{\tilde{d}}(\zeta_2))} \big \| \mathbb{E}_{\xi} \Big [\Big (   (f^{*})'\Big ( \frac{\ell(x;\xi)-\lambda c(\zeta_1,
\xi)-\eta_{x}^{\tilde{d}}(\zeta_1)}{\lambda \red{\beta}} \Big ) \nonumber\\
&\qquad -(f^{*})'\Big ( \frac{\ell(x;\xi)-\lambda c(\zeta_2,
\xi)-\eta_{x}^{\tilde{d}}(\zeta_2)}{\lambda \red{\beta}} \Big ) \Big )\cdot\nabla \ell(x;\xi) \Big ] \big \|^2 \nonumber\\
\overset{(i)}{\leq}& \frac{1}{2}\mathbb{E}_{(\zeta_1,\eta_{x}^{\tilde{d}}(\zeta_1)),(\zeta_2,\eta_x^{\tilde{d}}(\zeta_2)),\xi} \big \| \Big [   (f^{*})'\Big ( \frac{\ell(x;\xi)-\lambda c(\zeta_1,
\xi)-\eta_{x}^{\tilde{d}}(\zeta_1)}{\lambda \red{\beta}} \Big ) \nonumber\\
&\qquad - (f^{*})'\Big ( \frac{\ell(x;\xi)-\lambda c(\zeta_2,
\xi)-\eta_{x}^{\tilde{d}}(\zeta_2)}{\lambda \red{\beta}} \Big ) \Big ]\cdot \nabla \ell(x;\xi)  \big \|^2 \nonumber\\
\overset{(ii)}{=}& \frac{1}{2}\mathbb{E}_{(\zeta_1,\eta_{x}^{\tilde{d}}(\zeta_1)),(\zeta_2,\eta_x^{\tilde{d}}(\zeta_2)),\xi} |   (f^{*})'\Big ( \frac{\ell(x;\xi)-\lambda c(\zeta_1,
\xi)-\eta_{x}^{\tilde{d}}(\zeta_1)}{\lambda \red{\beta}} \Big ) \nonumber \\
&\qquad - (f^{*})'\Big ( \frac{\ell(x;\xi)-\lambda c(\zeta_2,
\xi)-\eta_{x}^{\tilde{d}}(\zeta_2)}{\lambda \red{\beta}} \Big ) |^2 \big \|\nabla \ell(x;\xi) \big \|^2 \nonumber\\
\overset{(iii)}{\leq}&\frac{1}{2}G^2 \mathbb{E}_{(\zeta_1,\eta_{x}^{\tilde{d}}(\zeta_1)),(\zeta_2,\eta_x^{\tilde{d}}(\zeta_2)),\xi} |   (f^{*})'\Big ( \frac{\ell(x;\xi)-\lambda c(\zeta_1,
\xi)-\eta_{x}^{\tilde{d}}(\zeta_1)}{\lambda \red{\beta}} \Big )-1 \nonumber \\
&\qquad - (f^{*})'\Big ( \frac{\ell(x;\xi)-\lambda c(\zeta_2,
\xi)-\eta_{x}^{\tilde{d}}(\zeta_2)}{\lambda \red{\beta}} \Big )+1 |^2 \nonumber\\
\overset{(iv)}{\leq}&G^2\mathbb{E}_{(\zeta_1,\eta_{x}^{\tilde{d}}(\zeta_1)),(\zeta_2,\eta_x^{\tilde{d}}(\zeta_2))}\Big( \mathbb{E}_{\xi}|\nabla_2 \mathcal{L}_{\zeta_1,\xi}(x;\eta_{x}^{\tilde{d}}(\zeta_1))|^2+\mathbb{E}_{\xi}|\nabla_2 \mathcal{L}_{\zeta_2,\xi}(x;\eta_{x}^{\tilde{d}}(\zeta_2))|^2\Big) \nonumber\\
\overset{(v)}{=}& G^2\mathbb{E}_{(\zeta_1,\eta_{x}^{\tilde{d}}(\zeta_1)),(\zeta_2,\eta_x^{\tilde{d}}(\zeta_2))}\Big( Var_{\xi}(\nabla_2\mathcal{L}_{\zeta_1,\xi}(x;\eta_{x}^{\tilde{d}}(\zeta_1)))+|\nabla_2 \mathcal{L}_{\zeta_1}(x;\eta_{x}^{\tilde{d}}(\zeta_1))|^2\nonumber \\
&\qquad+Var_{\xi}(\nabla_2\mathcal{L}_{\zeta_2,\xi}(x;\eta_{x}^{\tilde{d}}(\zeta_2)))+|\nabla_2 \mathcal{L}_{\zeta_2}(x;\eta_{x}^{\tilde{d}}(\zeta_2))|^2\Big)\nonumber\\
\overset{(vi)}{\leq}&G^2\Big(4M^2(\lambda \red{\beta})^{-2}(\sigma^2+\lambda\delta^2)+\mathbb{E}_{\zeta_1,\eta_x^{\tilde{d}}(\zeta_1)}|\nabla_2 \mathcal{L}_{\zeta_2}(x;\eta_{x}^{\tilde{d}}(\zeta_1))|^2+\mathbb{E}_{\zeta_2,\eta_x^{\tilde{d}}(\zeta_2)}|\nabla_2 \mathcal{L}_{\zeta_2}(x;\eta_{x}^{\tilde{d}}(\zeta_2))|^2\Big)\nonumber\\
\overset{(vii)}{\leq}&4G^2M^2(\lambda \red{\beta})^{-2}(\sigma^2+\lambda\delta^2)+2G^2\tilde{\varepsilon}^2\nonumber\\
=& 4G^2M^2(\lambda\red{\beta})^{-2}\sigma^2+4G^2M^2\red{\beta}^{-2}\delta^2+2\varepsilon^2
%\overset{(iii)}{\leq}& \frac{1}{2}G^2M^2(\lambda\red{\beta})^{-2} \mathbb{E}_{\zeta_1,\zeta_2,\xi}\big \| \lambda c(\zeta_1,\xi)-\lambda c(\zeta_2,\xi) \big \|^2 \nonumber\\
%=& \frac{1}{2}G^2M^2(\lambda\red{\beta})^{-2}\cdot 2\lambda^2\mathbb{E}_{\xi}\Big[ Var_{\zeta}(c(\zeta,\xi))\Big] \nonumber\\
%\overset{(iv)}{\leq}& G^2M^2\red{\beta}^{-2} \delta^2,
\label{6th_bound_of_var}
\end{align}
where (i) applies Jensen's inequality move expectation over $\xi$ out squared-norm;
(ii) extracts scalar $ |   (f^{*})'\Big ( \frac{\ell(x;\xi)-\lambda c(\zeta_1,
\xi)-\eta_{x}^{\tilde{d}}(\zeta)}{\lambda \red{\beta}} \Big ) - (f^{*})'\Big ( \frac{\ell(x;\xi)-\lambda c(\zeta_2,
\xi)-\eta_{x}^{\tilde{d}}(\zeta)}{\lambda \red{\beta}} \Big ) |$ out; (iii) applies $G$-Lipschitz assumption of $\ell(\cdot,\xi)$ stated at assumption \ref{assum1}. (iv) applies inequality $(a+b)^2\leq 2a^2+2b^2$ to decouple $|\nabla_2 \mathcal{L}_{\zeta_1,\xi}(x;\eta_x^{\tilde{d}}
(\zeta_1))|$ and $|\nabla_2 \mathcal{L}_{\zeta_2,\xi}(x;\eta_x^{\tilde{d}}
(\zeta_2))|$; (v) applies equality $\mathbb{E}\bigr[\rho(\varpi)^2\bigr]=Var_{\varpi}(\rho(\varpi))+(\mathbb{E}\bigr[\rho(\varpi)\bigr])^2$, which holds for any random variable $\varpi$ and $\rho(\cdot):\mathbf{R}^d\rightarrow \mathbf{R}$; (v) and (vi) apply Lemma \ref{thm: boundedmoment2} to upper bound $Var_{\xi}(\nabla_2 \mathcal{L}_{\zeta_1,\xi}(x;\eta_x^{\tilde{d}}(\zeta_1)))$, $Var_{\xi}(\nabla_2 \mathcal{L}_{\zeta_2,\xi}(x;\eta_x^{\tilde{d}}(\zeta_2)))$; (vii) applies \eqref{eq: inner near-optima} stated in Theorem \ref{thm: innersgd} to upper bound $\mathbb{E}_{\eta_{x}^{\tilde{d}}(\zeta_1)}|\nabla_2 \mathcal{L}_{\zeta_2}(x;\eta_x^{\tilde{d}}(\zeta_1))|^2$ and $\mathbb{E}_{\eta_{x}^{\tilde{d}}(\zeta_2)}|\nabla_2 \mathcal{L}_{\zeta_2}(x;\eta_x^{\tilde{d}}(\zeta_2))|^2$ by $\tilde{\varepsilon}^2={\varepsilon^2}/{G^2}$. \\
Combining \eqref{5th_bound_of_var} \eqref{6th_bound_of_var}, we have
\begin{align}
    & \mathbb{E}_{\zeta \sim \mathbb{P},\eta_x^{\tilde{d}}(\zeta),\xi \sim \mathbb{\nu}} \big \| \nabla_1 \mathcal{L}_{\zeta,\xi}(x,\eta_x^{\tilde{d}}(\zeta))\big \|^2 \nonumber\\
    %=& \mathbb{E}_{\zeta\sim \mathbb{P}}\big[\mathbb{E}_{\xi\sim \mathbb{\nu}}\big \| \nabla_1 \mathcal{L}_{\zeta,\xi}(x,\eta_x^{\tilde{d}}(\zeta))\big \|^2\big] \nonumber\\
    %=& \mathbb{E}_{\zeta\sim \mathbb{P}}\big[Var_{\xi}(\nabla_1 \mathcal{L}_{\zeta,\xi}\big(x,\eta_x^{\tilde{d}}(\zeta))\big)+ \big \|\nabla_1 \mathcal{L}_{\zeta}(x,\eta_x^{\tilde{d}}(\zeta)) \big \|^2)\big]\nonumber\\
    %= & \mathbb{E}_{\zeta\sim \mathbb{P}}\big[Var_{\xi}(\nabla_1 \mathcal{L}_{\zeta,\xi}\big(x,\eta_x^{\tilde{d}}(\zeta))\big)\big] + \mathbb{E}_{\zeta\sim \mathbb{P}} \big \|\nabla_1 \mathcal{L}_{\zeta}(x,\eta_x^{\tilde{d}}(\zeta)) \big \|^2 \nonumber\\
    %= & \mathbb{E}_{\zeta\sim \mathbb{P}}\big[Var_{\xi}\big(\nabla_1 \mathcal{L}_{\zeta,\xi}(x,\eta_x^{\tilde{d}}(\zeta))\big)\big] + Var_{\zeta}\big(\nabla_1 \mathcal{L}_{\zeta}(x,\eta_x^{\tilde{d}}(\zeta))\big) + \big \| \nabla_1 \mathcal{L}(x,\eta_x^{\tilde{d}}(\zeta))\big \|^2 \nonumber\\
    \leq& \mathbb{E}_{(\zeta \sim \mathbb{P},\eta_{x}^{\tilde{d}}(\zeta))}\big[Var_{\xi}\big(\nabla_1 \mathcal{L}_{\xi,\zeta}(x,\eta_x^{\tilde{d}}(\zeta))\big)\big] +\big[Var_{(\zeta,\eta_{x}^{\tilde{d}}(\zeta))}\big(\nabla_1 \mathcal{L}_{\zeta}(x,\eta_x^{\tilde{d}}(\zeta))\big) \big ] \nonumber\\
    &\qquad + \big[\big \| \nabla_1 \mathbb{E}_{\zeta\sim \mathbb{P},\eta_{x}^{\tilde{d}}(\zeta)}\big[\mathcal{L}_{\zeta}(x,\eta_x^{\tilde{d}}(\zeta))\big]\big \|^2\big], \nonumber \\
    \leq & 8G^2+10\varepsilon^2+24G^2M^2(\lambda \red{\beta})^{-2}\sigma^2+24G^2M^2\red{\beta}^{-2}\delta^2+\big \| \nabla_1 \mathbb{E}_{\zeta,\eta_x^{\tilde{d}}(\zeta)} \mathcal{L}(x,\eta_x^{\tilde{d}}(\zeta))\big \|^2. \label{final_bound_of_var}
\end{align}
For mini-batch stochastic gradient estimator \eqref{gradient3},
the RHS of \eqref{final_bound_of_var} becomes
\begin{align}
    & \mathbb{E}_{\zeta\sim \mathbb{P},\eta_{x}^{\tilde{d}}(\zeta),\xi_B\sim \nu} \big \| \nabla_1 \mathcal{L}_{\zeta,\xi}(x,\eta_x^{\tilde{d}}(\zeta))\big \|^2 \nonumber \\
    \leq &\frac{8G^2+24G^2M^2(\lambda \red{\beta})^{-2}\sigma^2+24G^2M^2\red{\beta}^{-2}\delta^2}{B}+\frac{10\varepsilon^2}{B}+\big \| \nabla_1 \mathbb{E}_{\zeta, \eta_{x}^{\tilde{d}}(\zeta)}[\mathcal{L}(x,\eta_x^{\tilde{d}}(\zeta))]\big \|^2, \nonumber
\end{align}
which gives the desired result.
\end{proof}

%\begin{proof}
%    The proof directly follows the proof developed of Lemma \ref{boundedmoment1}, by replacing $\big | \nabla_2 \mathcal{L}_{\zeta}(x,\eta_x^*(\zeta))\big|^2$ in \eqref{2nd_bound_of_var} with $\big |\nabla_2 \mathcal{L}_{\zeta}(x,\eta_{x}^{\tilde{d}}(\zeta)) \big|^2\leq \red{\beta}^2$ derived in \eqref{optimalrelation}.
%\end{proof}

%\subsection{Proof of Corollary \ref{corollary8}}

%\begin{proof}
%    If the function is K-smooth and the variance of its unbiased stochastic gradient estimator is upper bounded by $R_1$, following Corollary 2.2 in \citep{ghadimi2013stochastic}, the iteration complexity using SGD for finding an $\tilde{\delta}$-stationary point is $\mathcal{O}(\Delta K R_1\tilde{\delta}^{-4})$ .
%\end{proof}
\section{Proof of Theorem \ref{thm: convergenestsgd}}\label{proof: nested SGD convergence thm}
\mainconvergence*
\allowdisplaybreaks
\begin{proof}
   Regarding the objective function $\mathbb{E}_{\zeta \sim \mathbb{P}}\big[\Psi_{\zeta}(x)\big]$, we have the following descent lemma
    \begin{align}   
    \mathbb{E}_{\zeta\sim \mathbb{P}}\big[\Psi_{\zeta}(x_{t+1})\big]
        \leq \mathbb{E}_{\zeta\sim \mathbb{P}}\big[\Psi_{\zeta}(x_t)\big] + \langle \nabla \mathbb{E}_{\zeta\sim \mathbb{P}}\big[\Psi_{\zeta}(x_t)\big],x_{t+1}-x_t \rangle + \frac{K\gamma_t^2}{2} \big \| x_{t+1}-x_t\big \|^2.
        \label{descentlemma}
    \end{align}
    The proof of \eqref{descentlemma} is provided in Appendix \ref{pfdescentlemma}. 
    Replace $x_{t+1} - x_t $ by $-\gamma_t
    \hat{g}^{B}_t$, above inequality implies
    \begin{align}
        \mathbb{E}_{\zeta\sim \mathbb{P}}\big[\Psi_{\zeta}(x_{t+1})\big]
        \leq \mathbb{E}_{\zeta\sim \mathbb{P}}\big[\Psi_{\zeta}(x_t)\big] -\gamma_t \langle \nabla \mathbb{E}_{\zeta\sim \mathbb{P}}\big[\Psi_{\zeta}(x_t)\big],\hat{g}_t \rangle + \frac{K\gamma_t^2}{2} \big \| \hat{g}_t \big \|^2.
    \end{align}
    Since during parameters' update, the inexact mini-batch stochastic gradient, $\hat{g}^B_t$ is utilized, where its randomness comes from $x_t$, $\zeta$, $\big\{\xi\big\}_{\tilde{B}}$ and $\eta_{x}^{\tilde{d}}(\zeta)$. Taking expectation over $\zeta,\eta_{x}^{\tilde{d}}(\zeta), \big\{\xi \big\}_{B}$ on both sides conditioned on $x_t$, we have 
    \begin{align}
        &\mathbb{E}_{\zeta\sim \mathbb{P},\eta_{x_t}^{\tilde{d}}(\zeta),\xi_{B}\sim\nu} \Big[ \Psi_{\zeta}(x_{t+1})|x_t \Big]\nonumber\\
        \overset{(i)}{\leq}& \mathbb{E}_{\zeta\sim \mathbb{P}, \eta_{x_t}^{\tilde{d}}(\zeta),\xi_{B}\sim\nu} \Big[ \Psi_{\zeta}(x_{t})| x_t \Big] - \mathbb{E}_{\zeta\sim \mathbb{P},\eta_{x_t}^{\tilde{d}}(\zeta),\xi_{B}\sim \nu} \Big [  \langle \nabla \mathbb{E}_{\zeta\sim \mathbb{P}}\big[\Psi_{\zeta}(x_t)\big],\gamma_t \hat{g}^B_t \rangle | x_t \Big ] \nonumber\\
        &\qquad+ \mathbb{E}_{\zeta\sim \mathbb{P}, \eta_{x_t}^{\tilde{d}}(\zeta),\xi_{B}\sim \nu} \Big [ \frac{K\gamma_t^2}{2} \big \| \hat{g}^B_t \big \|^2 | x_t \Big] \nonumber\\
        \overset{(ii)}{\leq}& \mathbb{E}_{\zeta\sim \mathbb{P},\eta_{x_t}^{\tilde{d}}(\zeta),\xi_{B}\sim \nu } \Big[ \Psi_{\zeta}(x_{t})|x_t \Big] - \mathbb{E}_{\zeta\sim \mathbb{P},\eta_{x_t}^{\tilde{d}}(\zeta), \xi_{B}\sim\nu} \Big [ \left \langle \nabla \mathbb{E}_{\zeta\sim \mathbb{P}}\big[\Psi_{\zeta}(x_t)\big],\gamma_t \hat{g}^B_t\right \rangle | x_t \Big ]\nonumber \\ 
        &\qquad + \frac{K\gamma_t^2(R_1+10\varepsilon^2)}{2}\nonumber\\
        &\qquad +\frac{K\gamma_t^2}{2} \big \| \nabla \mathbb{E}_{\zeta\sim \mathbb{P},\eta_{x_t}^{\tilde{d}}(\zeta)}\big[\mathcal{L}_{\zeta}(x_t,\eta_{x_t}^{\tilde{d}}(\zeta))\big] \big \|^2 \nonumber\\
        \overset{(iii)}{\leq} & \mathbb{E}_{\zeta\sim \mathbb{P}, \eta_{x_t}^{\tilde{d}}(\zeta),\xi_{B}\sim \nu} \Big[ \Psi_{\zeta}(x_{t}) | x_t \Big] - \gamma_t\mathbb{E}_{\zeta\sim \mathbb{P}, \eta_{x_t}^{\tilde{d}}(\zeta), \xi_{B}\sim \nu} \Big [  \langle \nabla \mathbb{E}_{\zeta\sim \mathbb{P}}\big[\Psi_{\zeta}(x_t)\big], \hat{g}^B_t - g^B_t+g^B_t \rangle | x_t \Big ]\nonumber\\
         &\qquad+K\gamma_t^2
        \big \| 
        \nabla \mathbb{E}_{\zeta\sim \mathbb{P},\eta_{x_t}^{\tilde{d}}(\zeta) }\big[\mathcal{L}_{\zeta}(x_t,\eta_{x_t}^{\tilde{d}}(\zeta))\big]
        -\nabla \mathbb{E}_{\zeta\sim \mathbb{P},\eta_{x_t}^{\tilde{d}}(\zeta) }\big[ \Psi_{\zeta}(x_t)\big] \big \|^2 \nonumber\\ 
        &\qquad+
        K\gamma_t^2 \big\|\nonumber
        \nabla \mathbb{E}_{\zeta\sim \mathbb{P},\eta_{x_t}^{\tilde{d}}(\zeta)}\big[\Psi_{\zeta}(x_t)\big]\big \|^2\\
        &\qquad + \frac{KR_1\gamma_t^2}{2}+5K\varepsilon^2\gamma_t^2 \nonumber\\
        %\overset{}{=} & \mathbb{E}_{\zeta\sim \mathbb{P}, \eta_{x_t}^{\tilde{d}}(\zeta),\xi_{\tilde{B}}\sim \nu} \Big[ \Psi_{\zeta}(x_{t}) | x_t \Big] - \gamma_t\mathbb{E}_{\zeta\sim \mathbb{P}, \eta_{x_t}^{\tilde{d}}(\zeta),\xi_{\tilde{B}}\sim \nu} \Big [ \big \langle \nabla \mathbb{E}_{\zeta\sim \mathbb{P}}\big[\Psi_{\zeta}(x_t)\big], \hat{g}^B_t - g^B_t+g^B_t\big \rangle | x_t \Big ]\nonumber\\
        % &\qquad+K\gamma_t^2
        %\mathbb{E}_{\zeta\sim \mathbb{P},\eta_{x_t}^{\tilde{d}}(\zeta)}\Big[
        %\big \| \big[
        %\nabla\mathcal{L}_{\zeta}(x_t,\eta_{x_t}^{\tilde{d}}(\zeta))
        %-\nabla  \Psi_{\zeta}(x_t)\big] \big \|^2 | x_t\Big]\nonumber\\ 
        %&\qquad+
        %K\gamma_t^2\big\|\nonumber
        %\nabla \mathbb{E}_{\zeta\sim \mathbb{P}}\big[\Psi_{\zeta}(x_t)\big]\big \|^2 \\
        %&\qquad + \frac{KR_1\gamma_t^2}{2}+5KG^2\varepsilon^2\gamma_t^2 \nonumber\\
        \overset{(iv)}{\leq} & \mathbb{E}_{\zeta\sim \mathbb{P}, \eta_{x_t}^{\tilde{d}}(\zeta),\xi_{B}\sim \nu} \Big[ \Psi_{\zeta}(x_{t}) | x_t \Big] - \gamma_t\mathbb{E}_{\zeta\sim \mathbb{P}, \eta_{x_t}^{\tilde{d}}(\zeta), \xi_{B}\sim \nu} \Big [ \big \langle \nabla \mathbb{E}_{\zeta\sim \mathbb{P}}\big[\Psi_{\zeta}(x_t)\big], \hat{g}^B_t - g^B_t+g^B_t\big \rangle | x_t \Big ]\nonumber\\
         &\qquad+K\gamma_t^2
        \mathbb{E}_{\zeta\sim \mathbb{P}}\Big[
        \big \| 
        \nabla_1\mathbb{E}_{\eta_{x_t}^{\tilde{d}}(\zeta)}\big[\mathcal{L}_{\zeta}(x_t,\eta_{x_t}^{\tilde{d}}(\zeta))\big]
        -\nabla  \Psi_{\zeta}(x_t) \big \|^2 | x_t\Big]\nonumber\\ 
        &\qquad+
        K\gamma_t^2 \big\|\nonumber
        \nabla \mathbb{E}_{\zeta\sim \mathbb{P}}\big[\Psi_{\zeta}(x_t)\big]\big \|^2 + \frac{KR_1\gamma_t^2}{2}+5K\varepsilon^2\gamma_t^2 \nonumber\\
         \overset{(v)}{\leq} & \mathbb{E}_{\zeta\sim \mathbb{P}, \eta_{x_t}^{\tilde{d}}(\zeta),\xi_{B}\sim \nu} \Big[ \Psi_{\zeta}(x_{t})| x_t \Big] - \gamma_t\mathbb{E}_{\zeta\sim \mathbb{P}, \eta_{x_t}^{\tilde{d}}(\zeta),\xi_{B}\sim \nu} \Big [  \langle \nabla \mathbb{E}_{\zeta\sim \mathbb{P}}\big[\Psi_{\zeta}(x_t)\big], \hat{g}^B_t - g^B_t \rangle | x_t \Big ]\nonumber\\
        &\qquad -\gamma_t \mathbb{E}_{\zeta\sim \mathbb{P}, \eta_{x_t}^{\tilde{d}}(\zeta),\xi_{B}\sim \nu} \Big[\langle \nabla \mathbb{E}_{\zeta\sim \mathbb{P}}\big[\Psi_{\zeta}(x_t)\big],g^B_t \rangle | x_t \Big] \nonumber \\
        &\qquad + K\gamma_t^2 \big \| \nabla \mathbb{E}_{\zeta\sim\mathbb{P}}\big[\Psi_{\zeta}(x_t)\big]\big \|^2\nonumber \\
        &\qquad +K\varepsilon^2\gamma_t^2+ \frac{KR_1\gamma_t^2}{2}+5K\varepsilon^2\gamma_t^2
         \nonumber\\
         \overset{}{=}& \mathbb{E}_{\zeta\sim \mathbb{P}, \eta_{x_t}^{\tilde{d}}(\zeta),\xi_{B}\sim\nu} \Big[\Psi_{\zeta}(x_{t})| x_t \Big] - \gamma_t\mathbb{E}_{\zeta\sim \mathbb{P}} \Big [ \big \langle \nabla \mathbb{E}_{\zeta\sim \mathbb{P}}\big[\Psi_{\zeta}(x_t)\big], \mathbb{E}_{\eta_{x_t}^{\tilde{d}}(\zeta),\xi_{B}\sim \nu}[\hat{g}^B_t - g^B_t]\big \rangle| x_t \Big ]\nonumber\\
        &\qquad-(\gamma_t-K\gamma_t^2) \big \|\nabla \mathbb{E}_{\zeta \sim \mathbb{P}}\big[ \Psi_{\zeta}(x_t) \big] \big \|^2
        \nonumber\\
        &\qquad  + \frac{KR_1\gamma_t^2}{2}+6K\varepsilon^2\gamma_t^2\nonumber\\
        \overset{(vi)}{=}& \mathbb{E}_{\zeta\sim \mathbb{P}, \eta_{x_t}^{\tilde{d}}(\zeta), \xi_{B}\sim \nu} \Big[\Psi_{\zeta}(x_{t})| x_t \Big] - \gamma_t\mathbb{E}_{\zeta\sim \mathbb{P}} \Big [ \big \langle \nabla \mathbb{E}_{\zeta\sim \mathbb{P}}\big[\Psi_{\zeta}(x_t)\big], \nabla_1\mathbb{E}_{\eta_x^{\tilde{d}}(\zeta)}[\mathcal{L}_{\zeta}(x_t,\eta_{x}^{\tilde{d}}(\zeta))-\nabla \Psi_{\zeta}(x_t)]\big \rangle| x_t \Big ]\nonumber\\
        &\qquad-(\gamma_t-K\gamma_t^2) \big \|\nabla \mathbb{E}_{\zeta \sim \mathbb{P}}\big[ \Psi_{\zeta}(x_t) \big] \big \|^2
        \nonumber\\
        &\qquad  + \frac{KR_1\gamma_t^2}{2}+6K\varepsilon^2\gamma_t^2\nonumber\\
         \overset{(vii)}{\leq}& \mathbb{E}_{\zeta\sim \mathbb{P}, \eta_{x_t}^{\tilde{d}}(\zeta),\xi_{B}\sim \nu} \Big[ \Psi_{\zeta}(x_{t})| x_t \Big] + \gamma_t\mathbb{E}_{\zeta\sim \mathbb{P}}\Big[\big \|\nabla \mathbb{E}_{\zeta\sim \mathbb{P}}\big[\Psi_{\zeta}(x_t)\big] \big \|\cdot \big \|\nabla_1\mathbb{E}_{ \eta_{x_t}^{\tilde{d}}(\zeta)}[\mathcal{L}_{\zeta}(x_t,\eta_{x_t}^{\tilde{d}}(\zeta))]-\nabla \Psi_{\zeta}(x_t)\big] \big \||x_t\Big]\nonumber\\
        &\qquad-(\gamma_t-K\gamma_t^2) \big \|\nabla \mathbb{E}_{\zeta\sim \mathbb{P}}\big[\Psi_{\zeta}(x_t)\big] \big \|^2  \nonumber \\
        & \qquad  + \frac{KR_1\gamma_t^2}{2} +6K\varepsilon^2\gamma_t^2 \nonumber\\
        =& \mathbb{E}_{\zeta\sim \mathbb{P}, \eta_{x_t}^{\tilde{d}}(\zeta),\xi_{B}\sim \nu} \Big[ \Psi_{\zeta}(x_{t})| x_t \Big] + \gamma_t\big \|\nabla \mathbb{E}_{\zeta\sim \mathbb{P}}\big[\Psi_{\zeta}(x_t)\big] \big \|\cdot \mathbb{E}_{\zeta\sim \mathbb{P}}\big \|\nabla_1\mathbb{E}_{\eta_{x_t}^{\tilde{d}}(\zeta)}[\mathcal{L}_{\zeta}(x_t,\eta_{x_t}^{\tilde{d}}(\zeta))]-\nabla \Psi_{\zeta}(x_t)\big] \big \|\nonumber\\
        &\qquad-(\gamma_t-K\gamma_t^2) \big \|\nabla \mathbb{E}_{\zeta\sim \mathbb{P}}\big[\Psi_{\zeta}(x_t)\big] \big \|^2  \nonumber \\
        & \qquad  + \frac{KR_1\gamma_t^2}{2} +6K\varepsilon^2\gamma_t^2 \nonumber\\
        \overset{(viii)}{\leq}& \mathbb{E}_{\zeta\sim \mathbb{P}(\zeta),\xi_{B}\sim\nu,\eta_{x_t}^{\tilde{d}}} \Big[ \Psi_{\zeta}(x_{t}) | x_t \Big] + \gamma_t \varepsilon\big \|\nabla \mathbb{E}_{\zeta\sim \mathbb{P}}\big[\Psi_{\zeta}(x_t)\big] \big\| -(\gamma_t-K\gamma_t^2) \big \|\nabla \mathbb{E}_{\zeta\sim \mathbb{P}}\big[\Psi_{\zeta}(x_t)\big] \big \|^2 \nonumber \\
        & \qquad + \frac{KR_1\gamma_t^2}{2} +6K\varepsilon^2\gamma_t^2,
        %\overset{}{=}&\mathbb{E}_{\zeta\sim \mathbb{P}} \Big[ \Psi_{\zeta}(x_{t}) | x_t \Big] + \gamma_t G\varepsilon\mathbb{E}_{ \eta_{x}^{\tilde{d}}(\zeta)}\Big[\big \|\nabla \mathbb{E}_{\zeta\sim \mathbb{P}}\big[\Psi_{\zeta}(x_t)\big] \big \|| x_t\Big] \nonumber\\
        %&\qquad-(\gamma_t-K\gamma_t^2)\mathbb{E}_{\eta_{x}^{\tilde{d}}(\zeta)}\Big[ \big \|\nabla \mathbb{E}_{\zeta\sim \mathbb{P}}\big[\Psi_{\zeta}(x_t)\big] \big \|^2 | x_t\Big] +\frac{KR_1\gamma_t^2}{2} +5KG^2\varepsilon^2\gamma_t^2, \nonumber
    \end{align}
        where (i) applies descent lemma \eqref{descentlemma}; (ii) applies \eqref{bounded var1-1} stated in lemma \ref{thm: boundedmoment3} with $B=1$; (iii) applies $(a+b)^2\leq 2a^2+2b^2$ to further upper bound $\|\nabla \mathbb{E}_{\zeta\sim\mathbb{P},\eta_{x_t}^{\tilde{d}}(\zeta)}[\mathcal{L}_{\zeta}(x_t, \eta_{x_t}^{\tilde{d}}(\zeta))] \|^2$ by \\
        $\allowdisplaybreaks 2\|\nabla_1 \mathbb{E}_{\zeta\sim\mathbb{P},\eta_{x_t}^{\tilde{d}}(\zeta)}[\mathcal{L}_{\zeta}(x_t, \eta_{x_t}^{\tilde{d}}(\zeta))] - \nabla \mathbb{E}_{\zeta\sim\mathbb{P}}[\Psi_{\zeta}(x_t)] \|^2+2\|\nabla \mathbb{E}_{\zeta\sim\mathbb{P}}[\Psi_{\zeta}(x_t)] \|^2$ (the expectation over $\eta_{x_t}^{\tilde{d}}(\zeta)$ can be neglected as $\Psi_{\zeta}(x_t)$ doesn't contain randomness from $\eta_{x_t}^{\tilde{d}}(\zeta)$); (iv) applies Jensen's inequality to extract $\mathbb{E}_{\zeta\sim\mathbb{P},\eta_{x_t}^{\tilde{d}}(\zeta)}$ out from squared norm; (v) applies condition~\eqref{optimalrelation} stated in Theorem \ref{thm: grad_error_bound}; (vi) moves expectation over $\xi_{B}$ inside inner product;
        (vii) applies Cauchy-Schwarz inequality; (viii) again applies condition~\eqref{optimalrelation} stated Theorem \ref{thm: grad_error_bound} as such relationship holds for every $\zeta$.\\
        Then, for any $\gamma_t < \frac{1}{2K}$ by choice, we have $(\gamma_t - K\gamma_t^2)> \frac{\gamma_t}{2}$. Taking expectation over $x_t$, the above inequality further transformed to
    \begin{align}
    &\mathbb{E}_{x_t,\zeta\sim \mathbb{P},\eta_{x}^{\tilde{d}}(\zeta),\xi_{B}\sim \nu}\Big[\Psi_{\zeta}(x_{t+1}) \Big] \nonumber \\
    \leq & \mathbb{E}_{x_t,\zeta\sim \mathbb{P},\eta_{x}^{\tilde{d}}(\zeta),\xi_{B}\sim \nu} \Big[ \Psi_{\zeta}(x_{t}) \Big] + \gamma_t \varepsilon \mathbb{E}_{x_t}\big \|\nabla \mathbb{E}_{\zeta\sim \mathbb{P}}\big[\Psi_{\zeta}(x_t)\big] \big \|  -\frac{\gamma_t}{2}\mathbb{E}_{x_t}\big \|\nabla \mathbb{E}_{\zeta\sim \mathbb{P}}\big[\Psi_{\zeta}(x_t)\big] \big \|^2\nonumber\\
    &\qquad + \frac{KR_{1}\gamma_t^2}{2} +6K\varepsilon^2\gamma_t^2. \nonumber
    \end{align}
     %Notice that $\mathbb{E}_{x_t,\zeta \sim \mathbb{P},\xi_{\tilde{B}}\sim \nu}[ \Psi_{\zeta}(x_{t}) ]$ is a constant by tower property. The same argument applies for $\mathbb{E}_{x_t, \zeta\sim \mathbb{P},\xi_{\tilde{B}}\sim \nu}\big[\big \|\nabla \mathbb{E}_{\zeta\sim \mathbb{P}}\big[\Psi_{\zeta}(x_t)\big] \big \|^2| x_t\big]$.
     Re-arranging above terms, we have
    \begin{align}
        & \frac{\gamma_t}{2}\mathbb{E}_{x_t}\Big(\big \| \nabla \mathbb{E}_{\zeta\sim \mathbb{P}}\big[\Psi_{\zeta}(x_t)\big] \big \|^2 -2\varepsilon\big \| \nabla \mathbb{E}_{\zeta\sim \mathbb{P}}\big[\Psi_{\zeta}(x_t)\big] \big \|\Big) \nonumber\\ 
        =& \frac{\gamma_t}{2}\mathbb{E}_{x_t}\Big(\big \| \nabla \mathbb{E}_{\zeta\sim\mathbb{P}}\big[\Psi_{\zeta}(x_t)\big] \big \|-\varepsilon\Big)^2 - \frac{\gamma_t\varepsilon^2}{2}\nonumber\\
        \leq& \mathbb{E}_{x_t,\zeta\sim \mathbb{P},\eta_{x_t}^{\tilde{d}}(\zeta),\xi_{B}\sim\nu}\big[\Psi_{\zeta}(x_t) - \Psi_{\zeta}(x_{t+1})\big]  + \frac{KR_{1}\gamma_t^2}{2} + 6K\varepsilon^2\gamma_t^2. \nonumber
    \end{align}
    Re-arranging above inequality, we have
    \begin{align}
        \frac{\gamma_t}{2}\mathbb{E}_{x_t}\Big(\big \| \nabla \mathbb{E}_{\zeta\sim\mathbb{P}}\big[\Psi_{\zeta}(x_t)\big] \big \|-\varepsilon\Big)^2
        \leq& \mathbb{E}_{x_t,\zeta\sim \mathbb{P},\eta_{x_t}^{\tilde{d}}(\zeta),\xi_{B}\sim\nu}\big[\Psi_{\zeta}(x_t) - \Psi_{\zeta}(x_{t+1})\big] \nonumber\\
        &\qquad+ \frac{KR_{1}\gamma_t^2}{2} + 6K\varepsilon^2\gamma_t^2+\frac{\gamma_t\varepsilon^2}{2}. \nonumber
    \end{align}
     Applying $\frac{(a+b)^2}{2}\leq a^2+b^2$ and re-arranging above terms, we have
        \begin{align}
        &\frac{\gamma_t}{4} \mathbb{E}_{x_t}\big \| \nabla \mathbb{E}_{\zeta\sim \mathbb{P}}\big[\Psi_{\zeta}(x_t)\big] \big \|^2\nonumber\\
        \leq& \frac{\gamma_t}{2}\mathbb{E}_{x_t}\Big[\Big(\big \| \nabla \mathbb{E}_{\zeta\sim \mathbb{P}}\big[\Psi_{\zeta}(x_t)\big] \big \|-\varepsilon\Big)^2\Big] +\frac{\gamma_t \varepsilon^2}{2}\nonumber\\
        \leq& \mathbb{E}_{x_t,\zeta\sim \mathbb{P},\eta_{x_t}^{\tilde{d}}(\zeta),\xi_{B}\sim\nu}\big[\Psi_{\zeta}(x_t) - \Psi_{\zeta}(x_{t+1})\big] + \frac{K R_{1}\gamma_t^2}{2} + 6K \varepsilon^2\gamma_t^2 + \varepsilon^2\gamma_t. \nonumber
    \end{align}
    Summing the above inequality from $t=0$ to $T-1$ leads to
    \begin{align}
        &\sum_{t=0}^{T-1}\frac{\gamma_t}{4}\mathbb{E}_{x_t}\big \| \nabla \mathbb{E}_{\zeta\sim \mathbb{P}}\big[\Psi_{\zeta}(x_t)\big] \big \|^2\nonumber\\
        \leq& \sum_{t=0}^{T-1}\mathbb{E}_{x_t, \zeta\sim \mathbb{P}, \eta_{x_t}^{\tilde{d}}(\zeta),\xi_{B}\sim\nu}\big[\Psi_{\zeta}(x_t)- \Psi(x_{t+1})  \big] + 6K \sum_{t=0}^{T-1}\gamma_t^2\varepsilon^2  + \frac{KR_1}{2}\sum_{t=0}^{T-1} \gamma_t^2 + \varepsilon^2\sum_{t=0}^{T-1}\gamma_t\nonumber\\
        \leq& \mathbb{E}_{\zeta\sim \mathbb{P}}\big[\Psi_{\zeta}(x_0)-\Psi_{\zeta}(x^*)\big] +6K\varepsilon^2 \sum_{t=0}^{T-1}\gamma_t^2  + \frac{KR_1}{2}\sum_{t=0}^{T-1} \gamma_t^2 + \varepsilon^2\sum_{t=0}^{T-1}\gamma_t. \nonumber
    \end{align}
    Denoting $\Delta = \mathbb{E}_{\zeta\sim \mathbb{P}}\big[\Psi_{\zeta}(x_0)-\Psi_{\zeta}(x^*)\big]$ and choosing constant learning rate $\gamma_t=\gamma$, we have
    \begin{align}
        \mathbb{E}_{x_{\tilde{t}}}\big[ \| \nabla \mathbb{E}_{\zeta \sim \mathbb{P}}\big[\Psi_{\zeta}(x_{\tilde{t}}) \big] \big\|^2\big] =& \frac{1}{T}\sum_{t=0}^{T-1}\mathbb{E}_{x_t} \|\nabla \mathbb{E}_{\zeta\sim\mathbb{P}}[\Psi(x_t)] \|^2 \nonumber \\
        \leq& \frac{4\Delta}{T\gamma}+\frac{24K \varepsilon^2\gamma^2 T}{T\gamma} + \frac{2KR_{1}\gamma^2T}{T\gamma}+\frac{4  \varepsilon^2 T\gamma}{T\gamma }.
        \label{nestsgdconv1}
    \end{align}
    Then, choosing $\gamma= \min\{\frac{1}{24K}, \frac{\varepsilon^2}{2KR_1}\}$, we immediately have
    \begin{align}
        \frac{24K\varepsilon^2\gamma^2T}{T\gamma}\leq \varepsilon^2 \text{ and } \frac{2KR_1\gamma^2T}{T\gamma}\leq \varepsilon^2.
    \end{align}
    To make $\frac{4\Delta}{T\gamma}= \varepsilon^2$, we have
    \begin{align}
        T= \frac{4\Delta}{\gamma\varepsilon^2}\geq \max\{\frac{96\Delta K}{\varepsilon^2},\frac{8\Delta KR_1}{\varepsilon^{4}}\} =\mathcal{O}(\Delta K R_1\varepsilon^{-4}).
    \end{align}
    Combining all inequalities together, we have
    \begin{align}
         &\mathbb{E}_{x_{\tilde{t}}}\big[ \| \nabla \mathbb{E}_{\zeta \sim \mathbb{P}}\big[\Psi_{\zeta}(x_{\tilde{t}}) \big] \big\|^2\big] \leq 7\varepsilon^2,
    \end{align}
    which gives the desired result.
\end{proof}
\subsection{Proof of Descent Lemma \texorpdfstring{\eqref{descentlemma}}{(descentlemma)}}\label{pfdescentlemma}
\begin{lemma}
Denote $\eta_{x_t}^*(\zeta)\in \arg\min_{\eta} \mathcal{L}_{\zeta}(x_t,\eta)$.
Then, for $\mathbb{E}_{\zeta\sim\mathbb{P}}[\Psi(x_t)]$, we have the following descent lemma
\begin{align}
    \mathbb{E}_{\zeta\sim \mathbb{P}}\big[\Psi_{\zeta}(x_{t+1})\big] \leq \mathbb{E}_{\zeta\sim \mathbb{P}}\big[\Psi_{\zeta}(x_t)\big] + \langle \nabla \mathbb{E}_{\zeta\sim\mathbb{P}}\big[\Psi_{\zeta}(x_t)\big], x_{t+1}-x_t \rangle + \frac{K}{2} \left \| x_{t+1} - x_t \right \|^2, \label{eq: directional smooth descent lemma}
\end{align}
where $K = G^2 (\lambda \red{\beta})^{-1}M+L$.
\end{lemma}
\begin{proof}
    Notice that, applying Jensen's inequality and taking square-root on both sides, Lemma \ref{thm:directionsmooth} implies
    \begin{align}
        \big\| \nabla\mathbb{E}_{\zeta\sim \mathbb{P}}[ \Psi_{\zeta}(x)]- \nabla_1\mathbb{E}_{\zeta\sim \mathbb{P}}[ \mathcal{L}_{\zeta}(x',\eta_{x}^*(\zeta))]   \big \| \leq K\big \| x-x' \big \|. \nonumber
    \end{align}
    From fundamental theorem of calculus, we have
    \begin{align}
        &\Big |\mathbb{E}_{\zeta\sim \mathbb{P}}\big[\mathcal{L}_{\zeta}(x_{t+1},\eta_{x_t}^*(\zeta))\big] - \mathbb{E}_{\zeta\sim \mathbb{P}}\big[\Psi_{\zeta}(x_t)\big] - \langle \nabla \mathbb{E}_{\zeta\sim \mathbb{P}}\big[\Psi_{\zeta}(x_t)\big],x_{t+1} -x_t \rangle \ \Big | \nonumber\\
        =& \Big | \int_{0}^{1} \langle \nabla_1 \mathbb{E}_{\zeta\sim\mathbb{P}}\big[ \mathcal{L}_{\zeta}(x_t+ t(x_{t+1}-x_t),\eta_{x_t}^*(\zeta))\big], x_{t+1} - x_t \rangle - \langle \nabla \mathbb{E}_{\zeta\sim \mathbb{P}}\big[\Psi_{\zeta}(x_t)\big],x_{t+1}-x_t \rangle \mathrm{d}t\Big| \nonumber\\
        =& \Big | \int_{0}^{1} \langle \nabla_1 \mathbb{E}_{\zeta\sim \mathbb{P}} \big[\mathcal{L}_{\zeta}(x_t+ t(x_{t+1}-x_t),\eta_{x_t}^*(\zeta))\big] - \nabla_1 \mathbb{E}_{\zeta\sim \mathbb{P}}\big[\mathcal{L}_{\zeta}(x_t,\eta_{x_t}^*(\zeta))\big] , x_{t+1} - x_t \rangle \mathrm{d}t \Big |\nonumber\\
        \overset{(i)}{\leq} & \Big | \int_{0}^{1} \big \|\nabla_1 \mathbb{E}_{\zeta}\big[\mathcal{L}_{\zeta}(x_t+ t(x_{t+1}-x_t),\eta_{x_t}^*(\zeta))\big] - \nabla_1 \mathbb{E}_{\zeta\sim \mathbb{P}}\big[ \mathcal{L}_{\zeta}(x_t,\eta_x^*(\zeta))\big] \big \| \left \| x_{t+1} - x_t \right \| \mathrm{d}t\Big| \nonumber \\
        \overset{(ii)}{\leq} & \Big | \int_{0}^{1} t K\left \| x_{t+1} - x_t \right \|^2 \mathrm{d}t \Big| \nonumber \\
        \overset{}{\leq} & \frac{K}{2} \big \| x_{t+1} - x_t \big \|^2, \nonumber
    \end{align}
    where (i) applies Cauchy-Schwarz inequality; (ii) applies directional smoothness property stated at Lemma \ref{thm:directionsmooth}.\\
    Re-arranging above inequality, we have
    \begin{align}
        \mathbb{E}_{\zeta\sim \mathbb{P}}\big[\mathcal{L}_{\zeta}(x_{t+1},\eta_{x_t}^*(\zeta))\big] \leq \mathbb{E}_{\zeta\sim \mathbb{P}}\big[\Psi_{\zeta}(x_t)\big] + \langle \nabla \mathbb{E}_{\zeta\sim \mathbb{P}}\big[\Psi_{\zeta}(x_t)\big], x_{t+1}-x_t \rangle + \frac{K}{2} \left \| x_{t+1} - x_t \right \|^2.\nonumber
    \end{align} 
    Since for each $\zeta$, $\eta_{x_{t+1}}^*(\zeta) \in\arg\min_{\eta} \mathcal{L}_{\zeta}(x_{t+1},\eta)$, it holds that $\Psi_{\zeta}(x_{t+1}) = \mathcal{L}_{\zeta}(x_{t+1},\eta_{x_{t+1}}^*(\zeta)) \leq \mathcal{L}_{\zeta}(x_{t+1},\eta_{x_t}^*(\zeta)) $. Taking expectation over $\zeta$, we have $\mathbb{E}_{\zeta\sim \mathbb{P}}\big[\Psi_{\zeta}(x_{t+1})\big] \leq \mathbb{E}_{\zeta\sim \mathbb{P}}\big[\mathcal{L}_{\zeta}(x_{t+1},\eta_{x_t}^*(\zeta))\big] $. Combining this fact with above inequality gives us the desired result,
    \begin{align}
         \mathbb{E}_{\zeta\sim \mathbb{P}}\big[\Psi_{\zeta}(x_{t+1})\big] \leq \mathbb{E}_{\zeta\sim \mathbb{P}}\big[\Psi_{\zeta}(x_t)\big] + \langle \nabla \mathbb{E}_{\zeta\sim\mathbb{P}}\big[\Psi_{\zeta}(x_t)\big], x_{t+1}-x_t \rangle + \frac{K}{2} \left \| x_{t+1} - x_t \right \|^2. \nonumber
    \end{align}
\end{proof}
\section{Proof of Corollary \ref{thm: total complexity bound}}\label{Appendix: discussion of complexity}
\totalcomplexity*
\begin{proof}
    According to Theorem~\ref{thm: innersgd}, for Algorithm~\ref{alg1} to output $\eta_{x}^{\tilde{d}}(\zeta)$ satisfying the condition \eqref{eq: inner near-optima} in Theorem~\ref{thm: grad_error_bound}, a sample complexity of $\mathcal{O}(\hat{\Delta} G^4 K' R_2 \varepsilon^{-4})$ is required.
    Furthermore, by Theorem~\ref{thm: convergenestsgd}, Algorithm~\ref{alg2} requires $\mathcal{O}(1)$ mini-batch of $\zeta$ and $\xi$ per iteration and runs for $\mathcal{O}(\Delta K R_1 \varepsilon^{-4})$ iterations.
    Combining both results, the total sample complexity is $\mathcal{O}(\Delta \hat{\Delta} R_1 R_2 K K'G^4 \varepsilon^{-8})\sim \mathcal{O}(\varepsilon^{-8})$.
    
    For the total iteration complexity, i.e., $T\times D$, it can be improved as follows. By setting $B = \Theta(\varepsilon^{-2})$ and applying the conclusion from Lemma~\ref{thm: boundedmoment3} along with the descent inequality, we have
    \begin{align}
        &\mathbb{E}_{\zeta\sim \mathbb{P},\eta_{x_t}^{\tilde{d}}(\zeta),\xi_{B}\sim\nu} \Big[ \Psi_{\zeta}(x_{t+1})|x_t \Big]\nonumber\\
        \leq& \mathbb{E}_{\zeta\sim \mathbb{P},\eta_{x_t}^{\tilde{d}}(\zeta),\xi_{B}\sim \nu } \Big[ \Psi_{\zeta}(x_{t})|x_t \Big] - \mathbb{E}_{\zeta\sim \mathbb{P},\eta_{x_t}^{\tilde{d}}(\zeta), \xi_{B}\sim\nu} \Big [ \left \langle \nabla \mathbb{E}_{\zeta\sim \mathbb{P}}\big[\Psi_{\zeta}(x_t)\big],\gamma_t \hat{g}^B_t\right \rangle | x_t \Big ]\nonumber \\ 
        &\qquad + \frac{K\gamma_t^2(R_1\varepsilon^2+10\varepsilon^4)}{2}\nonumber\\
        &\qquad +\frac{K\gamma_t^2}{2} \big \| \nabla \mathbb{E}_{\zeta\sim \mathbb{P},\eta_{x_t}^{\tilde{d}}(\zeta)}\big[\mathcal{L}_{\zeta}(x_t,\eta_{x_t}^{\tilde{d}}(\zeta))\big] \big \|^2 \nonumber.
    \end{align}
    Following the same logic as Proof \ref{proof: nested SGD convergence thm}, we then have
    \begin{align}
        &\mathbb{E}_{x_t,\zeta\sim \mathbb{P},\eta_{x}^{\tilde{d}}(\zeta),\xi_{B}\sim \nu}\Big[\Psi_{\zeta}(x_{t+1}) \Big] \nonumber \\
    \leq & \mathbb{E}_{x_t,\zeta\sim \mathbb{P},\eta_{x}^{\tilde{d}}(\zeta),\xi_{B}\sim \nu} \Big[ \Psi_{\zeta}(x_{t}) \Big] + \gamma_t \varepsilon \mathbb{E}_{x_t}\big \|\nabla \mathbb{E}_{\zeta\sim \mathbb{P}}\big[\Psi_{\zeta}(x_t)\big] \big \|  -\frac{\gamma_t}{2}\mathbb{E}_{x_t}\big \|\nabla \mathbb{E}_{\zeta\sim \mathbb{P}}\big[\Psi_{\zeta}(x_t)\big] \big \|^2\nonumber\\
    &\qquad + \frac{KR_{1}\gamma_t^2\varepsilon^2}{2} +K\varepsilon^2\gamma_t^2+5K\gamma_t^2\varepsilon^{4}. \nonumber
    \end{align}
    Let $\gamma_t = \gamma$, re-arranging above inequality, summing $t$ from $0,\cdots,T-1$ and dividing by $T$, we have
    \begin{align}
        \frac{\gamma}{4T} \sum_{t=0}^{T-1}\mathbb{E}_{x_t}\big \| \nabla \mathbb{E}_{\zeta\sim \mathbb{P}}\big[\Psi_{\zeta}(x_t)\big] \big \|^2
        \leq \frac{\Delta}{T} + \frac{K R_{1}\gamma^2\varepsilon^2}{2} + {K \varepsilon^2\gamma^2} + {5K \varepsilon^4\gamma^2}+{\varepsilon^2\gamma}. \nonumber
    \end{align}
    For $\gamma=\min\{\frac{1}{2KR_1}, \frac{1}{4K} \}$, after $T=\max \{8\Delta KR_1, 16\Delta K \}\varepsilon^{-2}=\mathcal{O}(\Delta K R_1\varepsilon^{-2})$ iterations,
    above inequality further implies
    \begin{align}
        \mathbb{E}_{x_{\tilde{t}}}\big \| \nabla \mathbb{E}_{\zeta\sim \mathbb{P}}\big[\Psi_{\zeta}(x_{\tilde{t}})\big] \big \|^2\leq 7\varepsilon^2+5 \varepsilon^4=\mathcal{O}(\varepsilon^2). 
    \end{align}
    This concludes that by choosing $B,\tilde{B}\sim \Theta(\varepsilon^{-2})$ as suggested in Remark \ref{remark: large batch size for inner problem}, we have total iteration complexity $T\times D=\mathcal{O}(\Delta \hat{\Delta}R_1R_2K K' G^4\varepsilon^{-4})\sim \mathcal{O}(\varepsilon^{-4})$. Based on gradient dimension with respect to $x$ and $\eta$, we conclude the per-iteration complexity of algorithm \ref{alg1} and \ref{alg2} are $\mathcal{O}(1),\mathcal{O}(d)$ respectively. 
\end{proof}
%%%%%%%%%%%%%%%%%%%%%%%%%%%%%%%%%%%%%%%%%%%%%%%%%%%%%%%%%%%%%%%%%%%%%%%%%%%%%%%
%%%%%%%%%%%%%%%%%%%%%%%%%%%%%%%%%%%%%%%%%%%%%%%%%%%%%%%%%%%%%%%%%%%%%%%%%%%%%%%

\end{document}